\documentclass[a4paper, a4wide, 10pt, reqno]{amsart}
\usepackage{a4wide}
\usepackage{accents}
\usepackage{algorithm}
\usepackage{algpseudocode}
\usepackage{amsfonts}

\usepackage{pgfplots}
\usepackage{amsthm, mathrsfs}
\usepackage{amsmath}
\usepackage{amssymb}
\usepackage[english]{babel}
\usepackage{bm}
\usepackage{bbm}
\usepackage{caption}
\usepackage{changepage}
\usepackage{color}
\usepackage{enumitem}
\usepackage{esvect}
\usepackage[T1]{fontenc}

\usepackage{todonotes}

\usepackage[left=1in,right=1in,top=0.9in,bottom=1in]{geometry}
\usepackage{graphicx}
\usepackage{hyperref}
\usepackage[utf8]{inputenc}
\usepackage{mathtools}
\usepackage{nccmath}
\usepackage{relsize}

\usepackage{subcaption}
\usepackage{verbatim}

\theoremstyle{plain}
\newtheorem{theorem}{Theorem}[section]
\newtheorem{corollary}[theorem]{Corollary}

\newtheorem{proposition}[theorem]{Proposition}

\newtheorem{lemma}[theorem]{Lemma}

\newtheorem{conjecture}[theorem]{Conjecture}
\newtheorem{definition}[theorem]{Definition}
\theoremstyle{definition}
\newtheorem{remark}[theorem]{Remark}
\numberwithin{equation}{section}

\newcommand{\R}{{\mathbb R}}
\newcommand{\Z}{{\mathbb Z}}
\newcommand{\N}{{\mathbb N}}
\newcommand{\E}{\mathbb E}

\newcommand{\Prob}{\mathbb{P}}
\newcommand\eps{\varepsilon}

\newcommand{\sC}{{\mathscr C}}
\newcommand{\sG}{{\mathscr G}}

\newcommand{\sL}{{\mathscr L}}
\newcommand{\sT}{{\mathscr T}}
\newcommand{\sU}{{\mathscr U}}
\newcommand{\re}{{\mathrm e}}
\newcommand{\Co}{{\mathrm{Co}}}
\newcommand{\Si}{{\mathrm{Si}}}

\newcommand{\ind}[1]{\mathbbm{1}_{\{#1\}}}
\newcommand{\Ind}[1]{\mathbbm{1}{\{#1\}}}
\newcommand{\rd}{\mathrm{d}}

\renewcommand{\complement}{\mathsf{c}}

\pgfplotsset{compat=1.18}
\begin{document}
\title{The critical percolation window in growing random graphs}
\author[J.\ Jorritsma, P.\ Maillard, P.\ M\"orters]{Joost Jorritsma$^1$, Pascal Maillard$^2$, Peter M\"orters$^3$}
\address{\hspace{-12pt}$^1$Department of Statistics, University of Oxford, United Kingdom. 
\newline 
$^2$Institut de Mathématique de Toulouse, 
Université de Toulouse, France, and 
\newline 
\phantom{$^2$}Institut Universitaire de France.
\newline
$^3$Department Mathematik/Informatik, Universit\"at zu K\"oln, Germany.}
\email{joost.jorritsma@stats.ox.ac.uk \\ Pascal.Maillard@math.univ-toulouse.fr \\ \newline
\hspace*{72pt}moerters@mathematik.uni-koeln.de}

\begin{abstract}
We describe the critical window for percolation in the universality class of sparse growing random graphs. In 
our models, vertices arrive sequentially and connect independently to each earlier vertex~$v$ with probability proportional to a nonpositive
power of the arrival time of $v$, continuing until the graph has $n$ vertices. This class includes uniformly grown random graphs and inhomogeneous random graphs of preferential-attachment type. 
Whenever the critical percolation threshold is positive, we show that the critical 
window has width of order $(\log n)^{-2}$ 
and a secondary phase transition at its finite upper boundary.
Inside this window the largest component has size of order $\sqrt{n}/\log n$, and the susceptibility remains finite and independent of the  position in the window. The proofs couple component explorations to branching random walks killed outside an interval of length $\log n$, allowing sharp control of the barely subcritical and critical regimes.
\end{abstract}

\maketitle

\vspace{-2em}
\section{Introduction}
\noindent\textbf{Mean-field random graphs.\ }Phase transitions in random graph models describe the sudden emergence of large-scale connectivity as the edge density increases.  
The classical example is the percolation transition discovered by Erd\H{o}s and R\'enyi: In the random graph $G(n,c/n)$, a giant component---a connected component containing a positive fraction of the $n$ vertices---appears as the edge-density parameter $c$ passes the critical value~one.  
Below this point all components are small.  
A series of works by Bollob\'as, \L{}uczak, and Aldous~\cite{Aldous, B84, ER60, JKLP, L90} established a complete asymptotic description of this first-order phase transition.  
They showed that, as $c=1+\varepsilon_n\to1$,
the size of the largest component $\sL_n$ exhibits a characteristic `double jump' 
and has a critical window of width~$n^{-1/3}$:
\begin{align}
    |\sL_n|&\sim  2|\varepsilon_n|^{-2}\log(|\varepsilon_n|^3n), &\text{if }\varepsilon_nn^{1/3}&\to-\infty, \nonumber\\
    |\sL_n|&\asymp n^{2/3},&\text{if }\varepsilon_n n^{1/3}&\to\alpha\in\R,\label{eq:critical-er}\\
    |\sL_n|&\sim 2\varepsilon_n n,&\text{if }\varepsilon_nn^{1/3} &\to \infty.\nonumber
\end{align}
Here, we write $X_n\sim a_n$ if $X_n/a_n\to1$ in probability, and $X_n\asymp a_n$ if $X_n/a_n$ is tight on $(0,\infty)$. 
\smallskip

Critical Erd\H{o}s--R\'enyi graphs belong to the same universality class as high-dimensional critical nearest-neighbour bond percolation on~$\Z^d$, the configuration model with degree sequence of bounded third moment, and many other models. This class is often called the \emph{mean-field universality class}.  
For models in this class, the local neighbourhood of a typical vertex is well approximated by a critical Bienaym\'e--Galton--Watson branching process, and the probability that the component~$\sC$ of a typical vertex has size at least~$k$ decays as $1/\sqrt{k}$.  
Consequently, $|\sC|$ has an infinite first moment in the limit, and the \emph{susceptibility}---the mean component size averaged over all vertices---diverges at criticality.  
We refer the reader to~\cite{Bollobas, Heydenreich, JLR} for background, references, and related results. 

 \medskip
 \noindent\textbf{Growing random graphs.\ }
 We leave the mean-field paradigm, and study a family of sparse growing random graph models in which vertices are added sequentially and connect to earlier vertices.  Growing graph models have been of interest in the statistical physics literature for about two decades for their `anomalous percolation behaviour', see for example~\cite{bollobas2004phase, Doro01, KrapDer04, LiLiu2021, Li2016}.  Other than mean-field models,  growing models have finite susceptibility at criticality and exhibit an infinite order phase transition. More specifically, the proportion $\theta(p)$ of vertices in the giant component 
 satisfies, when the percolation probability $p$ approaches $p_c$ from above, for some model-dependent constant $C>0$,
\begin{equation}\label{eq:theta-intro}
\theta(p)=
\exp\bigg(-\frac{C+o(1)}{\sqrt{p-p_c}}\bigg).
\end{equation}
This was proved for a range of growing graph models,
see~\cite{BJR05, BR05, EMO, RIORDAN_2005} for rigorous results.
To our knowledge, no results are known  about the critical percolation window in these models. 
\medskip

\noindent\textbf{Our results.\ }
In this paper we investigate the behaviour of this class of models in a window around the critical probability for an accessible model including models of preferential-attachment type, see Definition \ref{def:pa-ann}. We show that our models feature,
\emph{\begin{enumerate}[leftmargin=*]
\item[(1)] a critical window of width $(\log n)^{-2}$, with largest component of order $\sqrt{n}/\log n$ (Theorem \ref{thm:largest});
\item[(2)] a secondary phase transition at the finite upper bound of the critical window (Corollary \ref{cor:critical});
\item[(3)] finite and constant susceptibility throughout the entire critical window (Theorem \ref{thm:susceptibility}).
\end{enumerate}}

\newgeometry{margin=1in}
We believe that this is compelling evidence for the existence of a \emph{universality class} containing many models of growing random graphs, both with exponentially-decaying or polynomially-decaying degree distribution (the latter are often called scale-free graphs). The class  is characterised by these three features.
Finite and constant susceptibility is a particularly striking aspect of this universality class, and we are not aware of earlier rigorous proofs of this phenomenon for scale-free graphs.  
In most statistical-physics and random-graph models, in particular mean-field models, the Ising model and percolation on~$\Z^d$, the susceptibility diverges near criticality.
In growing random graphs, throughout the window up to its finite upper bound, vertices in components of order $\sqrt{n}/\log n$ account for only $o(n)$ of the total summed component size, making their contribution to the mean negligible. Therefore,
even though the number of edges changes by order $n/(\log n)^2$ within the critical window, these additional edges do not alter the mean component size, while the size of the largest component changes by a multiplicative~factor.%
\medskip 

\noindent\textbf{Model example.\ }
A canonical model in our family is the \emph{uniformly grown random graph}, also known as \emph{Dubins' model} or \emph{$c/j$-model}. Vertices arrive one by one, and when the $j$th vertex arrives it  connects independently to all vertices $i<j$ with probability $c/j$. 
The infinite version of this model was studied extensively in  \cite{durrett1990critical, kalikow1988random, shepp1989connectedness, zhang1991power}, where the critical value $c_\mathrm{crit}=1/4$ was identified. Bollob\'as \emph{et al.}~in \cite{BJR05} proved an infinite-order percolation phase transition for the finite version. 
In contrast to Erd\H{o}s--R\'enyi graphs, in the subcritical regime the largest component is of polynomial size~\cite{BJR05,MS25} with exponent $1/2-\sqrt{\varepsilon}$ when $c=1/4-\varepsilon$. At criticality, Durrett showed in~\cite{D} that the expected size of the component containing the first vertex is at most of order $\sqrt{n}/\log n$, and Janson and Riordan proved in \cite{JansonRiordanSus} that susceptibility at criticality is finite for a related model, see also \cite{Doro01,D,KrapDer04}.

\smallskip
We prove that Durrett's bound is sharp, and  that the largest connected component has the same order. 
Zooming in on the critical percolation value, we set $c=c_\mathrm{crit}+\varepsilon_n$ and prove for $\varepsilon_n\to0$,
\begin{align}
    |\sL_n|&\asymp \sqrt{|\varepsilon_n|}n^{1/2-\sqrt{|\varepsilon_n|}} &\text{if }\varepsilon_n(\log n)^2&\to-\infty, \nonumber\\
    |\sL_n|&\asymp \sqrt{n}/\log n,&\text{if }\varepsilon_n(\log n)^2&\to\alpha\in(-\infty,\pi^2),\label{eq:critical-window-intro}\\
    |\sL_n|&\gg \sqrt{n}/\log n,&\text{if }\varepsilon_n(\log n)^2&\to\alpha\in[\pi^2,\infty].\nonumber
\end{align}
These asymptotics illustrate how the size of the largest component transitions from polynomial with exponent smaller than $1/2$ into the critical regime, revealing a qualitatively new phase structure compared to the mean-field universality class in \eqref{eq:critical-er}.
When $\varepsilon_n\sim \pi^2(\log n)^{-2}$,
the size of the largest  component jumps to a larger order than $\sqrt{n}/\log n$.  When $\alpha \ge \pi^2$, we expect 
the emergence of a single component that grows sublinearly but of asymptotically larger order than all other components, 
see  Conjecture~\ref{conj}.

\smallskip
We next explain heuristically why the critical window differs sharply from the mean-field case. In Section \ref{sec:main} we state the results formally, holding for a large class of growing random graphs that  includes a scale-free random graph of preferential-attachment type.

\medskip
\noindent\textbf{Heuristics: Size of the largest component.\ } We outline the intuition behind the $\sqrt{n}/\log n$ scale.
The main reason for the fundamentally different behaviour of the phase transition, compared to Erd\H{o}s--R\'enyi graphs, is that the  \emph{local limit} or \emph{Benjamini--Schramm limit}, describing the neighbourhood of a typical vertex, is no longer  
a Bienaym\'e--Galton--Watson tree. In uniformly grown random graphs, it is 
instead given by
the trace of a branching random walk with a killing boundary, and the typical component size $|\sC_n|$ exhibits the same properties as the  total progeny in such a killed branching random walk (KBRW) at criticality. Its distribution has a lighter tail than the typical component size in critical Erd\H{o}s--R\'enyi graphs, and satisfies 
\[
\Prob\big(|\sC_n|\ge k\big)\ \asymp \ \frac{1}{k(\log k)^2}\ \ll\ \frac{1}{\sqrt{k}}.
\]
For general KBRW this behaviour was proved  by A\"id\'ekon \emph{et al.}\ in~\cite{AidekonHuZindy} under assumptions that do not hold in the present context (in particular, they require finiteness of the number of offspring of a particle). These asymptotics are made precise in Theorem \ref{thm:typtail} for fixed $k$ not depending on $n$.

\smallskip 
Assuming these asymptotics also for large $k=k(n)$, it follows that $\lim_{n\to\infty}\E[|\sC_n|]<\infty$ leading to finite susceptibility, while $\lim_{n\to\infty}\E[|\sC_n|\log |\sC_n|]=\infty$. Moreover, a back-of-the-envelope calculation predicts the asymptotic size of the largest component $\sL_n$. Assuming that $|\sL_n|$ is of order $k$, the expected number of vertices in components of size at least~$k$ should be of order~$k$. So, we aim to find $k$ solving
\[
\begin{aligned}
k\ \asymp\ \E\big[\#\text{vertices in components of size at least $k$}\big]\ = \ n \Prob\big(|\sC_n|\ge k\big) \ \asymp \ \frac{n}{k (\log k)^2},
\end{aligned}
\]
which gives $k$ of order $\sqrt{n}/\log n$. 

\pagebreak[3]

\medskip
\noindent\textbf{Heuristics: Critical window.\ }
The local limit is invariant under the sequence $\varepsilon_n\to0$ and thus provides no information about the change of behaviour within the critical window. To analyse the transition, we go beyond the local limit. For a fixed vertex in the graph of size $n$, we couple the exploration of its component to a branching random walk  (BRW) whose particles are killed when they exit an interval of length~$\log n$; the spatial coordinate 
of a particle in the branching random walk
corresponds to the logarithm of the arrival time of the corresponding vertex, generations correspond to one step in a breadth-first exploration. Under this logarithmic rescaling of vertex positions, each particle in the BRW produces offspring to its right at rate $c$, matching the expected number of edges a vertex $j$ creates when connecting to later vertices with probability~$c/j$.

\smallskip
Using many-to-few formulas, we express moments of the total progeny through a  centered random walk with  $\sigma^2=1/(2c_\mathrm{crit}^2)=8$, started from $x$ and
stopped upon leaving an interval of length $\log n$. The walk describes the distinguished \emph{spine} (ancestral line) arising in the many-to-one change of measure; each step corresponds to the next generation along that spine (and its position to some vertex), and the hitting time records when the position of the spine no longer corresponds to a vertex with label in~$[n]$. This construction relates the first moment of the progeny to the moment-generating function of the spine's hitting time $\tau_n$ of the interval boundaries, evaluated at $\varepsilon_n/c_\mathrm{crit}=4\varepsilon_n$. 

\smallskip 
The width and finiteness of the critical window are determined by the radius of convergence of this transform, $\E_x[\exp(4\varepsilon_n\tau_n)]$. Because the mean hitting time is at most of order $(\log n)^2$, the width of the critical window in the graph is of order $(\log n)^{-2}$. The radius of convergence is governed by the smallest eigenvalue of the Laplacian for Brownian motion with absorbing boundaries and diffusion coefficient $D=\sigma^2/2=4$, namely $\lambda_1=D(\pi/\log n)^2$. The transform $\E_x[\exp(4\varepsilon_n \tau_n)]$ is finite precisely when $4\varepsilon_n<\lambda_1$; this gives the threshold $\varepsilon_n=(\pi/\log n)^2$ marking the upper edge of the critical window. 
The dependence on  $\varepsilon_n(\log n)^2$ appearing in this
analysis is reflected in the function $\Si$ in \eqref{eq:CS} which encodes how
the size of the largest component scales across the critical window, see Theorem~\ref{thm:largest}.

\smallskip
{The width of order  $\smash{(\log n)^{-2}}$  and the upper boundary of the window 
are consistent with
the previously known asymptotics  for the size of the largest component in the subcritical 
and weakly supercritical regime. In the subcritical regime $c=c_\mathrm{crit}-\varepsilon$,
the largest component has polynomial order $n^{1/2-\sqrt{\varepsilon}}$ which suggests that the largest component at criticality is of order $n^{1/2+o(1)}$. In the weakly supercritical regime, the largest component contains about $\theta(c_\mathrm{crit}+\varepsilon)n$ vertices,  with $\theta(c)$  defined in \eqref{eq:theta-intro} for $C=\pi/2$.
This suggests that the largest component contains 
$n^{1/2+\delta}$ vertices when $\varepsilon_n(\log n)^2>\pi^2$, matching the upper edge of the critical window. The intuition behind $\theta$ is that it describes the probability that the branching random walk started from a typical vertex with only a right killing boundary survives forever.
}

\smallskip
In the next section we introduce the general growing random graph model studied in this paper, which includes a class of scale-free growing random graphs behaving similarly to preferential attachment models. The same criticality characteristics persist for these models, marking fundamentally different behaviour from universality classes observed in rank-one random graphs; we compare the two classes in Section~\ref{sec:rank-one}.

\section{Model definition and main results}\label{sec:main}
We now define the model of $\gamma$-growing random graph studied in this paper. It depends on two parameters, $\gamma$ and $\beta$. Note that we define it in general, but focus on $\gamma\in[0,1/2)$ in this paper.

\begin{definition}[$\gamma$-growing random graph]\label{def:pa-ann}
	Let $\gamma\in(-\infty,1)$ and $\beta>0$. Let $\sG_1$ be the graph with a single vertex labeled 1. At discrete steps  $j\in\N\setminus\{1\}$ a new vertex with label $j$ arrives and connects independently by an edge to present vertices $i\in[j-1]:=\{1,\ldots,j-1\}$  with probability
    \begin{equation}\label{eq:conn-prob2}
     p_{ij} = \big(\beta i^{-\gamma}j^{\gamma-1}\big)\wedge 1.
    \end{equation}
	to form $\sG_j$ from $\sG_{j-1}$. 
    We write $\Prob^{\beta}$ and $\E^\beta$ for the law and expectation of the random graph, and omit in our notation the dependency on $\gamma$.
    We write $\sC_n(v)$ for the connected component containing $v$ in $\sG_n$ and $\sL_n(\beta)$ or $\sL_n$ for the largest connected component in $\sG_n$ with arbitrary tie-breaking rule. 
    \end{definition}
    Throughout, we let the edge-density parameter $\beta=\beta_n<1$ depend on the final size $n$, but regard it as fixed while the graph grows from a single vertex into $\sG_n$. In the present paper, the minimum in \eqref{eq:conn-prob2} is never attained at one. 
     When $\gamma=0$, the model reduces to the uniformly grown random graph. 
     When $\gamma > 0$, arriving vertices prefer connecting to old vertices. The connection probability is asymptotically equal to the expected connection probability in affine preferential attachment, in which arriving vertices connects to earlier vertices with probability proportional to an affine function of their degree. In both models, old vertices have a higher degree.   The asymptotic degree distribution follows a power law with exponent $\tau=1+1/\gamma>2$, and has finite variance when $\gamma<1/2$. When $\gamma<0$, arriving vertices favour connecting to young vertices, and all degrees remain finite in the limit. Although we expect that similar results hold for $\gamma<0$ compared to $\gamma\in[0,1/2)$, we restrict ourselves to $\gamma\ge 0$ to reduce  technicalities. 

     \smallskip 
     The $\gamma$-growing random graph can be phrased as an inhomogeneous random graph as defined in \cite{bollobasRiordanJanson2007} (for any $\gamma$), and is therefore called \emph{inhomogeneous random graph of preferential-attachment type} in~\cite{MS24} for $\gamma>0$. Its infinitary version is also a special case of a random graph model introduced by Durrett and Kesten~\cite{durrett1990critical}, and used in the analysis of the critical window in rank-one random graphs with infinite-variance degrees~\cite{Bha24}.
The absence of edge dependencies makes the model analytically tractable and allows percolation to be performed simultaneously with graph generation.  This allows a detailed description of the percolation phase transition, while retaining the key mechanism of preferential attachment. We expect classical preferential-attachment models to exhibit similar critical behaviour.

\smallskip
As in most non-spatial sparse graphs, the percolation phase transition is non-trivial only when the degree distribution has finite variance~\cite{DerMor13, morters2023tangent}, i.e., when $\tau>3$ or, equivalently, $\gamma<1/2$.
Formally, the critical edge-density parameter is given by
\begin{equation}\label{eq:beta-c}
\beta_c(\gamma):=\inf\Big\{\beta>0 \colon \sup_{\varepsilon>0}\liminf_{n\to\infty} \Prob^\beta\big(|\sL_n|>\varepsilon n\big)=1\Big\}\ = \ \max(1/4 - \gamma/2, 0),
\end{equation}
as shown in \cite{morters2023tangent} by analyzing the survival probability of a branching random walk, the local limit of the graph.
Indeed, $\beta_c>0$, if and only if $\gamma<\frac12$.

\medskip
In the subcritical phase, when 
$\beta<\beta_c$, the temporal growth leads to a largest component of size
\smash{$n^{\rho+o(1)}$} as shown in \cite{MS25} (see Remark \ref{rem:oncor} on the $o(1)$-term), which is much larger than the largest degree in the graph which is of order $n^\gamma$.  This is contrary to rank-one models  where the subcritical largest component has the same order as the largest degree \cite{J08, Lienau}, but similar to genuine preferential attachment models, see \cite{Ray} and upcoming work by Banerjee \emph{et al.}~\cite{BanerjeeSubcritial}. We have 
\begin{equation}\rho=\tfrac{1}{2}-\sqrt{4\beta_c|\beta-\beta_c|}=\tfrac12-\sqrt{(\tfrac12-\gamma)^2-\beta(1-2\gamma)}
>\gamma.\label{eq:rho-sub}\end{equation}

\smallskip
In the supercritical phase, when
$\beta>\beta_c$, there is a unique macroscopic component containing a proportion $\theta(\beta)$ of the vertices, where $\theta(\beta)$ can be expressed as the survival probability of the local limit. When $\beta\downarrow\beta_c$, we observe that the phase transition is of infinite order, as
\begin{equation}\label{eq:barely-sup}
\theta(\beta)=
\exp\bigg(-\frac{\pi/2+o(1)}{\sqrt{4\beta_c|\beta-\beta_c|}}\bigg), 
\end{equation}
 shown in \cite{Leif} for our setup, similar to the analysis in \cite{EMO} for a preferential-attachment model. 

\subsection{Main results}
We consider the graphs around the critical edge density $\beta_c$ whenever it is positive; thus, we assume $\gamma<1/2$. We let $(\beta_n)_{n\ge 1}$ be a sequence tending to $\beta_c$ and focus on a window of order $1/(\log n)^2$ around $\beta_c$: one should think of sequences such that $4\beta_c(\beta_n-\beta_c)(\log n)^2\to\alpha\in[-\infty, \infty]$. If
$\alpha\in(-\infty,\pi^2)$, this window marks the scale on which the largest component remains of the same order as at criticality, and reveals how the constant prefactor varies as $\alpha$ is changed. 
The following function describes the impact of varying $\alpha$, 
\begin{equation}\label{eq:CS}
    \Si(\alpha):=\begin{dcases}
        \frac{\sin(\sqrt{\alpha})}{\sqrt{\alpha}},&\text{if }\alpha>0,\\
        1,&\text{if }\alpha = 0,\\
        \frac{\sinh\big(\sqrt{|\alpha|}\big)}{\sqrt{|\alpha|}},&\text{if }\alpha<0.
    \end{dcases}
\end{equation}
For $\alpha<\pi^2$, $ \Si(\alpha)$ is an analytic, convex, and decreasing function, with $\Si(\alpha)\to\infty$ as $\alpha\to-\infty$, and $\Si(\alpha)\downarrow 0$ as $\alpha\uparrow\pi^2$. Note that $\Si$ should not be mistaken for the sine integral.

\smallskip
We state our main theorem: It establishes the critical window, and shows the precise scaling of the largest component in this window and just below it.  
\begin{theorem}[Largest component around criticality]\label{thm:largest}
	Fix $\gamma\in[0, 1/2)$. There exists a positive-valued function $\varepsilon\mapsto M_\varepsilon$ such that for any sequence $\beta_n\to\beta_c$ with $\limsup_{n\to\infty}4\beta_c(\beta_{n}-\beta_c)(\log n)^2<\pi^2$, there exists a constant $n_0$ such that for all $n\ge n_0$ and $\varepsilon>0$,
     \[
     \Prob^{\beta_n}\bigg(\big|\sL_n\big|\Big/ \frac{\sqrt n/\log n}{\Si\big(4\beta_c(\beta_{n}-\beta_c)(\log n)^2\big)}\, \in [1/M_\varepsilon, M_\varepsilon]\bigg) \ge 1-\varepsilon.
     \]
\end{theorem}

\begin{remark} 
    We stress that the index $n_0$ only depends on the sequence $(\beta_n)_{n\ge 1}$, not on $\varepsilon$. The function $\varepsilon\mapsto M_\varepsilon$ is universal: it does not depend on the sequence within the window. In a sense, the dependence on the sequence $(\beta_n)_{n\ge1}$ is entirely captured by the term involving the function $\Si$. 
\label{rem:onmaintheo}
\end{remark}

The following corollary restates the three asymptotics formulated in \eqref{eq:critical-window-intro} for uniformly growing random graphs, when $4\beta_c=1$, and generalizes these asymptotics to inhomogeneous random graphs of preferential-attachment type with non-trivial phase transition, i.e., when $\gamma\in(0, 1/2)$.
For a sequence of random variables $(X_n)_{n\ge1}$, we write $X_n\gg a_n$ if $X_n/a_n\to\infty$ in probability, and $X_n\asymp a_n$ if $(X_n/a_n)_{n\ge 1}$ is tight on $(0,\infty)$, that is, $\lim_{M\to\infty}\inf_{n\in \N}\Prob\big(\tfrac{1}M\le X_n/a_n\le M\big)=1.$
\begin{corollary}[Barely subcritical phase, critical window, and transition across $\pi^2$]\label{cor:critical}
    Fix $\gamma\in[0, 1/2)$. 
     Let  $\beta_n\to\beta_c$ be a sequence such that $4\beta_c(\beta_{n}-\beta_c)(\log n)^2\to\alpha\in[-\infty, \infty].$ Then, 
     \begin{align}
    |\sL_n(\beta_n)|&\asymp \sqrt{|\beta_n-\beta_c|}n^{1/2-\sqrt{4\beta_c|\beta_n-\beta_c|}} &\text{if }\alpha&\in[-\infty, 0), \nonumber\\
    |\sL_n(\beta_n)|&\asymp \sqrt{n}/\log n,&\text{if }\alpha&\in(-\infty,\pi^2),\nonumber\\
    |\sL_n(\beta_n)|&\gg \sqrt{n}/\log n,&\text{if }\alpha&\in[\pi^2,\infty].\nonumber
\end{align}
\end{corollary}

\begin{proof}
     If $\alpha\in(-\infty,\pi^2)$,  then 
     $\Si(4\beta_c(\beta_{n}-\beta_c)(\log n)^2)$ converges to a constant and Theorem~\ref{thm:largest} implies the tightness. 
    If $\alpha\ge\pi^2$, the limit follows from a straightforward monotone coupling of $|\mathscr L_n|$
    as a function of $\alpha$ together with the fact that $\Si(\alpha)\to0$ as $\alpha\to\pi^2$. For the first case,  it suffices to show that 
        \[
        \sqrt{4\beta_c|\beta_n-\beta_c|}n^{-\sqrt{4\beta_c|\beta_n-\beta_c|}}\asymp \frac{1}{\Si\big(4\beta_c(\beta_n-\beta_c)(\log n)^2\big)\log n}.
        \]
        Substituting the definition of $\Si$ in~\eqref{eq:CS}, and rearranging terms, this is equivalent to showing that 
        \[
        n^{\sqrt{4\beta_c|\beta_n-\beta_c|}}\asymp \sinh\big(\sqrt{4\beta_c|\beta_n-\beta_c|(\log n)^2}\big).
        \]
        There exists $\varepsilon,\delta_\varepsilon>0$ such that $\sqrt{4\beta_c|\beta_n-\beta_c|(\log n)^2}\ge\varepsilon$ for all large $n$, and 
        $\sinh(y)\in[ \delta_\varepsilon\re^y, \re^y]$ for $y\ge \varepsilon$. This finishes the proof.
\end{proof}

\begin{remark}\ \\[-3mm]
\begin{itemize}[leftmargin=18pt]
    \item[(i)] The asymptotics in the first and second case coincide when $\alpha\in(-\infty, 0)$. In the barely subcritical phase, when $\alpha=-\infty$, the largest component is of smaller order than $\sqrt{n}/\log n$. 
    \item[(ii)]
    Although we prove the Theorem \ref{thm:largest} (and thus Corollary~\ref{cor:critical}) for sequences $\beta_n\to\beta_c$, the proof can be adapted to the subcritical regime when $\beta<\beta_c$ is fixed, sharpening \cite{MS24} by removing the $o(1)$-term in the exponent $\rho$ in \eqref{eq:rho-sub} that describes the polynomial size of the largest component. This would lead to additional case distinctions in our proofs, and we refrain from including it.
\end{itemize}   
\label{rem:oncor}
\end{remark}

The following conjecture addresses what happens for $\beta_n\to\beta_c$ just above the critical window, and we plan to resolve it in forthcoming work. It is inspired by the infinite order phase transition  described in \cite{BJR05, EMO, Leif}, where  first $n\to\infty$, and then $\beta\downarrow \beta_c$, see also \eqref{eq:barely-sup}. The refined version of the conjecture (without the $o(1)$ term in the exponent compared to \eqref{eq:barely-sup}) is motivated by precise asymptotics for the survival probability of weakly super-critical killed branching Brownian motion \cite{Berestycki2010a}. 
\begin{conjecture}[Barely supercritical phase]
\label{conj}
Fix $\gamma\in[0,1/2)$. There exists a constant $c>0$ such that for any sequence  $\beta_n\to\beta_c$ such that $\liminf_{n\to\infty}4\beta_c(\beta_n-\beta_c)(\log n)^2> \pi^2$, 
\[
\big|\sL_n(\beta_{n})\big|\Big/ n\exp\bigg(-\frac{\pi/2}{\sqrt{4\beta_c(\beta_n-\beta_c)}}\bigg) \overset{\Prob}{\longrightarrow }c.
\]
\end{conjecture}
\begin{remark}When $4\beta_c(\beta_n-\beta_c)(\log n)^2\to\alpha\in(\pi^2,\infty)$, the conjecture predicts a component of polynomial size with exponent $1-\pi/(2\sqrt{\alpha})>1/2$. 
We anticipate a secondary critical window (of unknown width) when $\alpha=\pi^2$ in which the largest component jumps from order $\sqrt{n}/\log n$
to $\sqrt{n}$ to larger order, and in which the susceptibility blows up (see Theorem~\ref{thm:susceptibility} below).
\end{remark}

We next describe the tail of the component size of a typical vertex, and show that it behaves similar to the distribution of the total progeny in critical killed branching random walk \cite{aidekon2010tail, AidekonHuZindy}. The theorem applies  to any sequence $\beta_n\to\beta_c$, and does not require it to be in  the critical window. Let $O_n$ be a vertex chosen uniformly at random from $[n]$.  

\begin{theorem}[Tail of typical component size]\label{thm:typtail}
	Fix $\gamma\in[0,1/2)$.
    There exist constants $c, C>0$, such that for all $k\in\N$, and any sequence $\beta_n$ tending to $\beta_c$,
	\begin{equation}\label{eq:typtail}
        \frac{c}{k(\log k)^2} \leq \lim_{n \to\infty}\mathbb P^{\beta_n} \big( |\sC_{n}(O_n)| \ge k \big)\leq \frac{C}{k(\log k)^2}.
	\end{equation}
\end{theorem}
\begin{remark}
The summability in $k$ suggests that the first moment of $|\sC_n(O_n)|$ is bounded. However, this is not immediately implied as the result is proved for the limit $n\to\infty$ and fixed $k$.  
By analogy with the critical Erd\H{o}s–R\'enyi random graph, in the critical window we expect the same asymptotic to hold for any $k=k(n)$ smaller than the typical size of the largest component, which is of order $\sqrt{n}/\log n$ in the critical window.  
The most likely mechanism for a vertex to be in a component of size at least $k_n$ is the existence of a path from $O_n$ to an old vertex whose expected component size is of order $k_n$.  
Upper bounds for such $k_n$ follow from our techniques below, whereas matching lower bounds are more delicate, requiring a careful treatment of collisions in both phases of the exploration; see Sections~\ref{sec:overview}–\ref{sec:components}.
\end{remark}
\noindent We proceed to the susceptibility, which we define as in Janson and Riordan~\cite{JansonRiordanSus} as the (random) quantity
\[
\frac{1}{n}\sum_{v=1}^n|\sC_n(v)|\ =\ \frac{1}{n}\sum_{\sC\subseteq \sG_n}|\sC|^2,
\]
where the second sum runs over all connected components in $\sG_n$.
The next theorem shows that the limiting susceptibility remains finite and deterministic throughout the critical window, unlike in the mean-field universality class~\cite{Luczak}.
In particular, the limit is independent of the position of the sequence $(\beta_n)_{n\ge 1}$	
  within this window and it converges also in expectation. However, higher moments of the component sizes still diverge as criticality is approached.
\begin{theorem}[Finite susceptibility]\label{thm:susceptibility}
	Fix $\gamma\in[0, 1/2)$. Let $\beta_n\to\beta_c$ be such that $\limsup_{n\to\infty}4\beta_c(\beta_n-\beta_c)(\log n)^2<\pi^2$. Then,
	\begin{equation}\label{eq:thm-sus}
		\frac{1}{n}\sum_{v=1}^n|\sC_{n}(v)|\overset{\Prob, L^1}\longrightarrow 4(1-\gamma^2), \text{ as } n\to\infty.
	\end{equation}
    Moreover, 
    \begin{equation}\label{eq:thm-logsus}
		\frac{1}{n}\sum_{v=1}^n|\sC_{n}(v)|\log|\sC_n(v)|\overset{\Prob}\longrightarrow \infty, \text{ as }  n\to\infty.
	\end{equation}
\end{theorem}

\begin{remark}We expect the susceptibility to diverge if $\liminf_{n\to\infty}4\beta_c(\beta_n-\beta_c)(\log n)^2>\pi^2$, and we plan to address this in future work. Note that it would be implied by Conjecture \ref{conj}, as each vertex in the largest component of size $n^{1/2+\varepsilon}$ contributes  $n^{1/2+\varepsilon}$ to the sum. The boundary case $4\beta_c(\beta_n-\beta_c)(\log n)^2\to\pi^2$ remains open; if the sequence is such that the largest component is of order $\sqrt{n}$, we believe that the susceptibility may remain finite but random.%
\end{remark}

\subsection{Generalisation to related connection probabilities}
The results in Theorems \ref{thm:largest}, \ref{thm:typtail} and \ref{thm:susceptibility} are robust under small perturbations of the connection probability $p_{ij}$ defined in \eqref{eq:conn-prob2}. In fact, we only require that edges are independent, and that, for all $j>i\ge 1$, the connection probabilities satisfy 
\begin{equation}\label{eq:conn-prob-range}
    1-\exp\big(-\beta i^{-\gamma}j^{\gamma-1}\big)\ \le \ p_{ij}\ \le\ 1-\exp\big(-\beta (i-\tfrac12)^{-\gamma}(j-\tfrac12)^{\gamma-1}\big).\end{equation}
These bounds are only used directly in the proof of Proposition \ref{prop:domination-upper}, see from \eqref{eq:upper-12} onward for the proof that $p_{ij}$ satisfies the upper bound whenever $\beta\in(0, \frac12]$ (this restriction is harmless as $\beta_c\le 1/4$ for $\gamma\in[0,1/2)$ by \eqref{eq:beta-c}, and we assume $\beta_n\to\beta_c$).
\smallskip

Janson's contiguity results for random graphs, see \cite[Corollary 2.12(ii)]{J10}, allow us to extend almost all results to models with independent edges whose connection probabilities may violate \eqref{eq:conn-prob-range} but remain sufficiently close to $p_{ij}$. If edges appear with probability \smash{$p'_{ij}$}, then all results except the $L^1$-convergence in \eqref{eq:thm-sus} continue to hold whenever  \smash{$\sum_{i<j<\infty}(p_{ij}-p'_{ij})^2/p_{ij}<\infty$}.
The $L^1$-convergence in \eqref{eq:thm-sus} requires an extra condition; for instance, the upper bound  on the connection probability in \eqref{eq:conn-prob-range} suffices.

\subsection{Universality landscape}\label{sec:rank-one}
We conclude by positioning $\gamma$-growing random graphs within the broader landscape of universality classes. 
Their critical behaviour with finite-variance degree distribution forms a universality class distinct from the mean-field class (Erd\H{o}s--R\'enyi, configuration model with bounded third moment), and the heavy-tailed rank-one class containing the configuration model with finite variance and infinite third moment
degree distribution.

\smallskip
In rank-one models, the local limit in the critical window
is still a critical unimodular Bienaym\'e--Galton--Watson tree, and a depth-first exploration of the components in rank-one graphs with finite third moment gives a coupling with random walks with finite-variance jumps. This relates component sizes to excursions of Brownian motion with parabolic drift depending on the position in the critical window~\cite{Bha10, Dhara, Joseph2014}. When the third moment is infinite, similar techniques apply, now with infinite-variance jumps, yielding a relation to L\'evy processes~\cite{bhamidi2012novel, ConchonGoldschmidt2023stable, Dhara2}. Their critical components are smaller than in mean-field models but larger than in the growing case---leading to diverging susceptibility---while their scaling window is wider than in mean-field but narrower than in growing random graphs.

\smallskip 
For $\gamma$-growing random graphs the mechanism is fundamentally different. Here temporal growth governs component formation, and the local limit is the trace of a single-type branching random walk in which space encodes logarithmic arrival times, {observed first in \cite{DerMor13} for a related model}. {Much earlier, Bollob\'as, Janson and Riordan~\cite{BJR05}
studied} the uniformly grown model via path counting in the subcritical regime to establish the infinite-order phase transition. Janson and Riordan
used the local limit {(phrased as multi-type branching process)} to prove finite susceptibility at criticality for a related model through integral operators in \cite{JansonRiordanSus}. In their proofs they let $n\to\infty$ before sending $\varepsilon\to0$ to deduce critical behaviour. While these methods extend to a wide family of inhomogeneous random graphs, this order of limits is not designed to capture a critical window, since local limits do not depend on converging sequences \smash{$(\varepsilon_n)_{n\ge 1}$}.%
\smallskip

\pagebreak[3]

By contrast, our analysis couples breadth-first component explorations to two-sided killed branching random walks, which remain tractable for the growing and scale-free models here. They allow us to follow the $n$-dependence  inside the critical window. We expect genuine preferential-attachment graphs---where new vertices connect according to degree instead of age---to exhibit qualitatively the same behaviour and thus to fall into the same universality class. Their local limit is a multi-type branching random walk \cite{BBCSLocalLimitPA,DerMor13, garavaglia2022universality} reflecting degree correlations. Although edge dependencies make analysis substantially harder, the underlying temporal mechanism is similar. 
These distinctions explain why growing graphs exhibit a slow, infinite-order transition from the subcritical to the supercritical phase, with a critical window
bounded from above, separating them from graphs in the mean-field universality~class.\smallskip

We begin with an overview of our method of proof and leave the detailed proofs to Sections~4--10. \pagebreak[3]
\smallskip

\section{Overview of the proofs}\label{sec:overview}
To bound the size of the largest component, we start from the observation that early vertices are the most likely to belong to large components.  
The components containing the oldest vertices should intuitively have sizes of the same order as the largest one.  
Our goal is to show that this intuition is correct and that all remaining components are smaller to obtain an upper bound. We use the simple observation that if the largest component has size at least $s$, there must be at least $s$ vertices in a component of size at least $s$. By Markov's inequality,
\[
\Prob^{\beta_n}\big(|\sL_n|\ge s\big)\le \frac{1}{s}\sum_{v=1}^n\Prob^{\beta_n}\big(|\sC_n(v)|\ge s\big).
\]
Thus, obtaining lower and upper bounds for Theorems~\ref{thm:largest},~\ref{thm:typtail}, and~\ref{thm:susceptibility} reduce to bounding component sizes of individual
vertices for which we employ first- and second-moment methods.
To compute these moments,
we rely on breadth-first explorations of components in the graph, which can be coupled to a \emph{branching random walk (BRW)} killed beyond two barriers at distance $\log n$ apart. A BRW is a system of particles distributed on the real line, such that a particle at position $x$ in generation $k$ generates particles in generation $k+1$ according to some point process $\mathscr P$, recentred around $x$, independently from the other particles in generation $k$. Our model is related to a branching random walk in the following way: consider the breadth-first exploration process of the component containing some vertex $v\in [n]$. If we consider vertices as particles and consider the logarithm of the label of the vertex as the position of the particle, then this breadth-first exploration can be approximated by a \emph{killed branching random walk (KBRW)}, in which particles are killed at positions above $\log n$---preventing edges to vertices that arrive in the future---and below the origin---preventing edges to non-existing vertices with negative label. The point process $\mathscr P$ is a Poisson process with intensity measure $\mu_\beta$ defined, for $\gamma\in[0, 1)$, by
\begin{equation}
	\label{eq:mu_beta}
	\mu_\beta(\rd y) := \beta\cdot \big(\re^{(1-\gamma)y}\ind{y<0}+ \re^{\gamma y}\ind{y\ge 0}\big)\,\rd y = \beta \re^{y/2 - (1/2-\gamma)|y|}\,\rd y = \beta \re^{y/2 - 2\beta_c|y|}\,\rd y,
\end{equation}
where the last equality holds only for $\gamma\le1/2$.
The term for $y<0$ shows that each particle has almost surely finitely many children to its left (backwards in time, corresponding to older vertices). Moreover, as $\gamma$ increases, the left tail decays more and more slowly, reflecting preference to  older vertices. For $y>0$ and $\gamma\in[0,1)$, each particle produces  infinitely many children to its right (forward in time). In fact, under the logarithmic parametrisation, a particle located at distance $\log i$ from the left killing boundary produces roughly $(n/i)^\gamma$ offspring to its right when $\gamma>0$, matching the expected number of neighbours of vertex $i$ in a graph of size $n$; this  makes the $\log n$ spatial scale natural for the coupling. 

\smallskip
A connection between the exploration process of a nonlinear preferential attachment model and a KBRW was first observed in~\cite{DerMor13}, who used it to construct the local limit and identify the  critical value~$\beta_c$  for their model. In these models, the local limit can be seen as the KBRW starting at distance $X$ from a right killing boundary,  with $X$ an exponentially distributed random variable correcting for the logarithmic scaling. Note that in this limit, the left killing boundary plays no role. By contrast, in this paper we couple the exploration with the branching random walk for fixed large~$n$, and both killing boundaries become essential.  
\smallskip 

\pagebreak[3]
 The branching random walk with right killing boundary was also used in \cite{MS25} to study the largest component in the subcritical regime $\beta < \beta_c$ for the inhomogeneous random graph of preferential-attachment type.
 In the subcritical regime, all particles of the branching random walk drift to $+\infty$ with positive speed, so that the exploration process can be approximated by a finite lookahead algorithm. At criticality, when $\beta = \beta_c$,  the minimum particle in the branching random walk still drifts to $+\infty$ but with speed zero. This considerably complicates matters. \smallskip

We now describe the main ideas to bound component sizes of individual vertices in the proof of Theorem~\ref{thm:largest} on the size of the largest component. 
We use the coupling to the branching random walk killed outside $[0, \log n]$. 
Two-sided killing has been a key tool in
Aïdékon~\cite{aidekon2010tail} and Aïdékon \emph{et al.}~\cite{AidekonHuZindy}  to study the total progeny of the one-sided killed branching random walk, under the assumption of a finite number of offspring (and some bounded moments). There are three major difficulties in our study. 
\smallskip

First, we must allow for infinitely many offspring. We do so by reformulating the classical \emph{many-to-one} and \emph{many-to-two} formulas. Classically, these express sums over particles at a \emph{fixed generation} or along a stopping line in terms of a single random walk, see e.g.\ Biggins and Kyprianou~\cite{Biggins2004} or Shi~\cite{ShiLectureNotes}. We instead derive  formulas for moments of quantities summing over \emph{all} particles of the branching random walk, see Lemmas~\ref{lem:many-to-one}--\ref{lem:many-to-two}. While this may appear as a simple generalization, we are not aware of previous use of this approach in the branching random walk literature.\pagebreak[3]
\smallskip

Second, recall that we allow the intensity parameter $\beta$ to depend on $n$. This has the effect that the branching random walk with only a single killing boundary is non-critical. Our goal is to show that
adding a second killing boundary at distance $\log n$ makes it behave like a critical BRW whenever $(\beta_n)_{n\ge 1}$ is in the critical window. In the moment calculations, this results in a multiplicative factor $(\beta/\beta_c)^k$, where $k$ is the number of steps of the centred random walk $(S_k)_{k\ge0}$.
We therefore need sharp multiplicative bounds on the resolvent kernel of this random walk killed outside the interval $[0,L]$, with $L \approx \log n$, i.e., on 
 \[
 R_\beta g(x) = \E_{x}\left[\sum_{k=0}^\infty (\beta/\beta_c)^k g(S_k)\Ind{S_i\in [0,L]\,\forall i\le k}\right],
 \]
 for some functions $g$. Moreover, we need these bounds to be uniform for $x$ in the whole range of the interval. Mogulskii's theorem~\cite{MogulskiiSmallDeviations1975}, which is typically used to compare a random walk between two barriers to a Brownian motion (in the small deviation regime, which is the one that matters here), only gives bounds which are precise at the exponential scale, which is by far not sufficient for our purposes. Instead, we use the fact that the random walk  $(S_k)_{k\ge0}$ appearing in our case is quite special as its steps follow the Laplace distribution. The corresponding resolvent solves the Fredholm integral equation 
\[
f(x) = g(x) + \frac{\beta}{2\beta_c}\int_0^L \re^{-|x-y|} f(y)\,\rd y.
\]
When $g$ is an exponential function, there are exact formulas for the solutions to this equation, which yield 
bounds on first and second moments of the total progeny of the branching random walk killed outside $[0,L]$, which are sharp up to multiplicative constants, see Propositions~\ref{prop:first-moment-asymp} and~\ref{prop:second-moment-upper}. 
\smallskip

The third difficulty is to transfer the branching random walk results to the graph exploration.
The branching random walk evolves in continuous space, and we let  each vertex occupy a small interval in this space. Because particles evolve independently once created, many particles may (and will) occupy the same interval, creating a \emph{collision}. In the graph, only the first real particle in each interval corresponds to a discovery of a corresponding vertex, and any later particles lead to overcounting. We therefore call a colliding particle and its descendants \emph{fake}.
Proposition~\ref{prop:fake} shows that the expected number of fake particles is at most a small multiple of the expected total number of particles. This relies on a second moment bounding the expected number of collisions (Lemma~\ref{lem:fake}) and on estimates of the random-walk resolvent $R_\beta$. The resulting bounds are uniform in the sequence $(\beta_n)_{n\ge 1}$ within the critical window, which is essential for controlling the total number of fake particles generated from a small interval, which itself depends sensitively on the sequence and spatial position if a collision occurs.
A crucial ingredient is a bound on the  random walk resolvent~$R_\beta$ applied to $g$ being the indicator function of a small interval, see Lemma~\ref{lem:resolvent}. The explicit formulas used above are of no help here, as they require $g$ to be  exponential. We therefore obtain an upper bound through direct analytic arguments, making use of a reduction of the Fredholm integral equation to a second-order ODE, and bounding $g$ by suitably chosen functions. 
\smallskip

\pagebreak[3]
Having established Theorem~\ref{thm:largest}, the proof of Theorem~\ref{thm:typtail} reduces to bounding the distribution of the total progeny of the local limit of the graph, in which collisions do not occur. The local limit is independent of the sequence $\beta_n\to\beta_c$, and the analysis follows Aïdékon's approach~\cite{aidekon2010tail}, see Lemma~\ref{lem:local-lim-progeny}, using our previous lemmas with $\beta=\beta_c$.
\smallskip 

\pagebreak[3]

To prove Theorem~\ref{thm:susceptibility}, we show that the mean component size is asymptotically equal to the expected size of the local limit, and therefore does not depend on the sequence $(\beta_n)_{n\ge 1}$, provided it lies within the critical window. By contrast, Theorem~\ref{thm:largest} demonstrates that the largest component (typically containing some old vertices) depends sensitively on the value of $(\beta_n)_{n\ge 1}$ inside the critical window. Using a decomposition of components according to their oldest vertex, we show in Lemma \ref{lem:negligible} that the number of vertices in components with old vertices is still negligible. Hence the main contribution to the mean component size comes from components restricted to young vertices, which are largely unaffected by the value $\beta_n$ and determined by the local limit, whose expected size follows from Proposition \ref{prop:first-moment-asymp}.
\pagebreak[3]

\subsection*{Organisation of the paper} Section \ref{sec:components} formalises the coupling between a breadth-first component exploration and the killed branching random walk, and defines the local limit along with corollaries useful for proving Theorems \ref{thm:typtail}--\ref{thm:susceptibility}.  Section \ref{sec:many-to-few} presents the many-to-few formulas,  used in Section \ref{sec:progeny} to obtain bounds on the moments of the progeny of the branching random walk killed from two sides, and in Section~\ref{sec:real} to analyse collisions. The remaining sections prove the main results: Theorem \ref{thm:largest} on the largest component in Section \ref{sec:largest}; Theorem \ref{thm:typtail} on the component size distribution in Section \ref{sec:progeny-local}; and Theorem \ref{thm:susceptibility} on the finite susceptibility in Section \ref{sec:sus}.

\section{Branching random walks and component exploration}\label{sec:components}
To analyse the component of a given vertex~$v$ in the finite graph~$\sG_n$, we show in this section how to couple it with a killed branching random walk, inspired by a breadth–first exploration process. We then show that the same branching random walk describes the local limit of the model when parameters are fixed at their critical values.

\subsection{Coupling component exploration with a branching random walk}\label{sec:bfe}
In the $\gamma$-growing random graph the vertices arrive one per unit time; for the coupling we accelerate time exponentially, assigning to each vertex a \emph{position} that is a logarithmic function of its label.   We use two logarithmic embeddings, one for the lower and one for the upper bound on component sizes.

\smallskip\noindent\emph{Logarithmic vertex embedding.\ }
For the lower bound, we place vertex~$i$ at position~$x_i^-=\log i$, and  we associate to vertex~$i$ the interval to its right until the position of vertex $i+1$. We set for $i,m,n\in\N$, and $x\in[0,\infty)$,
\begin{equation}\label{eq:i-lower}
	x_i^-=\log i, \qquad I_i^-=\big[\log i,\log(i+1)\big),\qquad 
	I^-_{[m,n]}=\big[\log m,\log(n+1)\big), \qquad \ell^-(x)=\lfloor \re^x\rfloor,
\end{equation}
where $\ell^-(x)$ yields the vertex label corresponding to position $x$. That is, $\ell^-(x)=i$ iff $x\in I_i^-$.

\smallskip
For the upper bound we associate slightly larger intervals to each vertex. We place vertex $i$ at position $x^+_i=\log(2i-1)$ and associate to
each vertex its interval to the left. We set for $i,m,n\in\N$, and $x\in[0,\infty)$,
\begin{equation}\label{eq:i-upper}
\begin{aligned}
x^+_i&=\log(2i-1), 
	&I_i^+&=
    \begin{dcases}
        \big(\log(2i-3), \log(2i-1)\big],\phantom{..}&\text{if }i\ge2,\\
        \{0\},&\text{if }i=1,
    \end{dcases}\\
    \ell^+(x)&=\begin{dcases}\lceil \tfrac{1}{2}\re^x\rceil+1,&\text{if }x>0,\\
    1,&\text{if }x=0,\end{dcases}
    &I^+_{[m,n]}&=\begin{dcases}
        \big(\log(2m-3),\log(2n-1)\big],&\text{if }m\ge2,\\
       \big [0,\log(2n-1)\big],&\text{if }m=1.
    \end{dcases}
    \end{aligned}
\end{equation}
As $\log(-1)$ is undefined, vertex $1$ has a singleton as associated interval, requiring slightly special treatment later. Note that $|I_i^+|=\log\big(1+\tfrac{1}{i-3/2}\big)=1/i+O(1/i^{2})$ as $i\to\infty$, so $|I_i^+|$ and $|I_i^-|$ exhibit the same asymptotic scaling. From now on,
we drop the superindex $\pm$ whenever it is clear from the context.

\medskip \noindent\emph{Branching random walks.\ }
We now proceed to define branching random walks, following Jagers~\cite{Jagers1989}. Note that we only deal here with branching random walks where every particle has an infinite number of children, in contrast to large parts of the literature, see e.g.~Shi~\cite{ShiLectureNotes}.
Let $\mathscr U$ be the Ulam--Harris tree consisting of words over the alphabet $\N := \{1,2,\ldots\}$ of arbitrary length, i.e.,
\begin{equation}\label{eq:ulam-harris}
	\mathscr U = \bigcup_{n=0}^\infty \N^n,\quad\text{ with } \quad\N^0:=\{\emptyset\}.
\end{equation}
For $s\in \mathscr U$, let $|s|$ denote the generation of $s$, i.e.,~the length of the word $s$. We write $s\le t$ if $s$ is an ancestor of $t$ in $\mathscr U$, i.e.,~if $s$ is a prefix of $t$.

\pagebreak[3]
A branching random walk is a random field $(X_s)_{s\in \mathscr U}$ taking values in $\R$ and defined as follows. Let $\xi = (\xi_1,\xi_2,\ldots)$ be a random sequence taking values in $\R$ and denote by $\mathcal L_\xi$ its law. Let $(\xi^s)_{s\in\mathcal U}$ be iid copies of $\xi$. The \emph{branching random walk with reproduction law $\mathcal L_\xi$ and starting from $x\in\R$} is defined recursively as follows:
\begin{equation}\label{eq:displacement}
	X_\emptyset= x,\quad \forall s\in\mathcal U,\ j\ge 1: X_{sj}=X_s+\xi^s_j.
\end{equation}
Heuristically, a branching random walk is a particle system on the real line, where each particle independently gives birth to an infinite number of children which, starting from the position of their parent, jump according to the law of the random vector $\xi$. We therefore call an element $s\in\mathcal U$ sometimes a \emph{particle} of the branching random walk and $X_s$ its position. \smallskip

We now specify the reproduction law relevant for our coupling, and give notation for the progeny in which particles are killed outside an interval. 

\begin{definition}[$\gamma$-branching random walk]\label{def:pa-brw}
	Fix $\gamma\in[0, 1/2)$, and let $\beta>0$. Let $\xi$ be the random sequence of atoms in increasing order of a Poisson point process on $\R$ with intensity $\mu_\beta$ defined in~\eqref{eq:mu_beta}. The $\gamma$-branching random walk $(X_s)_{s\in\mathcal U}$ starting from $x\in\R$ is defined to be the branching random walk with reproduction law $\mathcal L_\xi$ starting from $x$. We denote by $\Prob_x^\beta$ its law and write $\E_x^\beta$ for the expectation with respect to it.
\end{definition}
    
    Let $I\subseteq\R$ be an interval. We are interested in the \emph{total progeny} of the $\gamma$-branching random walk killed outside $I$, defined by
\begin{equation}\label{eq:total-progeny}
	T_{I} := \sum_{t\in \sU} \Ind{\forall s \le t: X_s \in I}.
\end{equation} 
Let $m\le v\le n$.
	We denote by \smash{$\sT^\pm_{[m,n]}(v)$} the  total progeny of the $\gamma$-branching random walk starting from $x_v^\pm$ killed outside the interval $\smash{I^\pm_{[m,n]}}$.

\medskip \noindent\emph{Revealing the branching random walk.\ }
We view the branching random walk as being revealed as follows: Each generation is exposed by exploring all offspring of the preceding generation from its leftmost particle to its rightmost particle; for each particle, we reveal its offspring from left to right in increasing spatial position. Formally, we define a total order $\prec$ on particles $s,t\in\sU$. We say $s\prec t$ if \emph{(i)} $|s|<|t|$; or if \emph{(ii)} $|s|=|t|$ and if $s'$ and $t'$ denote the parents of $s$ and $t$, respectively, we have $X_{s'} < X_{t'}$, or if \emph{(iii)}~$s$ and $t$ have the same parent $u$ and $s = ui$, $t = uj$ with $i<j$.
\smallskip

Note that $\prec$ is almost surely well-defined, since $\mu_\beta$ is non-atomic (see~\eqref{eq:mu_beta}), and therefore all particle positions are distinct almost surely. Also note that in \emph{(iii)}, we have $i<j$ if and only if $X_s < X_t$, by definition of the reproduction law of the $\gamma$-branching random walk.

\begin{definition}[Real, fake, and colliding particles]\label{def:real}
Let \smash{$\sT^\pm_{[m,n]}(v)$} be the total progeny of the $\gamma$-branching random walk restricted to \smash{$I^\pm_{[m,n]}$}.  We declare the root to be \emph{real}, and assign types to all other particles in \smash{$\sT^\pm_{[m,n]}$} recursively in the order $\prec$.  If a particle $t$ has position \smash{$X_t\in I^\pm_i$} for a vertex $i\in[m,n]$, then $t$ is called \emph{real} if its parent $t'$ is real and if there is no real particle \smash{$s\in \sT^\pm_{[m,n]}$} with $s\prec t$ and \smash{$X_s\in I^\pm_i$}.  If \smash{$I^\pm_i$} contains a real particle $s\prec t$ and $t'$ is real, we call $t$ both \emph{colliding} and \emph{fake}.  All descendants of these colliding particles  are also called fake.  We write \smash{$\mathrm{Real}(\sT^\pm_{[m,n]}(v))$} and \smash{$\mathrm{Fake}(\sT^\pm_{[m,n]}(v))$} for the sets of real and fake particles, respectively, and we denote the set of colliding particles with position in \smash{$I^\pm_i$} by \smash{$\mathrm{Colliding}(\sT^\pm_{[m,n]}(v),I_i^\pm)$}. Note that \smash{$\mathrm{Colliding}(\sT^\pm_{[m,n]}(v),I^\pm_i)\subseteq \mathrm{Fake}(\sT^\pm_{[m,n]}(v))$} for every $i$. 
\end{definition}
Ancestral lines of real particles provide an approximation for the edges encountered in a breadth-first exploration of the cluster of a vertex $v$, after projecting particle positions via $\ell$, see Remark \ref{rem:intuition} below.
The following proposition makes the coupling between the $\gamma$-growing random graph and the real particles in the $\gamma$-branching random walk precise. We set up some final notation.
Let $\sC_{[m,n]}(v)$ denote the connected component containing $v\in[m,n]$ in the induced subgraph on the vertices $[m,n]:=\{m,\ldots, n\}$, denoted by $\sG_{[m,n]}$. Moreover, let $d_{[m,n]}(u,v)$ denote the graph distance between $u$, and $v$, i.e., the least number of edges in $\sG_{[m,n]}$ on a path between $u$ and $v$.

\begin{proposition}[Coupling the branching random walk with a component]\label{prop:domination-upper}
Let $\beta\in(0,1/2]$ and $\gamma\in[0,1)$. Consider the  $\gamma$-growing random graph from Definition~\ref{def:pa-ann}, and the $\gamma$-branching random walk from Definition~\ref{def:pa-brw}. Assume first that $m, v, n\in\N$ such that $m\le v\le n$. There exists a coupling of $\sT^-_{[m,n]}(v)$ and $\sC_{[m,n]}(v)$ such that on this coupling, for all $k\in\N$,
\begin{equation}\label{eq:coupling-lower}
\big\{\ell^-(X_s): s\in\mathrm{Real}\big(\sT^-_{[m,n]}(v)\big), |s|\le k\big\}\, \subseteq\, 
			\{j\in \sC_{[m,n]}(v): d_{[m,n]}(v, j)\le k\}.
\end{equation}
Assume next that $v\le n$, and $m\in[1+\mathbbm1_{v\ge2}, v]$.
There exists a coupling of $\sT^+_{[m,n]}(v)$ and $\sC_{[m,n]}(v)$ such that on this coupling, for all $k\in\N_0$,	
	\begin{equation}\label{eq:coupling-upper}
			\{j\in \sC_{[m,n]}(v): d_{[m,n]}(v, j)\le k\}                                       \, \subseteq\,  \big\{\ell^+( X_s): s\in\mathrm{Real}\big(\sT^+_{[m,n]}(v)\big), |s|\le k\big\}.\end{equation}
\end{proposition}
\begin{proof}
	Fix $m\le v\le n$. We construct a probability space on which we simultaneously sample the killed branching random walk \smash{$\sT_{[m,n]}^\pm(v)$} and the component $\sC_{[m,n]}(v)$. We use the construction both for the lower and upper bound. To do so, we use the following collections of iid uniform random variables: 
    \begin{equation}
        U_{s, j}\sim\mathrm{Unif}[0,1], \quad s\in \sU,\  j\in[m,n], \qquad \text{and}\quad R_{i,j}\sim \mathrm{Unif}[0,1], \quad i,j\in[m,n], \ i<j.
    \end{equation}
     We construct a sample of the branching random walk killed outside $\smash{I_{[m,n]}^\pm}$ iteratively. We start the branching random walk with the root positioned at $\log v$ or $\log(2v-1)$ for lower and upper bound, respectively.
    Let $q_{s,j}^\pm$ be the probability that $s\in\sU$, given its position $X_s$, produces at least one offspring in the interval \smash{$I_{j}^\pm$}. Then we may couple the branching random walk with the uniform random variables such that for all $s\in \sU$ and $j\in[m,n]$,
    \[
    \Ind{\exists i\in\N: X_{si}\in I_j^\pm}=\Ind{U_{s,j}\le q_{s, j}^\pm}.
    \]
    Given the sample of the branching random walk and the uniform random variables, we determine the component of $\sC_{[m,n]}(v)$ via the order $\prec$ of branching random walk particles. We call particles in the branching random walk real and fake following Definition \ref{def:real}. Each interval $\smash{I_i^\pm}$ contains at most one real particle. Assume first that $\smash{I_i^\pm}$ contains a real particle, denoted by $s_i$. 
    Then we couple the edges in $\sG_{[m,n]}$ by setting 
    \begin{equation}\label{eq:coupling-edges}
    \Ind{i\sim j}=\Ind{U_{s_i, j}\le p_{ij}}, \qquad
      \parbox[t]{0.6\linewidth}{if there exists a real particle $s_i$ such that  $X_{s_i}\in I_i^\pm$, \\and there is no real particle $s_j$ such that $X_{s_j}\in I_j^\pm$ with $s_j\prec s_i$.}
    \end{equation}
    Edges $\{i,j\}$ in $\sG_{[m,n]}$ are now uniquely determined provided that $I_i^\pm$ or $I_j^\pm$ contains a real particle in \smash{$\sT^\pm_{[m,n]}$}, as the offspring of the smallest real particle in the ordering decides presence of the edge.
    The only undetermined edges $\{i,j\}$ in $\sG_{[m,n]}$ are those such that  $I_i^\pm\cup I_j^\pm$ contains no real particle. We set, for $i<j$,
    \[
    \Ind{i\sim j}=\Ind{R_{i,j}\le p_{ij}}, \qquad \text{if both $I_i^\pm$ and $I_j^\pm$ contain no real particles}.
    \]
    Note that the described coupling indeed preserves the marginals for the connected component and killed branching random walk. We now prove the lower bound \eqref{eq:coupling-lower}.

    \medskip\noindent\emph{Lower bound.\ }
    On the described coupling, provided that $q_{t,j}^-\le p_{ij}$ for each real particle $t\in I_i$, for all $i, j\in[m,n]$, we have the following implication:  If a particle $s$ in generation $k$ is real, then its ancestral line towards the root has a corresponding path in $\sC_{[m,n]}$ by projecting particle positions on the ancestral path to vertices using $\ell^-(x)=\lfloor\re^x\rfloor$. We show next that $\smash{q_{t,j}^-}\le p_{ij}$ is satisfied for all particles $t$. We recall that $\smash{q_{t,j}^-}$ denotes the probability that a particle $t$ produces at least one offspring in $\smash{I_{j}^-}$. Bounding $p_{ij}\ge1-\exp(-p_{ij})$, it suffices to show that, for $t$ with $X_t\in I_i^-$,
	\begin{equation}\label{lower}
		q_{t,j}^-= 1-\exp\big(-\mu_\beta
		\big( (\log j -X_{t}, \log (j+1) -X_{t}]\big)\big) \le 1-\exp\big(\beta(i\vee j)^{\gamma-1}(i\wedge j)^{-\gamma}\big).
	\end{equation}
	By definition of 
    $\mu_\beta$ in~\eqref{eq:mu_beta}, $q_{t,j}^-$ is maximized when $X_{t}$ is at the leftmost value in $I^-_i=[\log i, \log(i+1))$. Thus, 
	\begin{align*}
		\mu_\beta
		\big( (\log j -X_{t}, & \log (j+1) -X_{t} ]\big)
		\leq \mu_\beta \big( (\log (j/i), \log ((j+1)/i) ]\big).
	\end{align*}
	If $i<j$ and $\gamma\neq 0$, then
	\begin{align*}
		\mu_\beta \big( (\log (j/i), \log ((j+1)/i) ]\big)
		 & = \beta \int_{\log (j/i)}^{\log ((j+1)/i)}
		\re^{\gamma y} \, \rd y
		= \tfrac\beta\gamma \Big(((j+1)/i)^\gamma-(j/i)^\gamma \Big)                                                 \\
		 & =\tfrac\beta\gamma  \big((1+1/j)^\gamma-1 \big) j^\gamma i^{-\gamma} \leq \beta j^{\gamma-1} i^{-\gamma}.
	\end{align*}
    The last bound is easily verified by substituting $x=1/j$, then verifying the bound at $x=0$, and observing that the derivative in $x$ of the left-hand side is at most the derivative of the right-hand side. When $i<j$ and $\gamma=0$, 
    \[
    \mu_\beta \big( (\log (j/i), \log ((j+1)/i) ]\big) = \beta\big(\log((j+1)/i)-\log(j/i)\big)=\beta\log(1+1/j)\le \beta/j.
    \]
    	Similarly, if $i>j$, then for any $\gamma<1$
	\begin{align*}
		\mu_\beta \big( (\log (j/i), \log ((j+1)/i) ]\big)
		 & = \beta \int_{\log (j/i)}^{\log ((j+1)/i)}
		\re^{(1-\gamma) y} \, \rd y
		= \tfrac\beta{1-\gamma} \Big(((j+1)/i)^{1-\gamma}-(j/i)^{1-\gamma} \Big)            \\
		 & =\tfrac\beta{1-\gamma}  \big((1+1/j)^{1-\gamma}-1 \big) j^{1-\gamma}i^{\gamma-1}
		\leq \beta i^{\gamma-1} j^{-\gamma} .
	\end{align*}
    These three cases prove~\eqref{lower} and the lower bound \eqref{eq:coupling-lower} follows.

\pagebreak[3]

\medskip\noindent\emph{Upper bound.\ } We next prove \eqref{eq:coupling-upper} via the same coupling. Recall that in the upper bound embedding the interval $I_1^+$ is the singleton $\{0\}$. Since we consider the component of $v$ in the induced subgraph on $[m,n]$ with $m\ge 1+\mathbbm{1}_{v\ge 2}$, vertex~$1$ does not appear in $\sC_{[m,n]}(v)$ whenever $v\ge 2$. In the remaining case $v=1$, the root is a real particle in $\smash{I_1^+}=\{0\}$, and $1$ is contained in the set on the right-hand side in~\eqref{eq:coupling-upper}. Thus, it suffices to show that 
\begin{equation}\label{eq:coupling-upper2}
\{j\in \sC_{[m,n]}(v): d_{[m,n]}(v, j)\le k, j\ge 2\}                                       \, \subseteq\,  \big\{\ell^+( X_s): s\in\mathrm{Real}\big(\sT^+_{[m,n]}(v)\big), |s|\le k\big\}.
\end{equation}
We show below that $p_{ij}\le \smash{q^+_{t,j}}$ for all real particles $t\in \smash{I^+_i}$ and all $i, j\in[m,n]$ and $j\ge 2$. Under this assumption, we claim that \eqref{eq:coupling-upper2} follows by induction on~$k$. For $k=0$, it is trivial, as both sets consist only of the vertex $v$. Assuming the inclusion holds for $k$, we advance the induction and argue next that it also holds for $k+1$. 
We decompose
\[
\begin{aligned}
\big\{j\in\sC_{[m,n]}(v) &: d_{[m,n]}(v,j)\le k+1\big\}\\ 
&= \big\{j\in\sC_{[m,n]}(v) : d_{[m,n]}(v,j)\le k\big\} \cup \big\{j\in\sC_{[m,n]}(v) : d_{[m,n]}(v,j)= k+1\big\}.
\end{aligned}
\]
By the induction hypothesis, the first set on the right-hand side is a subset of the real particles in the first $k$ generations. So, we have to argue that 
\begin{equation}\label{eq:upper-coupling-inclusion} \big\{\ell^+( X_s): s\in\mathrm{Real}\big(\sT^+_{[m,n]}(v)\big), |s|\le k+1\big\}\supseteq \big\{j\in\sC_{[m,n]}(v) : d_{[m,n]}(v,j)= k+1, j\ge 2\big\}.
\end{equation}
Thus, we consider the event that a breadth-first exploration of $\sC_{[m,n]}(v)$ discovers a vertex $j\in[m,n]$ exactly at generation $k+1$. In particular, we work on the event that $d_{[m,n]}(v,j)>k$. For proving \eqref{eq:upper-coupling-inclusion}, we may assume without loss of generality that there is no real particle $s$ in the first~$k$ generations such that $\ell^+(X_s)=j$. Let 
$$\mathrm{Real}_k(v):=\big\{\ell^+(X_s) : |s|=k, s\in\mathrm{Real}\big(\smash{\sT^+_{[m,n]}}(v)\big)\big\}.$$
By \eqref{eq:coupling-edges}, presence of an edge $\{i,j\}$ is decided  by the offspring of the smallest real particle in $\smash{I_i^+ \cup I_j^+}$ with respect to the ordering $\prec$. There are no real particles in the union of $\big(\smash{I_i^+} : i\in\mathrm{Real}_k(v)\big)$ before generation $k$ by definition of real particles, and $\smash{I_j^+}$ contains no real particle before generation $k+1$ by assumption. Therefore, none of the edges $\{i, j\}$ with $i\in\mathrm{Real}_k$ has been determined before generation $k+1$. These edges are revealed exactly when we expose the offspring of the real particles in generation~$k$.

\smallskip 
If none of the real particles in generation $k$ produces an offspring in $\smash{I^+_j}$, then $U_{s,j}> \smash{q^+_{s,j}}$ for all real particles~$s$ in generation $k$. As we assume that $\smash{q^+_{s,j}}\ge p_{ij}$ for all real particles~$s$ with $X_s\in \smash{I^+_i}$ and all $i, j\in[m,n]$ with $j\ge 2$, all edges between $\mathrm{Real}_k$ and $j$ are absent on this event, and we necessarily have that $d_{[m,n]}(v, j)\ge k+2$ as we work on the event that $d_{[m,n]}(v,j)>k$.

\smallskip 
Therefore, the only way that $d_{[m,n]}(v,j)$ can be discovered by the breadth-first exploration in generation $k+1$ while none of the real particles up to generation $k$ project to $j$, is if there is a smallest real particle $s$ in generation $k$ (with respect to $\prec$) that produces an offspring $si$ with position $X_{si}\in \smash{I^+_{j}}$. The first offspring that $s$ produces in $\smash{I^+_j}$ must be real because $\smash{I^+_j}$ did not contain a real particle before.
The containment in  \eqref{eq:upper-coupling-inclusion} follows provided that $\smash{q^+_{t,j}}\ge p_{ij}$ for all real particles $t\in \smash{I^+_i}$ and all $i, j\in[m,n]$ with $j\ge 2$. Thus,~\eqref{eq:coupling-upper2} and~\eqref{eq:coupling-upper} follow under the same condition.  

\medskip
\noindent\emph{Verification of the one-step inequality $q^+_{t,j}\ge p_{ij}$.}
We finish the proof of \eqref{eq:coupling-upper} by showing the inequality $\smash{q^+_{t,j}}\ge p_{ij}$ for all real particles $t\in \smash{I_i^+}$ and all $i,j\in[m,n]$ with $j\ge 2$.  Since $\smash{q^+_{t,j}}$ is the probability that particle $t$ produces at least one offspring in $\smash{I_j^+}$, we will show that for any $t$ with $\smash{\ell^+}(X_t)=i$, and $j\ge 2$,
	\begin{equation}\label{upper}
		q^+_{t,j} = 1-\exp\Big(-\mu_\beta
		\big( (\log (2j-3) -X_{t}, \log(2j-1) -X_{t}]\big)\Big) \ge \beta(i\vee j)^{\gamma-1}(i\wedge j)^{-\gamma}.
	\end{equation}
    Note that $j\ge 2$ implies that $\log(2j-3)$ is well-defined.
	By the definition of the measure $\mu_\beta$ in~\eqref{eq:mu_beta}, the left-hand side is minimized when $X_{t}$ is at the rightmost value in $I^+_i$, which is $\log (2i-1)$ for all $i\ge 1$ by~\eqref{eq:i-upper}. Thus,
    \begin{align*}
		\mu_\beta
		\big( (\log (2j-3) -X_{t}, & \log(2j-1) -X_{t} ]\big)
		\geq \mu_\beta \Big( (\log \big(\tfrac{2j-3}{2i-1}\big), \log \big(\tfrac{2j-1}{2i-1}\big) \big]\Big).
	\end{align*}
    We distinguish four cases. 

    \smallskip 
    \noindent\emph{Case 1: $1\le i<j$, $\gamma>0$.\ }
	If $1\le i<j$ and $\gamma\neq 0$, then 
	\begin{align}
		\mu_\beta \Big( (\log \big(\tfrac{2j-3}{2i-1}\big), \log \big(\tfrac{2j-1}{2i-1}\big) \big]\Big)
		 & = \beta \int_{\log \frac{2j-3}{2i-1}}^{\log \frac{2j-1}{2i-1}}
		\re^{\gamma y} \, \rd y
		= \tfrac\beta\gamma \Big(\big(\tfrac{2j-1}{2i-1}\big)^\gamma-\big(\tfrac{2j-3}{2i-1}\big)^\gamma \Big)                                                \label{upper-int} \\
		 & =\tfrac\beta\gamma  (2i-1)^{-\gamma}\Big((2j-1)^\gamma-(2j-3)^{\gamma}\Big) \nonumber\\&= \beta(i-\tfrac12 )^{-\gamma}\cdot\tfrac{1}{\gamma}\Big((j-\tfrac12 )^\gamma-(j-\tfrac32)^{\gamma}\Big).\nonumber
	\end{align}
    We use that $(x^a-(x-1)^a)/a\ge x^{a-1}$ for $x>1$ and $a\in(0,1)$, and that $1-\exp(-x)\ge x-x^2/2$ for $x>0$. This yields,
    \begin{align}
    1-\exp\Big(-\mu_\beta
		&\big( (\log (2j-3) -X_{s}, \log(2j-1) -X_{s}]\big)\Big)\nonumber\\
        &\ge
        1-\exp\Big(\beta(i-\tfrac12 )^{-\gamma}(j-\tfrac12 )^{\gamma-1}\Big) \label{eq:upper-12}\\
        &\ge \beta(i-\tfrac12 )^{-\gamma}(j-\tfrac12 )^{\gamma-1} - \tfrac{\beta^2}{2}(i-\tfrac12 )^{-2\gamma}(j-\tfrac12 )^{2(\gamma-1)}.\nonumber
        \end{align}
    Recall from \eqref{upper} that we aim to show that the right-hand side is at least $\beta i^{-\gamma}j^{\gamma-1}$. Rearranging terms, we obtain that \eqref{upper} is implied if 
    \begin{equation}\nonumber
    \begin{aligned}
    \tfrac{\beta}{2}&\le (i-\tfrac12 )^{\gamma}(j-\tfrac12 )^{1-\gamma}-i^{-\gamma}j^{\gamma-1}(i-\tfrac12 )^{2\gamma}(j-\tfrac12 )^{2(1-\gamma)} \\
    &= \big(\tfrac{i-\frac12 }{i}\big)^{\gamma}\big(\tfrac{j-\frac12 }{j}\big)^{1-\gamma}\Big(i^\gamma j^{1-\gamma} - (i-\tfrac12 )^{\gamma}(j-\tfrac12 )^{1-\gamma}\Big).
    \end{aligned}
    \end{equation}
    The first two factors on the right-hand side are increasing in $i$ and $j$ as $\gamma\in[0,1)$. Their product is at least $\tfrac12 $ as $j>i\ge 1$. Since we assume $\beta\le \tfrac12 $, it suffices to show that 
    \begin{equation}\label{eq:1ijgamma}
    \tfrac12 \le i^\gamma j^{1-\gamma}- (i-\tfrac12 )^{\gamma} (j-\tfrac12 )^{1-\gamma}.
    \end{equation}
    To this end, consider the function 
    $$
    f(x,y) = x^\gamma y^{1-\gamma} - (x-\tfrac12 )^\gamma(y-\tfrac12 )^{1-\gamma}, 
    \text{ for $1\le x\le y$.}  $$
    Its derivative with respect to $y$ is
    \[
    \frac{\partial}{\partial y}f(x,y) = (1-\gamma)\left(\big(\tfrac x y\big)^\gamma - \big(\tfrac{x-\frac12 }{y-\frac12 }\big)^{\gamma}\right) \ge 0,   
    \]
    since $x/y \ge (x-\tfrac12 )/(y-\tfrac12 )$ for $y\ge x\ge 1$. Moreover, we have that
    $
    f(x,x) = x - (x-\tfrac12 ) = \tfrac12 
    $
    for every $x\ge 1$.
    It follows that $f(x,y) \ge \tfrac12 $ for all $y\ge x\ge 1$, and therefore $f(i,j) \ge \tfrac{1}{2}$ for all $1\le i<j$. This proves~\eqref{eq:1ijgamma} when $1\le i<j$, $\gamma>0$.

    \smallskip 
    \noindent\emph{Case 2: $i\!>\!j\ge\! 2$, $\gamma>0$}. Verbatim reasoning applies as in Case 1, with the exponent  $\gamma$ replaced by $1\!-\!\gamma$.

    \smallskip 
    \noindent\emph{Case 3: $j>i\ge 1$, $\gamma=0$}. Recalling \eqref{upper-int}, we have in this case, using $\log(1+1/x)\ge 1/(1+x)$ for $x>0$,
    \begin{align*}
    1-\exp\Big(-\mu_\beta \Big( (\log \big(\tfrac{2j-3}{2i-1}\big), \log \big(\tfrac{2j-1}{2i-1}\big) \big]\Big)\Big)
    &= 1-\exp\Big(-\beta\log\big(\tfrac{2j-1}{2j-3}\big)\Big) \\
    &= 1-\exp\Big(-\beta\log\big(1+\tfrac{1}{j-\frac32}\big)\Big) 
    \ge 1-\exp\Big(-\tfrac{\beta}{j-\frac12 }\Big).
    \end{align*}
    Following the same reasoning as in Case 1 from \eqref{eq:upper-12}, the right-hand side is at least $\beta/j$.
    
    \smallskip 
    \noindent\emph{Case 4: $i>j\ge 2$, $\gamma=0$.}
    We use a similar strategy. We compute the integral as in \eqref{upper-int}, with the exponent $\gamma$ replaced by $1-\gamma=1$. We obtain
    \begin{align*}
    1-\exp\Big(-\mu_\beta \Big( (\log \big(\tfrac{2j-3}{2i-1}\big), \log \big(\tfrac{2j-1}{2i-1}\big) \big]\Big)\Big)
    &= 1-\exp\Big(-\beta \big(\tfrac{2j-1}{2i-1}-\tfrac{2j-3}{2i-1}\big)\Big) = 1-\exp\big(-\tfrac{\beta}{i-\frac12 }\big).
    \end{align*}
    Analogously to Case (1) from \eqref{eq:upper-12}, it follows that the right-hand side is at least $\beta/i$ for $i\ge 2$ and $\beta\in(0, \tfrac12 ]$.    This proves~\eqref{upper} for all values $i,j\in[m,n]$ with $j\ge 2$, and finishes the proof of \eqref{eq:coupling-upper}. 
\end{proof}
\begin{remark}[Precision of the coupling]
For the purposes of this paper, a simpler coupling would have sufficed. For instance, in the lower bound we could declare a particle fake if it has an ancestor that has position in an interval with another particle, irrespective of whether that particle is real or fake or where it appears in the ordering $\prec$. For the upper bound, we could dominate the component by projecting all particles in the killed branching random walk, without distinguishing real or fake particles. 
We formulated this more refined coupling above because it yields a sharper and more transparent correspondence between the branching random walk and a component in the graph. This finer structure may be useful in future work requiring even more precise control of the $\gamma$-growing random graph.    
\end{remark}

\begin{remark}[Relation between ancestral lines and edges]\label{rem:intuition}Ancestral lines towards real particles approximate the edges in a component. This approximation is less accurate for particles with small position, corresponding to early vertices, as the absolute difference between the left-hand side and the right-hand side in \eqref{lower} and \eqref{upper} can be largest for small values of $i$ and~$j$.

\pagebreak[3]

\smallskip 
Some colliding particles have a corresponding edge in $\sC_n(v)$, but not every collision does. Suppose a particle $t$ with real parent $t'$ with position $X_{t'}\in I_i$ collides with a real particle $s\prec t$ with $X_s\in I_j$. By~\eqref{eq:coupling-edges},  this collision only corresponds to an edge $\{i,j\}$ if the edge has not been determined before, which only occurs when  $|s|=|t|$, or when $|t'|=|s|$ and $X_{t'}<X_s$.  Collisions outside these two configurations are artifacts of the branching random walk and have no corresponding edge.  In the lower-bound construction $\smash{\sT^-_{[1,n]}}$, every collision satisfying one of these configurations yields an edge in $\sC_n(v)$ that closes a cycle. In the upper-bound construction $\smash{\sT^+_{[1,n]}}$ the ancestral lines of the branching random walk form supersets of the edges in $\sC_n(v)$, and each cycle has an edge corresponding to such a collision, but not every such collision corresponds to an edge closing a cycle. 
\end{remark}
\subsection{Local limit is the trace of the $\gamma$-BRW}\label{sec:local}
The construction of the killed $\gamma$-BRW in Section~\ref{sec:bfe} is inspired by the  local limit of inhomogeneous random graphs of preferential-attachment type in \cite{morters2023tangent},  and constructed for genuine preferential attachment models in \cite{BBCSLocalLimitPA,DerMor13, garavaglia2022universality}. We are not aware of previous occurrences for uniformly grown random graphs ($\gamma=0$). Informally, the local limit describes the limiting graph structure around a typical vertex up to a fixed (but arbitrarily large) number of generations $r$ as $n\to\infty$. As we see below, the local limit of $\gamma$-growing random graphs attachment is the \emph{trace} of a killed branching random walk on the negative half line. Here, the trace is the graph obtained from a branching random walk by viewing particles as vertices, and including an edge between vertices representing the particles of a parent and its offspring in the branching random walk. 
We give a brief description of the local limit here, as it helps in proving Theorems~\ref{thm:typtail} and~\ref{thm:susceptibility}. We refer to \cite{vdH2} for a more elaborate discussion. We start with some general definitions.
\smallskip 

\pagebreak[3]
A rooted graph is a couple $(G, o)$ of a graph $G=(V,E)$ and some, possibly random, distinguished vertex $o\in V$, which we call the root of $G$.    We call two rooted graphs $(G_1, o_1)$ and $(G_2, o_2)$ isomorphic, i.e., $(G_1, o_1)\simeq(G_2, o_2)$, if there exists a bijection $\phi$ 
from the connected component of $o_1$ in $\mathscr G_1$
to the connected component of $o_2$ in $\mathscr G_2$
such that $\phi(o_1)=o_2$, and that $\{u,v\}$ is an edge in $(G_1, o_1)$ if and only if $\{\phi(u),\phi(v)\}$ is an edge in $(G_2, o_2)$. Let $\mathfrak{G}$ be the space of isomorphism classes of rooted locally finite graphs. We write $B_G(v, r)$ for the subgraph of $G$ induced on  vertices at graph distance at most $r$ from $v\in V$. Define
\[
\begin{aligned}
R\big((G_1, o_1), (G_2, o_2)\big)&:=\max\big\{r\in\N: B_{G_1}(o_1, r)\simeq B_{G_2}(o_2, r)\big\}, 
\\
d_{\mathfrak G}\big((G_1, o_1), (G_2, o_2)\big)&:=1/\big(1+R\big((G_1, o_1), (G_2, o_2)\big)\big).
\end{aligned}
\]   
Then, $(\mathfrak G, d_{\mathfrak G})$ is a Polish space. 
We call a finite rooted graph $(G, O)$ uniformly rooted, if $O$ is chosen uniformly at random among the vertices of $G$.
We say that a sequence of uniformly rooted graphs $(G_n, O_n)_{n \ge 1}$ converges locally in probability to $(G_{\infty}, O)$ having law $\mu$, if for every bounded and continuous function $h\colon\mathfrak G\to\R$,
\begin{equation}\label{eq:local-l1}
    \frac1{|V_n|}\sum_{v\in V_n} h(G_n, v) 
    \overset\Prob\longrightarrow\E_\mu[h(G_\infty, O)],\qquad \text{as }n\to\infty,
    \end{equation} 
where $V_n$ is the vertex set of $G_n$. Note that the left hand side is the expectation of $h(G_n,O_n)$ over the uniform root $O_n$ for the fixed graph $G_n$, and thus is a random variable if $G_n$ itself is a random graph.

    \begin{theorem}[Local limit]\label{thm:local}
        Consider the $\gamma$-growing random graph with fixed $\gamma\in[0, 1)$. Let $\beta_n$ be a sequence tending to $\beta$. Then $(\sG_n, O_n)$ and $(\sG_n^\mathrm{poi}, O_n)$ converge locally in probability to the trace of the $\gamma$-BRW with parameters $\gamma$ and $\beta$ with the root at the origin, with particles killed upon entering the interval $(X, \infty)$ for an independent exponential random variable $X$.
    \end{theorem}
    The theorem shows that the $n$-dependence of $\beta$ vanishes in the local limit. A proof of the theorem was given by the third author in \cite{morters2023tangent} for $\beta$ not depending on $n$ and the graph $\sG_n$. We explain the heuristics of the limit, and how to extend the proof to $\beta=\beta_n\to\beta_c$, and  leave the details to the reader.
\begin{proof}[Proof sketch]
To prove local convergence in probability, by \cite[Theorem 2.15]{vdH2} it suffices to prove that for any fixed rooted graph $H$ and $r$
\[
\frac{1}{|V_n|}\sum_{v\in V_n}\ind{B_{G_n}(v_n, r)\simeq H} \overset{\Prob}\longrightarrow \Prob\big(B_G(o, r)\simeq H\big).
\]
Such proofs proceed via a first and second moment method. The second moment is discussed in detail in \cite[Section 3]{vdH2} for more general inhomogeneous random graphs. We discuss only the first moment, corresponding to the distribution of the $r$-neighbourhood of a uniform vertex up to graph distance $r$.\smallskip

This distribution is understood via a similar exploration as given in Section \ref{sec:bfe}, except that particles have now position on $(-\infty, X]$, rather than on $I_{[1,n]}\sim[0, \log n]$. We start from the same partition of the positive half line as in Section \ref{sec:bfe} for large $n$. To sample a cell uniformly at random (corresponding to taking a uniform root) from the right killing boundary, the exponential speed-up from Section \ref{sec:bfe} needs to be corrected for. This corresponds to sampling the root at exponential distance $X$ (conditioned to be at most $\log n$) from the right-killing boundary, which is at position approximately $\log n$. We then translate the branching random walk by $-\log n + X$ to put the root at the origin.
Taking a limit of $n\to\infty$, the left killing barrier drifts to $-\infty$, and hence in the local limit there is no killing of particles to the left. Only particles on $(X,\infty)$ are killed, which corresponds to ``removing edges to vertices that arrive in the future'' in the pre-limit. The interval occupied by the particles of the first $r$ generations is tight. In this tight set, cells become infinitesimally small in the limit, so no collisions can happen in the limit, and we do not need a distinction between real and fake particles. \smallskip

This construction allows to determine the local limit for fixed $\beta$. To show that the local limit of the graph with a sequence $\beta_n$ tending to some $\beta>0$ does not depend on $\beta_n$, a simple coupling argument works by taking the canonical coupling of the breadth-first explorations with $\beta_n$ and $\beta$: for any fixed number of generations $r$, we can choose $n$ sufficiently large such that the explorations agree up to generation $r$ with probability arbitrarily close to one.     
\end{proof}

    \begin{corollary}[Truncated component sizes]\label{cor:truncated}
    Consider the $\gamma$-growing random graph with fixed $\gamma\in[0, 1)$. Let $\beta_n$ be a sequence tending to $\beta$. For every $k\in\N$, as $n\to\infty$,
        \begin{align}
            \frac{1}{n}\sum_{v=1}^n\ind{|\sC_n(v)|\ge k}&\overset{\Prob, L^1}{\longrightarrow} \Prob^{\beta}_{0}\big(T_{(-\infty, X]}\ge k\big),\label{eq:local-csd} \\
            \frac{1}{n}\sum_{v=1}^n|\sC_n(v)|\ind{|\sC_n(v)|\le k}&\overset{\Prob, L^1}{\longrightarrow} \E^{\beta}_{0}\big[T_{(-\infty, X]}\ind{T_{(-\infty, X]}\le k}\big],\label{eq:local-sus}\\
            \frac{1}{n}\sum_{v=1}^n|\sC_n(v)|\log|\sC_n(v)|\ind{|\sC_n(v)|\le k}&\overset{\Prob, L^1}{\longrightarrow} \E^{\beta}_{0}\big[T_{(-\infty, X]}\log T_{(-\infty, X]}\ind{T_{(-\infty, X]}\le k}\big].\label{eq:local-logsus}
        \end{align}
    \end{corollary}
    \begin{proof}
        The three functions in the sum are bounded and depend only on the $k$-neighbourhood of the graph rooted at $v$, implying that they are continuous. The limits in probability follow by Theorem \ref{thm:local} and \eqref{eq:local-l1}. The convergence holds also in $L^1$ by boundedness of the averages on the left-hand sides.
    \end{proof}
    We explain why this corollary is central for Theorems~\ref{thm:typtail} and~\ref{thm:susceptibility}. For the tail of the typical component sizes in Theorem~\ref{thm:typtail}, the corollary reduces it to studying the tail of the total progeny of the local limit, and ---crucially--- collisions do not occur in the local limit, which simplifies the proof. For proving the second limit Theorem~\ref{thm:susceptibility},  we  truncate component sizes at level $k$, use \eqref{eq:local-logsus}, and then let $k\to\infty$. The fact that this latter limit tends to infinity will follow from our analysis of the decay of the right-hand side in \eqref{eq:local-csd}.  To obtain a lower bound for the  first limit in Theorem \ref{thm:susceptibility}, we can argue similarly. An upper bound does not naively follow from \eqref{eq:local-sus}: We have to argue that the contribution coming from components larger than $k$ do not contribute significantly. Local convergence does not provide the tools for this. To handle that tail, we use the coupling with the branching random walk where $\beta$ may depend on $n$, and thus is not necessarily equal to $\beta_c$. Our bounds there will crucially use that  $\limsup_{n\to\infty} 4\beta_c(\beta_n-\beta_c)(\log n)^2<\pi^2$.

\section{Preliminaries: Many-to-few formulas} \label{sec:many-to-few}
This section collects two key preliminaries for analysing  the $\gamma$-branching random walk in the finite-$n$ pre-limit, where $\beta$ may depend on~$n$, and the critical local limit with $\beta=\beta_c$.  
We extend the classical \emph{many-to-one} and \emph{many-to-two} formulas, which usually express expectations over particles in a fixed generation or along a stopping line (see Biggins--Kyprianou~\cite{Biggins2004} or Shi~\cite{ShiLectureNotes}), to identities summing over \emph{all} particles in the tree.  These identities relate such global sums to expectations taken along a single distinguished \emph{spine} of the process, and reduces the analysis of the branching system to that of a non-branching random walk.  
The multiplicative factors $(\beta/\beta_c)^n$  appearing in the forthcoming formulas originate from the Poissonian displacement distribution~\eqref{eq:mu_beta}, changing $\beta$ rescales the offspring intensity by~$\beta/\beta_c$. We now introduce the associated random walk.

\smallskip 
Recall that a Laplace-distributed random variable $Y$ with parameter $a$ has distribution 
\[
\Prob(Y\le y) = \frac{1}{2a}\int_{-\infty}^y\re^{-|x|/a}\, \rd x,
\]
and variance  $2a^2$.
Under $\Prob_x$, denote by $(S_n)_{n\ge 1}$ the random walk started from $x$ with Laplace(1)-distributed steps, we simply call this a Laplacian random walk. Note that such a random walk can be constructed from a Brownian motion by stopping it at the times of a Poisson process (see e.g.~\cite[Formula 1.1.0.5]{HandbookBM}), but we do not use this fact in this paper.
%
The many-to-one formula that we state next is the main tool for computing moments of the total progeny.
\begin{lemma}[Many-to-one formula]\label{lem:many-to-one} Consider the $\gamma$-branching random walk with parameter $\gamma\in[0, 1/2)$. Let \smash{$(S_n)_{n\ge 1}$} be a Laplacian random walk with parameter one. Let $F\colon\bigcup_{n=0}^\infty {\R}^n\to [0,\infty)$ be measurable.
Then we have that
	\begin{equation}\label{eq:many-to-one}
		\E_x^\beta\left[\sum_{t\in \sU} F((X_s)_{s\le t})\right] = \re^{- x/2}\E_{ 2\beta_c x}\left[\sum_{n=0}^\infty (\beta/\beta_c)^nF\big(S_0/(2\beta_c),\ldots,S_n/(2\beta_c)\big)\re^{ S_n/(4\beta_c)}\right].
	\end{equation}
\end{lemma}

\begin{proof}
Recall that the offspring distribution of the branching random walk is the Poisson process with intensity $\mu_\beta$ defined in \eqref{eq:mu_beta}. Hence, for every non-negative, measurable function $f$, we have
\begin{align*}
	\label{eq:many-to-one-1}
	\E_x^\beta\left[\sum_{i=1}^\infty f(X_i)\right] 
    = \int f(x+y) \mu_\beta(\rd y) 
    &= \beta \int f(x+y)\re^{y/2} \re^{-2\beta_c|y|}\,\rd y\\
    &= (\beta/\beta_c) \re^{-x/2}\int f(z)\re^{z/2} \big(\beta_c\re^{-2\beta_c|z-x|}\big)\,\rd z.
\end{align*}
Let $S_1^{(\gamma)}=x+X_1^{(\gamma)}$, with $X_1^{(\gamma)}$ Laplace distributed with parameter $1/(2\beta_c)$. This gives
\begin{equation}
	\label{eq:many-to-one-1}
	\E_x^\beta\left[\sum_{i=1}^\infty f(X_i)\right] = (\beta/\beta_c)\re^{-x/2} \E_x^{\beta_c}\big[f(S_1^{(\gamma)})\re^{S_1^{(\gamma)}/2}\big].
\end{equation}
    We next rewrite the expectation on the right-hand side in~\eqref{eq:many-to-one} in terms of a Laplacian random walk with parameter $1/(2\beta_c)$. Let $X_i\sim\mathrm{Laplace}(1)$, $S_n=x_0+X_1+\ldots+X_n$, and \smash{$X_i^{_{(\gamma)}}\sim\mathrm{Laplace}(1/(2\beta_c))$}.
        If $Y\sim\mathrm{Laplace}(1)$, then $aY\sim\mathrm{Laplace}(a)$. Taking $x_0=2\beta_cx$,
        \[
        \frac{1}{2\beta_c}S_n = \frac{2\beta_cx}{2\beta_c} + \sum_{i=1}^n\frac{1}{2\beta_c}X_i \overset{d}= x + \sum_{i=1}^n X_i^{(\gamma)} =: S_n^{(\gamma)}.
        \]
	By linearity of expectation, it now suffices to prove that for every $n\ge 0$ and every  measurable $F\colon {\R}^n\to [0,\infty)$, we have 
	\begin{equation}
		\label{eq:induction_to_show}
        \begin{aligned}
		\E_x^\beta\left[\sum_{|t|=n} F((X_v)_{s\le t})\right] &= \re^{-x/2}\E_x\left[(\beta/\beta_c)^nF(S_0^{(\gamma)},\ldots,S_n^{(\gamma)})\re^{S_n^{(\gamma)}/2}\right] \\ &=\re^{- x/2}\E_{ 2\beta_c x}\left[(\beta/\beta_c)^nF\big(S_0/(2\beta_c),\ldots,S_n/(2\beta_c)\big)\re^{ S_n/(4\beta_c)}\right].
        \end{aligned}
	\end{equation}
	We prove this by induction on $n$. The case $n=0$ is trivial. Assume the equality holds for some $n$ and all $F$ as above. We aim to prove it for $n+1$. We decompose according to the particles at the first generation. Define the function $F_x(x_1,\ldots,x_n) \coloneqq F(x,x_1,\ldots,x_n)$. We have,
	\begin{align*}
		 \E_x^\beta&\Bigg[\sum_{|t|=n+1} F((X_s)_{s\le t})\Bigg]        
		 = \sum_{i=1}^\infty \E_x^\beta\Bigg[ \sum_{|t|=n} F_x((X_{is})_{s\le t})\Bigg]                                                                                                                    \\
		 & = \sum_{i=1}^\infty \E_x^\beta\Bigg[ \E_{X_i}^\beta\Bigg[\sum_{|t|=n} F_x((X_s)_{s\le t})\Bigg]\Bigg]                                                 &  & \text{(branching property)}                 \\
		 & = \sum_{i=1}^\infty \E_x^\beta\left[ \re^{-X_i/2} \E_{X_i}\left[(\beta/\beta_c)^nF_x(S_0^{(\gamma)},\ldots,S_n^{(\gamma)})\re^{S_n^{(\gamma)}/2}\right]\right] &  & \text{(induction hypothesis)}               \\
		 & = \re^{-x/2} (\beta/\beta_c) \E_x\left[ \E_{\tilde S^{(\gamma)}_1}\left[(\beta/\beta_c)^nF_x(S_0^{(\gamma)},\ldots,S_n^{(\gamma)})\re^{S_n^{(\gamma)}/2}\right]\right]               &  & \text{(Equation \eqref{eq:many-to-one-1})},
	\end{align*}
	where $\tilde S^{(\gamma)}_1$ is distributed like $S_1^{(\gamma)}$ under $\Prob_x$. By the Markov property of the random walk, we get,
	\begin{align*}
		\E_x^\beta\left[\sum_{|t|=n+1} F((X_s)_{s\le t})\right]
		 & = \re^{-x/2} (\beta/\beta_c) \E_x\left[ (\beta/\beta_c)^nF_x(S_1^{(\gamma)},\ldots,S_{n+1}^{(\gamma)})\re^{S_{n+1}^{(\gamma)}/2}\right] \\
		 & = \re^{-x/2} \E_x\left[ (\beta/\beta_c)^{n+1}F(S_0^{(\gamma)},\ldots,S_{n+1}^{(\gamma)})\re^{S_{n+1}^{(\gamma)}/2}\right].
	\end{align*}
	This proves \eqref{eq:induction_to_show} with $n+1$ in place of $n$ and completes the induction step. The lemma follows.
\end{proof}

We proceed to the many-to-two formula, which allows to derive second moments of the total progeny. It will also be used to control the number of collisions.

\begin{lemma}[Many-to-two formula]
	\label{lem:many-to-two}
    Consider the $\gamma$-branching random walk with parameter $\gamma\in[0,1/2)$.
	For $F\colon\bigcup_{n=0}^\infty {\R}^n\to [0,\infty)$ a measurable function, define the operators $M$ and $M^*$  by
	\begin{align*}
		MF(x_0,\ldots,x_n) &:= \E_{x_n}^\beta\left[\sum_{t\in \sU} F(x_0,\ldots,x_{n-1},(X_s)_{s\le t})\right]\\
        M^*F(x_0,\ldots,x_n) &:= \E_{x_n}^\beta\left[\sum_{t\in \sU,\,t\ne \emptyset} F(x_0,\ldots,x_{n-1},(X_v)_{s\le t})\right]\\
        &= MF(x_0,\ldots,x_n) - F(x_0,\ldots,x_n).
	\end{align*}
	Then,
	\[
		\E_x^\beta\bigg[\bigg(\sum_{t\in \sU}F((X_r)_{r\le t})\bigg)^2\bigg] = M(F^2)(x) + M(M^*F)^2(x).
	\]
\end{lemma}

\begin{proof}
Define the operator $Q$ by 
\[
QF(x_0,\ldots,x_n) := \E_{x_n}^\beta\left[\sum_{i,j\ge 1,\,i\ne j} F(x_0,\ldots,x_n,X_i)F(x_0,\ldots,x_n,X_j)\right]. 
\]
We have, 
        \begin{align*}
    \E_x^\beta\bigg[ & \bigg(\sum_{t\in \sU}F((X_r)_{r\le t})\bigg)^2\bigg]
    = \E_x^\beta\Bigg[\sum_{s, t \in \sU}F((X_r)_{r\le t})F((X_r)_{r\le s})\Bigg]\\
    &= \E_x^\beta\Bigg[\sum_{t\in \sU}  F((X_r)_{r\le w})^2\Bigg] + \E_x^\beta\Bigg[\sum_{s, t\in \sU,\,t\ne s}F((X_r)_{r\le t})F((X_r)_{r\le s})\Bigg].
    \end{align*}
    The first term equals $M(F^2)(x)$. For the second term we obtain, by decomposing according to the most recent common ancestor of $t$ and $s$,  
    \begin{align*}
    \E_x^\beta\Bigg[\sum_{s, t\in \sU,\,t\ne s}F((X_r)_{r\le t})F((X_r)_{r\le s})\Bigg]&= \E_x^\beta\Bigg[\sum_{w\in \sU}\sum_{\substack{i,j=1\\i\neq j}}^\infty\sum_{\substack{s, t\in \sU}}  F((X_r)_{r\le wit})F((X_r)_{r\le wjs})\Bigg] \\
    &= \E_x^\beta\Bigg[\sum_{w\in \sU}\sum_{\substack{i,j=1\\i\neq j}}^\infty  MF((X_r)_{r\le wi})MF((X_r)_{r\le wj})\Bigg]\\
    &= \E_x^\beta\Bigg[\sum_{w\in \sU}QMF((X_r)_{r\le w})\Bigg]= MQMF(x).
    \end{align*}
    In total, we have
    \begin{equation}
        \label{eq:almost_there}
    \E_x^\beta\bigg[\bigg(\sum_{t\in \sU}F((X_r)_{r\le t})\bigg)^2\bigg] = M(F^2)(x) + MQMF(x).
    \end{equation}
    To complete the proof, recall that the offspring distribution of the $\gamma$-branching random walk is a Poisson process with intensity measure $\mu_\beta$ defined in \eqref{eq:mu_beta}. Letting \smash{$\mu^{[2]}_\beta$} denote its second factorial moment measure (see e.g.~\cite[Sections 5 and 6]{DV2003}), we have \smash{$\mu^{[2]}_\beta = \mu_\beta\otimes \mu_\beta$} by \cite[Proposition~6.3.III]{DV2003}, and therefore,
    \begin{align*}
    QF(x_0,\ldots,x_n) 
    &= \int F(x_0,\ldots,x_n,x)F(x_0,\ldots,x_n,y)\,\mu^{[2]}_\beta(dx,dy)\\
    &= \left(\int F(x_0,\ldots,x_n,x)\,\mu_\beta(dx)\right)^2= \left(M^0 F(x_0,\ldots,x_n)\right)^2,
    \end{align*}
    where we set
    \[
    M^0 F(x_0,\ldots,x_n) = \int F(x_0,\ldots,x_n,x)\,\mu_\beta(dx) = \E_x^\beta\bigg[\sum_{i=1}^\infty F(x_0,\ldots,x_n,X_i)\bigg].
    \]
    Note that $M^0MF = M^* F$, whence $QMF = (M^*F)^2$. 
    Together with \eqref{eq:almost_there}, this proves the lemma.
\end{proof}

\section{Progeny of the branching random walk killed from two sides}\label{sec:progeny}

This section gives uniform moment bounds for the total progeny of the $\gamma$-branching random walk (BRW) killed outside large intervals. These estimates are the core input in the proofs of the main theorems: They show that the progeny is concentrated around its mean and they explain the threshold $\pi^2$ in our statements. 
It arises from a hitting time equation for the killed Laplacian random walk (see Lemma \ref{lem:pgf} below) and marks where exponential moments of the hitting time ---and consequently the BRW progeny--- blow up.
Throughout this section, we 
study the $\gamma$-BRW by itself, postponing
the connection to the graph of size~$n$ to Section~\ref{sec:real}.
We adopt the following notation: 
\begin{itemize}[leftmargin=1em]
    \item[-] \emph{Global scale $L$.\ }The parameter $L$ should be thought of encoding the graph size via $L\approx\log n$. In our many-to-one lemma, 
    {we scaled the jump distribution of the random walk describing the spine such that it becomes independent of $\gamma$. As a result, it lives on an interval of $\gamma$-dependent length $2\beta_cL$, which should be thought of as $2\beta_c\log n$. Nevertheless, we keep using $L$ for the length of intervals.}%
    \smallskip
    \item[-] \emph{Deviation from criticality.\ } We write $\rho_L:=(\beta_n-\beta_c)/\beta_c$ for the normalized deviation from criticality. The moment estimates below hold for sequences $\beta_L\to\beta_c$  (equivalently $\rho_L\to0$) satisfying the constraint \[\limsup_{L\to\infty}\rho_L(2\beta_c L)^2\approx\limsup_{n\to\infty}\frac{\beta_n-\beta_c}{\beta_c}(2\beta_c\log n)^2=\limsup_{n\to\infty}4\beta_c(\beta_n-\beta_c)(\log n)^2<\pi^2.\]
    \item[-] \emph{Interior cutoff $K$.\ }In addition to the global scale $L$, we use a smaller cut-off $K\in[L_0, L]$. Restricting the progeny to $[0, K]$ isolates the contribution of particles that stay within a reduced region, corresponding to the young vertices near the right end of the interval $[0, \log n]$. Allowing $K$ to grow with $L$ but remain strictly smaller is crucial for the lower bound on the largest component and for an upper bound on the susceptibility.  We assume $K$ is at least a large constant, so that asymptotics in our progeny moments can kick in.
\end{itemize}   
With these conventions we  define the auxiliary functions $\Si$ and $\Co$,  extending the definition of $\Si$ from~\eqref{eq:CS}, and introducing its counterpart:
\begin{equation}\label{eq:CSx}
    \Si(\alpha, y):=\begin{dcases}
        \frac{\sin(\sqrt{\alpha}y)}{\sqrt{\alpha}},&\text{if }\alpha>0,\\
        y,&\text{if }\alpha = 0,\\
        \frac{\sinh(\sqrt{|\alpha|}y)}{\sqrt{|\alpha|}},&\text{if }\alpha<0;
    \end{dcases}\qquad \Co(\alpha, y):=\begin{dcases}
        \cos(\sqrt{\alpha}y),&\text{if }\alpha>0,\\
        1,&\text{if }\alpha = 0,\\
        \cosh(\sqrt{|\alpha|}y),&\text{if }\alpha<0.
    \end{dcases}
\end{equation}
For fixed $y\in[0,1]$ both functions are analytic in $\alpha$ and on the interval $(0,\pi^2)$
the function
$\alpha\mapsto\Si(\alpha, y)$ is decreasing and positive. As a function on the interval $[0,1]$ the function $y\mapsto \Si(\alpha, y)$ is convex for $\alpha<0$ and concave for $\alpha>0$. For $\alpha\le(\pi/2)^2$,  it 
is increasing. 
The functions $\Si$ and $\Co$ are the eigenfunctions that appear when solving the Fredholm integral equation / second–order ordinary differential equation  associated to the killed Laplace random walk, see Lemma \ref{lem:exponential-moments1} below. 

\smallskip 
The next lemma provides upper bounds on $y\mapsto\Si(\alpha, y)$  independent of $\alpha$. The function can be bounded by a linear function in $y$ when $\alpha>0$, which is the relevant case for the $\gamma$-BRW just above criticality. When $\alpha<0$ the increase is exponential for large $y$, but still linear for $y$ close to the origin.

\begin{lemma}\label{lem:bound-si}
    For any $L\ge 0$, $\rho\in\R$, and  $x\ge 0$, 
    \[
    \Si(\rho L^2, x/L)\le \frac{x}{L}\Big(2+ \ind{\rho<0}\re^{\sqrt{|\rho|}x}\Big)\le \frac{3x}{L}\re^{\sqrt{|\rho|}x}, \qquad \Co(\rho L^2, x/L)\le \ind{\rho\ge 0}+\ind{\rho<0}\re^{\sqrt{|\rho|}x}.
    \]
    \begin{proof}
        When $\rho=0$, $\Si(\rho L^2, x/L)=x/L$ by definition. If $\rho>0$, then using $\sin(y)\le y$ for $y\ge 0$, 
        \[
        \Si(\rho L^2, x/L)= \frac{\sin(\sqrt{\rho}x)}{\sqrt{\rho} L}\le x/L.
        \]
        When $\rho<0$, using $\sinh(y)=\tfrac{1}{2}(\re^y-\re^{-y})\le 2y$ for $y\in[0,1]$,
        \[
        \Si(\rho L^2, x/L) = \frac{\sinh(\sqrt{|\rho|}x)}{\sqrt{|\rho|} L} \le \ind{\sqrt{|\rho|}x\le 1} \frac{2x}{L} + \ind{1/\sqrt{|\rho|}<x} \frac{\re^{\sqrt{|\rho|}x}}{\sqrt{|\rho|}L} \le \frac{2x}L + \frac{x}{L}\re^{\sqrt{|\rho|}x}.
        \]
        The upper bound on $\Co(\rho L^2, x/L)$ is immediate since $\cos(x)\le 1$, and $\cosh(x)\le \re^x$.
    \end{proof}
\end{lemma}

The next two propositions provide  uniform first– and second–moment estimates for the total progeny that are valid for all starting positions and for every admissible sequence $\beta_L$. 
\begin{proposition}[First moment]\label{prop:first-moment-asymp}
Consider the $\gamma$-branching random walk with parameter $\gamma\in[0,1/2)$.
Let $\beta_L\to\beta_c$ be such that $\limsup_{L\to\infty}4\beta_c(\beta_L-\beta_c) L^2<\pi^2$. Let $\varepsilon\in(0,1).$ There exists a constant $L_0$ such that for all $L\ge L_0$, all $K\in[L_0,L]$, and all $x\in[0,K]$, 
    \[
    \E_x^{\beta_L}\big[T_{[0, K]}\big]\bigg/\bigg((1+4\beta_c)\frac{\Si(4\beta_c(\beta_L-\beta_c) K^2, x/K)+1/(2\beta_cK)}{\Si(4\beta_c(\beta_L-\beta_c) K^2,1)}\re^{(K-x)/2} +1-(4\beta_c)^2\bigg) \in[1-\varepsilon, 1+\varepsilon].
    \]
\end{proposition}

\begin{proposition}[Second moment]\label{prop:second-moment-upper}
    Consider the $\gamma$-branching random walk with parameter $\gamma\in[0,1/2)$. There exists a constant $C>0$ such that for any  $\beta_L\to\beta_c$  such that $\limsup_{L\to\infty}4\beta_c(\beta_L-\beta_c) L^2<\pi^2$ there exists a constant $L_0>0$ such that for all $L\ge L_0$, all $K\in[L_0,L]$, and all $x\in[0,K]$, 
\begin{equation}\label{eq:second-moment-upper}
    \E_x^{\beta}\big[T_{[0, K]}^2\big]\le  C\frac{\Si\big(4\beta_c(\beta_L-\beta_c) K^2, 1-x/ K\big) + 1/ K}{ K^2\Si\big(4\beta_c(\beta_L-\beta_c) K^2, 1\big)^3}\re^{ K-x/2}.
    \end{equation}
\end{proposition}
We prove the propositions in Section \ref{sec:brw-progeny}, after having established preliminaries on the Laplacian random walk in Section \ref{sec:rw-progeny}. Before that, we present two corollaries of the two moment bounds. The first corollary suffices to prove the upper bound on the largest component, see page \pageref{pr:largest-upper}.

\begin{corollary}\label{cor:upper}
Consider the $\gamma$-branching random walk with parameter $\gamma\in[0,1/2)$. There exists a constant $C>0$ such that for any  $\beta_L\to\beta_c$ such that $\limsup_{L\to\infty}4\beta_c(\beta_L-\beta_c) L^2<\pi^2$ there exists a constant $L_0>0$ such that for all $L\ge L_0$ and all $K\in[L_0, L]$, all $x\in[0, K]$ and $R>0$, 
\begin{equation}\label{eq:progeny-markov}
\Prob_x^{\beta_L}\Big(T_{[0,K]}\ge R \E_0^{\beta_L}\big[T_{[0,K]}\big]\Big) \le \frac{C}{R^2}\frac{\Si\big(4\beta_c(\beta_L-\beta_c) K^2, 1-x/ K\big)+1/ K}{\Si\big(4\beta_c(\beta_L-\beta_c) K^2,1\big)}\re^{-x/2}.
    \end{equation}
\end{corollary}
\begin{proof} Assume $L_0$ is sufficiently large that we can apply Proposition \ref{prop:first-moment-asymp} for $\varepsilon=1/2$. 
We apply Markov's bound to the second moment. Using a lower bound on the first moment for $x=0$ ($\beta_c\le 1/4$ so $1-(4\beta_c)^2\ge 0$; $\Si(\alpha, 0)=0$), and the upper bound on the second moment, we obtain for some $C>0$,
    \begin{align}
        \Prob_x^{\beta_L}\Big(T_{[0, K]}\ge R \E_0^{\beta_L}\big[T_{[0, K]}\big]\Big)&\le 
        \frac{\E_x^{\beta_L}\big[T_{[0, K]}^2\big]}{R^2\E_0^{\beta_L}\big[T_{[0,
        K]}\big]^2}\nonumber\le \frac{C}{R^2}
        \frac{\Si(4\beta_c(\beta_L-\beta_c)K^2, 1-x/ K)+1/ K}{\Si(4\beta_c(\beta_L-\beta_c)K^2)}\re^{-x/2}.\nonumber\qedhere
    \end{align}
    \end{proof}

The next corollary establishes tightness from below for the progeny divided by its expectation. 
\begin{corollary}\label{cor:brw-progeny}
    Consider the $\gamma$-branching random walk with parameter $\gamma\in[0,1/2)$. There exist positive-valued functions $\varepsilon\mapsto\delta_\varepsilon$ and $\varepsilon\mapsto K_\varepsilon$ such that for any sequence $\beta_L\to\beta_c$ such that $\limsup_{L\to\infty}4\beta_c(\beta_L-\beta_c) L^2<\pi^2$ there exists a constant $L_0>0$ such that for all $L\ge L_0$ and all $K\in[L_0+K_\eps, L]$, and all $\varepsilon>0$, 
\begin{equation}\label{eq:progeny-early}
\Prob_0^{\beta_L}\Big(T_{[0, K]}\ge \delta_\varepsilon\E_0^{\beta_L}\big[T_{[0, K]}\big]\Big) \ge 1-\varepsilon.
\end{equation}
\end{corollary}
\begin{proof} 
    We assume that $L_0$ is sufficiently large so that Propositions \ref{prop:first-moment-asymp} and \ref{prop:second-moment-upper} apply, and that $\beta_L\ge \beta_c/2$ for all $L\ge L_0$. Let $K'>L_0$.
        By Proposition~\ref{prop:first-moment-asymp}, there exists a constant $c>0$ such that 
        \[
        \begin{aligned}
    \E_0^{\beta_L}\big[T_{[0, K']}\big] \ge \frac{c}{K'\Si\big(4\beta_c(\beta_L-\beta_c){K'}^2, 1\big)}\re^{K'/2}.
    \end{aligned}
    \]
    Similarly, by Proposition~\ref{prop:second-moment-upper}, we have for some $C'>0$
        \[
        \begin{aligned}
        \E_0^{\beta_L}\big[T_{[0, K']}^2\big]\le \frac{C'}{{K'}^2\Si\big(4\beta_c(\beta_L-\beta_c){K'}^2, 1\big)^2}\re^{ K'}.
        \end{aligned}
        \]
    Combined with the lower bound on the first moment, Paley--Zygmund yields that 
    \begin{equation}\label{eq:lower-init}
    \Prob^{\beta_L}_0\Big(T_{[0, K']}\ge \tfrac{1}{2}\E^{\beta_L}_0\big[T_{[0, K']}\big]\Big)\ge\frac{\E^{\beta_L}_0\big[T_{[0,K']}\big]^2}{4\E^{\beta_L}_0\big[T_{[0,K']}^2\big]}\ge c^2/(4C')=:\theta.
    \end{equation} 
    We now bootstrap this lower bound from $\theta$ to $1-\varepsilon$ at the expense of an $\varepsilon$-dependent factor on the right-hand side inside the probability. To do so, we consider the subprogeny of the offspring of the root. Informally, for the total progeny to be unexpectedly small, all offspring of the root should have a small subprogeny. Let $T_{[a,b]}(i)$ denote the progeny of the $i$th child of the root, with particles killed outside $[a,b]$. We let $K_\varepsilon, \delta_\varepsilon>0$ be two constants, and first consider the offspring of the root inside the interval $[0, K_\varepsilon]$.  
    Then, 
    \[
    \Big\{T_{[0, K]} < \tfrac{1}{2}\E_{0}^{\beta_L}\big[T_{[0, K-K_\varepsilon]}\big]\Big\}\subseteq\Big\{\forall i \in\N: X_i\in[0, K_\varepsilon],\ \ T_{[X_i, K-(K_\varepsilon - X_i)]}(i) < \tfrac{1}{2}\E_{0}^{\beta_L}\big[T_{[0, K-K_\varepsilon]}\big]\Big\}.
    \]
    Thus, for each offspring of the root (within distance at most $K_\varepsilon$ from the root), we only consider the total progeny to its right in an interval of length $K-K_\varepsilon\ge L_0$. 
    Using the independence, translating all offspring by $-X_i$, and using \eqref{eq:lower-init} for $K'=K-K_\varepsilon$,
    \[
    \Prob_0^{\beta_L}\Big(T_{[0, K]} < \tfrac{1}{2}\E^{\beta_L}_{0}\big[T_{[0, K-K_\varepsilon]}\big]\Big) \le \E_0^{\beta_L}\Big[(1-\theta)^{\#\{i: X_i\in[0, K_\varepsilon]\}}\Big]. 
    \]
    By the definition of the displacement measure $\mu_\beta$ in \eqref{eq:mu_beta}, when $\gamma\ge 0$, the number of particles in $[0, K_\varepsilon]$ stochastically dominates a Poisson random variable with mean $\beta_L\log(K_\varepsilon)\ge \beta_c\log(K_\varepsilon)/2$. Thus, we set $K_\varepsilon$ as the solution of the equation 
    \[
    \varepsilon = \E\Big[(1-\theta)^{\mathrm{Poi}(\beta_c\log(K_\varepsilon)/2)}\Big],
    \]
    and obtain 
    \[
     \Prob_0^{\beta_L}\bigg(T_{[0, K]} < \frac{1}{2}\frac{\E^{\beta_L}_{0}\big[T_{[0, K-K_\varepsilon]}\big]}{\E^{\beta_L}_{0}\big[T_{[0, K]}\big]}\E^{\beta_L}_{0}\big[T_{[0, K]}\big]\bigg) \le \varepsilon.
    \]
    The proof is finished if we can show that for $K-K_\varepsilon\ge L_0$, 
    \begin{equation}\label{eq:ratio-progeny-exp}
    \frac{\E^{\beta_L}_{0}\big[T_{[0, K-K_\varepsilon]}\big]}{\E^{\beta_L}_{0}\big[T_{[0, K]}\big]} \ge \re^{-K_\varepsilon/2}\frac{1}{3(1+K_\varepsilon)}. 
    \end{equation}
    Assume $L_0$ is sufficiently large that we can apply Corollary \ref{cor:brw-progeny} for $\eps_{\ref{cor:brw-progeny}}=1/2$. Then, 
    \[
    \frac{\E^{\beta_L}_{0}\big[T_{[0, K-K_\varepsilon]}\big]}{\E^{\beta_L}_{0}\big[T_{[0, K]}\big]} \ge \frac{1}{3}\re^{-K_\varepsilon/2} \frac{K\Si\big(4\beta_c(\beta_L-\beta_c)K^2\big)}{(K-K_\varepsilon)\Si\big(4\beta_c(\beta_L-\beta_c)(K-K_\varepsilon)^2\big)}.
    \]
    Since $\alpha\mapsto\Si(\alpha)$ is decreasing, the last ratio is at least $1$ when $\alpha\le 0$. So we may assume that $\beta_L>\beta_c$. Substituting the definition of $\Si$ in \eqref{eq:CS}, we obtain 
    \[
    \begin{aligned}
    \frac{\E^{\beta_L}_{0}\big[T_{[0, K-K_\varepsilon]}\big]}{\E^{\beta_L}_{0}\big[T_{[0, K]}\big]} &\ge \frac{1}{3}\re^{-K_\varepsilon/2}\frac{\sin(\sqrt{4\beta_c(\beta_L-\beta_c)}K)}{\sin(\sqrt{4\beta_c(\beta_L-\beta_c)}(K-K_\varepsilon))}\\ 
    &\ge \frac{1}{3}\re^{-K_\varepsilon/2}\frac{\sin(\sqrt{4\beta_c(\beta_L-\beta_c)}K) }{\sin(\sqrt{4\beta_c(\beta_L-\beta_c)}K) + \sqrt{4\beta_c(\beta_L-\beta_c)}K_\varepsilon} \\
    &= \frac{1}{3}\re^{-K_\varepsilon/2}\bigg(1+\frac{\sqrt{4\beta_c(\beta_L-\beta_c)}}{\sin(\sqrt{4\beta_c(\beta_L-\beta_c)}K)}K_\varepsilon\bigg)^{-1}.
    \end{aligned}
    \]
    Since $\limsup_{L\to\infty}4\beta_c(\beta_L-\beta_c)L^2<\pi^2$, and we only consider $L$ such that $\beta_L>0$, we may assume that $L_0$ is sufficiently large so that the ratio between brackets is at most $K_\varepsilon$ for $L\ge K\ge L_0$, proving \eqref{eq:ratio-progeny-exp}. This finishes the proof for $\delta_\varepsilon=\re^{-K_\varepsilon/2}/(6(1+K_\varepsilon))$.
\end{proof}

\subsection{Laplace--Fourier transform of the stopped random walk}\label{sec:rw-progeny}
By the many-to-one formula in Lemma~\ref{lem:many-to-one}, the expectation of the total progeny can be written in terms of a Laplacian random walk with parameter one. In particular, for the $\gamma$-BRW with parameter $\beta$,
\begin{equation}\label{eq:mto-progeny}
\E_x^\beta\big[T_{[0, L]}\big] = \re^{-x/2}\E_{2\beta_cx}\Bigg[\sum_{k=0}^{\infty}(\beta/\beta_c)^k\Ind{S_i\in[0, 2\beta_cL]\forall i\le k}\re^{S_k/(4\beta_c)}\Bigg].
\end{equation}
To study the right-hand side and related expressions, we introduce for $b>0$ $\rho>-1$, and $\pm\in\{-, +\}$,
\begin{equation}\label{eq:h}
H_b^\pm(x, \rho, L):=\E_x\Bigg[\sum_{k=0}^{\tau_L-1}(1+\rho)^k\re^{\pm bS_k}\Bigg],
\end{equation}
where 
\begin{equation}\label{eq:hitting-time}
\tau_L:=\min\big\{k:S_k\notin[0,L]\big\}
\end{equation}
denotes the hitting time of $[0, L]^\complement$. As mentioned at the beginning of Section~\ref{sec:many-to-few}, the Laplacian random walk can be constructed from a Brownian motion by stopping it at the times of a Poisson process. One might therefore aim to obtain estimates on $H^\pm_b$ using analogous quantities for Brownian motion killed outside an interval. However, this is not immediate, since killing the Brownian motion as soon as it quits the interval is not the same as killing it only at the times of the Poisson process. Instead, we will use below that $H^\pm_b$ solves a Fredholm integral equation
whose solution and asymptotics we compute explicitly. To guarantee finiteness of the solution, we establish an initial bound.
From \eqref{eq:mto-progeny} it follows that, 
\begin{equation}\label{eq:progeny-pgf}
\begin{aligned}
\re^{-x/2}\E_x\Big[(\beta/\beta_c)^{\tau_{2\beta_cL}}\Big]\le \E_x^\beta\big[T_{[0, L]}\big]&=\re^{-x/2}H^+_{1/(4\beta_c)}(2\beta_cx,\beta/\beta_c-1, 2\beta_c L)\\ &\le \re^{(L-x)/2} \tfrac{\beta_c}{\beta-\beta_c}\E_{2\beta_cx}\Big[(\beta/\beta_c)^{\tau_{2\beta_cL}}-1\Big],\end{aligned}
\end{equation}
motivating to first study the probability-generating function (pgf) of the hitting time: We show that it blows up for $\beta/\beta_c=1+\rho$ close to the radius of convergence. 
Define
\begin{equation}\label{eq:sl}
\rho_L^* := \min\Big\{\rho\ge0: \cos\big(\sqrt{\rho}L/2\big) - \sqrt{\rho}\sin\big(\sqrt{\rho}L/2\big) =0\Big\}.
\end{equation}

\begin{lemma}[PGF of the hitting time]\label{lem:pgf}
    Let $(S_n)_{n\ge 0}$ denote the Laplacian random walk. Then,
    \begin{equation}
        \E_{L/2}\big[(1+\rho)^{\tau_L}\big] =\begin{dcases}
            \frac{1+\rho}{\cosh(\sqrt{|\rho|}L/2)+\sqrt{|\rho|}\sinh({\sqrt{|\rho|}L/2})},&\text{if }\rho\in[-1,0],\\
             \frac{1+\rho}{\cos(\sqrt{\rho}L/2)-\sqrt{\rho}\sin({\sqrt{\rho}L/2)}},&\text{if }\rho\in(0,\rho_L^*),\\
             \infty,&\text{if }\rho\ge \rho_L^*.
        \end{dcases}
    \end{equation}
    Moreover, $\E_x[(1+\rho)^{\tau_L}]<\infty$ for all $x\in[0,L]$ for all $\rho\in[-1,\rho_L^*)$. We have 
    $\rho_L^*L^2\to\pi^2$ as $L\to\infty$.
\end{lemma} 
\begin{proof}
  We first analytically derive an expression for the expectation for $\rho\in[-1,0]$.
Let $S_k=L/2+X_1+\ldots+X_k$, where $X_i\sim_\mathrm{iid} \mathrm{Laplace}(1)$. For $\theta\in (0,1)$, $\E[\re^{\theta X_i}]=1/(1-\theta^2)$, and hence 
\[
M_k:= (1-\theta^2)^k\re^{\theta S_k}\text{ defines a martingale}.
\]
We write $\tau=\tau_L$.
Since $\theta\in(0,1)$ and $S_n\in[0, L]$ for $n\le \tau$, $M_{n\wedge\tau}$ is uniformly integrable for all $n$. By the optional stopping theorem,
\[
\re^{-\theta L/2} \E_{L/2}[M_\tau]=1.
\]
By symmetry, $S_\tau\ge L$ with probability $1/2$. On this event, the last step of the random walk moves in the positive direction. By the memoryless property of Laplace distributed random variables conditioned on their sign (which is an exponential distribution), the overshoot is exponentially distributed with parameter $1$, independent of $\tau$. Thus, on the event $\{S_\tau> L\}$, $S_\tau=L+Y$ with $Y\sim\mathrm{Exp}(1)$. Otherwise, it exits via 0, and $S_\tau=-Y$. Using $\E[\re^{\theta Y}]=1/(1-\theta)$ for $\theta<1$,  we have
\[
\begin{aligned}
1 &= \re^{-\theta L/2} \frac{1}{2}\re^{\theta L}\E\big[\re^{\theta Y}\big]\E_{L/2}\big[(1-\theta^2)^\tau\big]+\re^{-\theta L/2} \frac{1}{2}\E\big[\re^{-\theta Y}\big]\E_{L/2}\big[(1-\theta^2)^\tau\big] \\ &= \E_{L/2}\big[(1-\theta^2)^\tau]\Big(\frac{1/2}{1-\theta}\re^{\theta L/2} + \frac{1/2}{1+\theta}\re^{-\theta L/2}\Big) \\
&= \E_{L/2}\big[(1-\theta^2)^\tau]\frac{1/2}{1-\theta^2}\Big((1+\theta)\re^{\theta L/2} + (1-\theta)\re^{-\theta L/2}\Big)\\&=\E_{L/2}\big[(1-\theta^2)^\tau\big]\frac{1}{1-\theta^2}\big(\cosh(\theta L/2) + \theta\sinh(\theta L/2)\big).
\end{aligned}
\]
We write $1+\rho=s$.
Rearranging and substituting $\theta=\sqrt{1-s}$ yields for $s\in [0,1]$
\[
\Phi(s):=\E_{L/2}[s^\tau] = \frac{s}{\cosh(\sqrt{1-s} L/2)+\sqrt{1-s} \sinh(\sqrt{1-s} L/2)}=:F(s).
\]
Since $\Phi$ and $F$ are both analytic functions that agree on $[0,1]$, by analytic continuation they must agree on the set $[0,R)$ where $R$ corresponds to the radius of convergence of the power series representation, which is identical to the first singularity point of $F$. The first singularity point of $F$ is equal to the smallest root of the denominator, which corresponds to the $s=1+\rho_L^*$ defined in~\eqref{eq:sl}, proving the formula of the probability generating function.\smallskip

The hitting time of $[0,L]^\complement$ of the random walk started from $x\in[0,L]$ is stochastically dominated by the hitting time of $[0,L]^\complement$ when the random walk is started from the centre $L/2$, so $
\E_x[(1+\rho)^{\tau_L}]<\infty$.
As $L\!\to\!\infty$, $\rho_L^*$ behaves asymptotically like the smallest positive root of $\cos(\sqrt{\rho}L/2)\!=\!0$. So, $\rho_L^*L^2\to \pi^2$. 
\end{proof}
In view of \eqref{eq:progeny-pgf},  Lemma \ref{lem:pgf} guarantees finiteness of the expected total progeny for all $\rho$ sufficiently close to one. The next lemma exploits this observation to derive an explicit expression for $H$ by solving a Fredholm integral equation.
\begin{lemma}[Exact solution of $H^\pm$]\label{lem:exponential-moments1}
Let $b\ge 0$. Let $\rho\in[-1, \rho^*_L)$.  
Then,
\begin{align} \label{eq:exp-moments-pos}
    H_{b}^+(x,\rho, L)
    = \frac{1+\rho}{b^2 +\rho}&\Bigg(\big(1+b\big)\frac{\Si(\rho L^2,x/L) + \frac{1}{ L}\Co(\rho L^2, x/L)}{(1-\rho)\Si(\rho L^2, 1) + \frac{2}{L}\Co(\rho L^2, 1)}\re^{bL} \\ &\hspace{5pt}- \big(b-1\big)\frac{\Si(\rho L^2, 1-x/L) + \frac{1}{L}\Co(\rho L^2, 1-x/L)}{(1-\rho)\Si(\rho L^2, 1) + \frac{2}{L}\Co(\rho L^2, 1)}\Bigg) + \frac{b^2-1}{b^2+\rho}\re^{ bx}\nonumber.
    \end{align}
Moreover, 
\begin{equation}\label{eq:h-minus}
H^-_{b}(x, \rho, L)=\re^{-Lb}H^+_{b}(L-x, \rho, L).
\end{equation}
\end{lemma}

\begin{proof}
We give the proof for $H^+$, defined in \eqref{eq:h}, and explain at the end the adaptations for $H^-$.
    Since $S_n\in[0,L]$ for $n<\tau_L$, when $\rho\neq0$
    \[
    H^+_b(x, \rho, L)\le \re^{bL}\E_x\Bigg[\sum_{k=0}^{\tau_L-1}(1+\rho)^k\Bigg]=\re^{bL}\E_x\Bigg[\frac{(1+\rho)^{\tau_L}-1}{\rho}\Bigg],
    \]
     which is finite by Lemma \ref{lem:pgf}, also when $\rho=0$.
    We next find an expression for $H^+_b$, using that it is finite.\smallskip

        As $S_k$ is a symmetric random walk killed upon entering $[0,L]^\complement$, the distribution of $(L-S_k)$ and~$S_k$, started from $L-(L-x)=x$, agree. It follows by definition of $H^+$ in~\eqref{eq:h} that 
        \begin{equation}\label{eq:b-snstar}
    H^+_{b}(x, \rho, L) = \re^{bL}\E_{L-x}\left[\sum_{k=0}^\infty (1+\rho)^k \re^{-bS_k}\ind{S_i\in[0,L] \forall i\le k}\right]=:\re^{bL}G_\rho(L-x).
        \end{equation}
        We make the change of variables $y=L-x$.
         By conditioning on the first step of the random walk, it follows that $G_\rho(y)$ solves the Fredholm integral equation
        \begin{equation}\label{eq:fredholm}
        G_\rho(y)=\re^{- by} + \frac{1+\rho}{2}\int_0^{ L}\re^{-|y-z|}G_\rho(z)\, \rd z.
        \end{equation}
        Such equations have been extensively studied in the literature, mostly using Wiener-Hopf techniques, see e.g.\ the references in Ponomarev~\cite{PonomarevAsymptoticSolution2021}. The particular example \eqref{eq:fredholm} appears for example in Latter~\cite{LatterApproximateSolutions1958} and Gautesen~\cite{GautesenClassFredholm1989}. On page 28 in Latter~\cite{LatterApproximateSolutions1958}, we have the following formula for the unique finite solution
        \begin{equation}\label{eq:fredholm-sol}
\begin{aligned}G_\rho(y) 
&= \left(\frac{1+ b}{i\sqrt{\rho}-1}\re^{i L\sqrt{\rho}} + \frac{1- b}{i\sqrt{\rho}+1}\re^{- b L}\right)\frac{\re^{-iy\sqrt{\rho}}}{D( b^2+\rho)}\\
&\hspace{15pt} +\left(\frac{1+ b}{i\sqrt{\rho}+1}\re^{-i L\sqrt{\rho}} + \frac{1- b}{i\sqrt{\rho}-1}\re^{- b L}\right)\frac{\re^{iy\sqrt{\rho}}}{D( b^2+\rho)}\\
&\hspace{15pt}+ \frac{ b^2-1}{ b^2+\rho}\re^{- by},\quad D \coloneqq \frac{\re^{-i L\sqrt{\rho}}}{(i\sqrt{\rho}+1)^2} - \frac{\re^{i L\sqrt{\rho}}}{(i\sqrt{\rho}-1)^2}.\end{aligned}
\end{equation}
We reformulate this formula  for $\rho\ge0$ and obtain 
\[
\begin{aligned}
D&= \frac{1}{\big((1+i\sqrt{\rho})(1-i\sqrt{\rho})\big)^2}\left((i\sqrt{\rho}-1)^2\re^{-i L\sqrt{\rho}}-(i\sqrt{\rho}+1)^2\re^{i L\sqrt{\rho}}\right)\\ &= \frac{1}{(1+\rho)^2}\Big((1-\rho)(\re^{-i L\sqrt{\rho}}-\re^{i L\sqrt{\rho}})-2i\sqrt{\rho}(\re^{-i L\sqrt{\rho}}+\re^{i L\sqrt{\rho}})\Big)\\ &= \frac{-2i}{(1+\rho)^2}\Big((1-\rho)\sin( L\sqrt{\rho}) + 2\sqrt{\rho}\cos( L\sqrt{\rho})\Big).
\end{aligned}
\]
Similarly, 
\[
\begin{aligned}
\frac1{i\sqrt{\rho}-1}\re^{i\sqrt{\rho}( L-y)} + \frac1{i\sqrt{\rho}+1}\re^{-i\sqrt{\rho}( L-y)} &= -\frac{1}{1+\rho}\Big((i\sqrt{\rho}+1)\re^{i\sqrt{\rho}( L-y)} + (i\sqrt{\rho}-1)\re^{-i\sqrt{\rho}( L-y)}\Big) \\
&= -\frac{2i}{1+\rho}\Big(\sqrt{\rho}\cos(\sqrt{\rho}( L-y)) + \sin(\sqrt{\rho}( L-y))\Big),
\end{aligned}
\]
and 
\[
\frac1{i\sqrt{\rho}+1}\re^{-i\sqrt{\rho}y} + \frac1{i\sqrt{\rho}-1}\re^{i\sqrt{\rho}y} = -\frac{2i}{1+\rho}\Big(\sqrt{\rho}\cos(\sqrt{\rho}y) + \sin(\sqrt{\rho}y)\Big).
\]
Thus, when $\rho>0$, 
\begin{equation}
\begin{aligned}
G_\rho(y) &= \frac{1+\rho}{ b^2 +\rho}\bigg((1+ b)\frac{\sin(\sqrt{\rho}( L-y)) + \sqrt{\rho}\cos(\sqrt{\rho}( L-y))}{(1-\rho)\sin( L\sqrt{\rho}) + 2\sqrt{\rho}\cos( L\sqrt{\rho})} \\
&\hspace{50pt}+ (1- b)\frac{\sin(\sqrt{\rho}y) + \sqrt{\rho}\cos(\sqrt{\rho}y)}{(1-\rho)\sin( L\sqrt{\rho}) + 2\sqrt{\rho}\cos( L\sqrt{\rho})}\re^{- b L}\bigg) + \frac{ b^2-1}{ b^2+\rho}\re^{- by}.\label{eq:fredholm-sol-pos}
\end{aligned}
\end{equation}
Substituting these values in~\eqref{eq:b-snstar}, and dividing all numerators and denominators by $\sqrt{\rho L^2}$ (see definition of $\Si, \Co$ in~\eqref{eq:CSx}), 
~\eqref{eq:exp-moments-pos} follows for $\rho\ge0$. For $\rho< 0$ analogous reasoning yields the answer. 
Alternatively, when $\rho< 0$, one can use analytic continuation as a function of $\rho$, combined with $\sin(iy) =i\sinh(y)$ and $\cos(iy)=\cosh(y)$.
For $H^-$, the  analysis of $G_\rho$ combined with~\eqref{eq:b-snstar} yields~\eqref{eq:h-minus}, as we have 
\[
H^-_{b}(x,\rho, L) = \E_{x}\left[\sum_{k=0}^\infty(1+\rho)^k\re^{-bS_k}\ind{S_i\in[0,L]\forall i\le k}\right]=G_\rho(x).\qedhere
\]
    \end{proof}
    
\pagebreak[3]
    
\begin{corollary}[Asymptotics of $H$]\label{cor:hb-asymp}
    Let $\rho_L\to0$ be such that $\limsup_{L\to\infty}\rho_L L^2<\pi^2$, and let $\varepsilon\in(0,1).$ For each $b\ge 1$, there exists a constant $L_0$ such that for all $L\ge L_0$, all $K\in[L_0,L]$, and all $x\in[0,K]$, 
    \[
    H_b^+(x, \rho_L, K)\bigg/\bigg((1+1/b)\frac{\Si(\rho_L K^2, x/K)+1/K}{\Si(\rho_L K^2)}\re^{bK} +(1-1/b^2)\re^{bx}\bigg) \in[1-\varepsilon, 1+\varepsilon].
    \]
    Additionally, for $b\in(0,1)$, there exists a constant $L_0'$ such that for all $L\ge L_0$, all $K\in[L_0,L]$, and all $x\in[0,K]$, 
    \[
    H_b^+(x, \rho_L, K)\le (1+\varepsilon)(1+1/b)\frac{\Si(\rho_L K^2, x/K)+1/K}{\Si(\rho_L K^2)}\re^{bK}.
    \]
    \begin{proof}
        We prove an upper bound on $H^+$ and leave it to the reader to verify the lower bound using analogous bounds. Let $\delta>0$ be a sufficiently small constant depending on $\varepsilon$. We first show that the negative term on the second line in \eqref{eq:exp-moments-pos} is negligible to the positive term on the first line. To do so, we have to show that there is $L_1$ such that for all $L\ge L_1$, $K\in[L_1, L]$ and $x\in[0, K]$,
        \[
        \bigg|\frac{b-1}{b+1}\frac{\Si(\rho_L K^2, 1-x/K) + \frac{1}{K}\Co(\rho_LK^2, 1-x/K)}{\Si(\rho_L K^2, x/K) + \frac{1}{K}\Co(\rho_LK^2, x/K)}\re^{-bK}\bigg| \le \delta.
        \]
        For the numerator, we invoke Lemma \ref{lem:bound-si} to bound $\Si(\rho_LK^2, 1-x/K)\le 3\re^{\sqrt{\rho_L}K}$, and $\Co(\rho_LK^2, 1-x/K)\le \re^{\sqrt{|\rho_L|}K}$. 
        The bound is implied if 
        \[
        \bigg|\frac{b-1}{b+1}\frac{4\re^{\sqrt{|\rho_L|}K-bK/2}}{\big(\Si(\rho_L K^2, x/K) + \frac{1}{K}\Co(\rho_LK^2, x/K)\big)\re^{bK/2}}\bigg| \le \delta.
        \]
        As $\rho_L\to 0$, the numerator can be made arbitrarily small if $K$ is at least a large constant. Thus, the whole expression can be made small if the denominator is bounded away from $0$. When $\rho_L\le 0$, the denominator is increasing in $x$ (easy to verify by differentiation), and therefore minimized at $x=0$. In this case, the denominator is at least $\re^{bK}/K$, which is bounded away from $0$ for $K\ge 1$. When $\rho_L>0$,
        \[
        \inf_{x\in[0, K]}\Big(\Si(\rho_L K^2, x/K) + \frac{1}{K}\Co(\rho_LK^2, x/K)\Big)\re^{bK/2} = \inf_{x\in[0, K]}\Big(\frac{\sin(\sqrt{\rho_L}x)}{K\sqrt{\rho_L}} + \frac{1}{K}\cos(\sqrt{\rho_L}x)\Big)\re^{bK/2}.
        \]
        When $L$ is sufficiently large, it follows by differentiation that the function is increasing for $\sqrt{\rho_L}x\in[0, \pi/2]$ and minimized at $0$, and the function  is at least $1$ when $K$ is at least a large constant. For $x\in[\pi/(2\sqrt{\rho_L}), K]$, the $\sin$ is decreasing, and 
        \[
        \inf_{x\in[\pi/(2\sqrt{\rho_L}), K]}\bigg(\frac{\sin(\sqrt{\rho_L}x)}{K\sqrt{\rho_L}} + \frac{1}{K}\cos(\sqrt{\rho_L}x)\bigg)\re^{bK/2} \ge \frac{1}{K}\bigg(\frac{\sin(\sqrt{\rho_L}L)}{\sqrt{\rho_L}} - 1\bigg)\re^{bK/2} \ge 1
        \]
        for all $K$ sufficiently large, using that $\limsup_{L\to\infty} \rho_L L^2<\pi^2$. This proves that there exists $L_1$ such that for all $L\ge L_1$, all $K\in[L_1, L]$, and all $x\in[0, K]$, 
        \[
        \bigg|\frac{b-1}{b+1}\frac{\Si(\rho_L K^2, 1-x/K) + \frac{1}{K}\Co(\rho_LK^2, 1-x/K)}{\Si(\rho_L K^2, x/K) + \frac{1}{K}\Co(\rho_LK^2, x/K)}\re^{-bK}\bigg| \le 
        \bigg|\frac{b-1}{b+1}\big(2+3\re^{\sqrt{|\rho_L|}K}\big)\re^{-bK/2}\bigg|\le \delta.
        \]
        Thus, by Lemma \ref{lem:exponential-moments1}, using $\rho_L\downarrow 0$, whenever $L$ is larger than some $L_1$ and $K\in[L_1, L]$,
        \[
        \begin{aligned}
        H_b^+(x, \rho_L, K) &\le \frac{1+\rho_L}{b^2+\rho_L}(1+\delta)(1+b)\frac{\Si(\rho_L K^2, x/K) + \frac{1}{K}\Co(\rho_L K^2, x/K)}{(1-\rho)\Si(\rho_L L^2) + \frac{2}{K}\Co(\rho_L L^2)}\re^{bL} + \frac{b^2-1}{b^2+\rho_L}\re^{bx}\\
        &\le (1+\delta)^2\frac{1+b}{b^2}\frac{\Si(\rho_L K^2, x/K) + \frac{1}{K}\Co(\rho_L K^2, x/K)}{(1-\rho_L)\Si(\rho_L K^2) + \frac{2}{K}\Co(\rho_L K^2)}\re^{bK} + (1+\delta)(1-1/b^2)\re^{bx}.
        \end{aligned}
        \]
        We next  show that 
        \begin{equation}\label{eq:asymp-h-1}
        \Si(\rho_L K^2, x/K) + \frac{1}{K}\Co(\rho_L K^2, x/K) \le (1+\delta)\Big( \Si(\rho_L K^2, x/K) + \frac{1}{K}\Big).
        \end{equation}
        Rearranging terms, this is equivalent to 
        \[
        \Co(\rho_L K^2, x/K) \le \delta K \Si(\rho_L K^2, x/K) + 1+\delta.
        \]
        When $\rho_L\ge 0$, the left-hand side is at most one, while the right-hand side is at least $1+\delta$. When $\rho_L<0$, the inequality becomes equivalent to 
        \[
        \frac{1}{2}\Big(\re^{\sqrt{|\rho|}x}+\re^{-\sqrt{|\rho|}x}\Big) \le \frac{\delta K}{2}\Big(\re^{\sqrt{|\rho|}x}-\re^{-\sqrt{|\rho|}x}\Big) + 1+\delta,
        \]
        which is equivalent to 
        \[
        1+\delta K\le (\delta K-1)\re^{2\sqrt{|\rho|}x} + 2(1+\delta)\re^{\sqrt{|\rho|}x}.
        \]
        The right-hand side is minimized at $x=0$, in which case it is equal to $1+\delta (K+1)$. This proves \eqref{eq:asymp-h-1}, and we obtain 
        \[
        H_b^+(x, \rho_L, K) \le \frac{(1+\delta)^3}{b^2}(1+b)\frac{\Si(\rho_L K^2, x/K) + \frac{1}{K}}{(1-\rho_L)\Si(\rho_L K^2) + \frac{2}{K}\Co(\rho_L K^2)}\re^{bK} + (1+\delta)(1-1/b^2)\re^{bx}.
        \]
        To finish the upper bound for $b\ge 1$, we will argue that there exists $L_2$ such that for all $L\ge L_2$ and $K\in[L_2, L]$,
        \begin{equation}\label{eq:asymp-h-2}
        (1-\rho_L)\Si(\rho_L K^2) + \frac{2}{K}\Co(\rho_L K^2)\ge (1-\delta)\Si(\rho_L K^2).
        \end{equation}
        When $L$ is sufficiently large, $\rho_L<\delta/2$, and it suffices if $(\delta/2)\Si(\rho_L K^2) - 2/K \ge 0$. Since $\alpha\mapsto\Si(\alpha)$ is decreasing, it suffices if $(\delta/2)\Si(\rho_L L^2) - 2/K \ge 0$. We may assume $L$ is so large that $\Si(\rho_L L^2)\ge \liminf_{L\to\infty}\Si(\rho_L L^2)/2$, and then assume $K$ is large enough depending on $\delta$ such that the bound holds. This proves that, when $b\ge 1$,
        \[
        H_b^+(x, \rho_L, K) \le \frac{(1+\delta)^3}{b^2(1-\delta)}(1+b)\frac{\Si(\rho_L K^2, x/K) + \frac{1}{K}}{\Si(\rho_L K^2)}\re^{bK} + (1+\delta)(1-1/b^2)\re^{bx}.
        \]
        Choosing $\delta$ sufficiently small depending on $\varepsilon$ finishes the upper bound for $b\ge 1$. Lower bounds can be obtained using analogous bounds. For $b\in(0,1)$, only the bound $(b^2-1)\re^{bx}/(b^2+\rho)\le (1+\delta) (1-1/b^2)\re^{bx}$ becomes negative, and is bounded from above by 0. 
    \end{proof}
\end{corollary}
\subsection{Branching random walk}\label{sec:brw-progeny}
Armed with the asymptotics of the function $H^+$ in Corollary \ref{cor:hb-asymp}, we are ready to prove the main statements of this section.
\begin{proof}[Proof of Proposition~\ref{prop:first-moment-asymp}]
    By the many-to-one formula in Lemma~\ref{lem:many-to-one} and the formula for $T_{[0, K]}$ in~\eqref{eq:total-progeny}, we find that 
    \[
    \begin{aligned}
    \E_x^{\beta_L}\big[T_{[0,K]}\big]&=\re^{-x/2}\E_{2\beta_cx}\Bigg[\sum_{k=0}^\infty(\beta_L/\beta_c)^k\ind{S_i/(2\beta_c)\in[0, K]\forall i\le k}\re^{S_k/(4\beta_c)}\Bigg] \\
    &= \re^{-x/2}H^+_{1/(4\beta_c)}(2\beta_cx, \beta_L/\beta_c-1, 2\beta_cK).
    \end{aligned}
    \]
    The bounds follow from Corollary \ref{cor:hb-asymp} for $b=1/(4\beta_c)$, and the change of variables $L_{\ref{cor:hb-asymp}}=2\beta_cL$, $K_{\ref{cor:hb-asymp}}=2\beta_cK$, and $\rho_L=\beta_L/\beta_c-1$.
\end{proof}

\begin{proof}[Proof of  Proposition \ref{prop:second-moment-upper}]
	We set $F(x_0,\ldots,x_n) = \Ind{x_k\in [0, K]\,\forall k\le n}$, so $F=F^2$. Let $M$ be the operator from the many-to-two formula, Lemma \ref{lem:many-to-two}, which states that
	\begin{align}
		\label{eq:second_moment_1}
		\E_x^{\beta_L}[T_{[0, K]}^2] = (MF)(x) + M(M^\ast F)^2(x).
	\end{align}
	Also note that
\begin{align}\label{eq:second_moment_1b}
		MF(x_0,\ldots,x_n) &= F(x_0,\ldots,x_n)\E_{x_n}^{\beta_L}[T_{[0, K]}],
	\end{align}
    and 
    \begin{equation}\label{eq:second_moment_1c}
        \big(M^\ast F(x_0,\ldots,x_n)\big)^2 \le (M F(x_0,\ldots, x_n)\big)^2=F(x_0,\ldots,x_n)\E_{x_n}^{\beta_L}[T_{[0, K]}]^2.
    \end{equation}
    We assume $L_0$ sufficiently large so that we can use
	 Proposition~\ref{prop:first-moment-asymp} for $\varepsilon=1/2$. Using and $(a+b)^2\le 2a^2+2b^2$, we have  for some $C_1=C_1(\gamma)$, 
    \begin{equation}\label{eq:second_moment_1a}
    \E_{x_n}^{\beta_L}[T_{[0, K]}]^2 \le C_1 \frac{\Si\big(4\beta_c(\beta_L-\beta_c) K^2, x_n/K\big)^2 + 1/ K^2}{\Si^2\big(4\beta_c({\beta_L}-\beta_c)K^2\big)} \re^{ K-x_n} + C_1.
    \end{equation}
    We claim that there exists a constant $C_2=C_2(\gamma)$ such that, for all $x_n\in[0,  K]$ and any $ K\ge 1$,
    \begin{equation}\label{eq:second-pr-1}
    \left(\Si\big(4\beta_c({\beta_L}-\beta_c) K^2, x_n/ K\big)^2 + 1/ K^2\right)\re^{- x_n/4}\le C_2/ K^2.
    \end{equation}
    Indeed, by Lemma~\ref{lem:bound-si},
    \[
    \Si^2\big(4\beta_c({\beta_L}-\beta_c) K^2, x_n/ K\big)\re^{-x_n/4} \le \frac{4}{ K^2}\Big(x_n+x_n\re^{x_n\sqrt{4\beta_c|{\beta_L}-\beta_c|}}\Big)^2\re^{-x_n/4}.
    \]
    For $L_0$ sufficiently large, $2\sqrt{4\beta_c|{\beta_L}-\beta_c|}<1/8$ for any $L\ge L_0$.
    As a result, the right-hand side is bounded from above by $C_3/ K^2$ for some constant $C_3>0$, proving~\eqref{eq:second-pr-1}. Using~\eqref{eq:second-pr-1} in~\eqref{eq:second_moment_1a}, there exists a constant $C_4>0$ such that 
    \[
    \E_{x_n}^{{\beta_L}}[T_{[0, K]}]^2 \le  \frac{C_4}{ K^2\Si^2\big(4\beta_c({\beta_L}-\beta_c) K^2\big)} \re^{ K-(3/4)x_n} + C_4.
    \]
    Recalling~ \eqref{eq:second_moment_1c}, 
    \begin{align}
    \big(M^\ast F(x_0,\ldots, x_n)\big)^2 &\le F(x_0,\ldots, x_n)\frac{C_4}{ K^2\Si^2\big(4\beta_c({\beta_L}-\beta_c) K^2\big)}\re^{ K-(3/4)x_n} + C_4F(x_0,\ldots, x_n)\nonumber \\&=:G(x_0,\ldots, x_n)\frac{C_3}{ K^2\Si^2\big(4\beta_c({\beta_L}-\beta_c) K^2\big)}\re^{ K} +  C_3F(x_0,\ldots, x_n).\label{eq:G1-G2}
    \end{align}
	We have by the many-to-one formula (Lemma~\ref{lem:many-to-one}) and the definition of $H^-$ in~\eqref{eq:h},
    \begin{equation}\label{eq:second_moment_6}
	\begin{aligned}
		MG(x) &= \re^{-x/2} \E_{2\beta_cx}
        \left[({\beta_L}/\beta_c)^n\Ind{S_k\in [0,2\beta_c K]\,\forall k\le n}\re^{S_n/(4\beta_c)-(3/(8\beta_c))S_n}\right] \\
        &= \re^{-x/2}H^-_{ 1/(8\beta_c)}(2\beta_cx, {\beta_L}/\beta_c-1, 2\beta_c K).
	\end{aligned}
    \end{equation}
    We assume $L_0$ is sufficiently large so that we can use the upper bound on $H^-$ from Corollary \ref{cor:hb-asymp} for $b=1/(8\beta_c)$ and $\varepsilon=1/2$ after the change of variables $L_{\ref{cor:hb-asymp}}=2\beta_cL$, $K_{\ref{cor:hb-asymp}}=2\beta_cK$, and $\rho_L=\beta_L/\beta_c-1$. We find for some $C_5>0$,
    \[
    \begin{aligned}
    MG(x)&\le C_5\frac{\Si\big(({\beta_L}/\beta_c-1)(2\beta_c  K)^2, 1-x/ K\big)+\frac{1}{2\beta_c K}}{\Si\big(({\beta_L}/\beta_c-1)(2\beta_c K)^2\big)}  \re^{-x/2} +2\ind{\beta_c\le 1/8}\Big(\frac{1}{8\beta_c}-1\Big). 
    \end{aligned}
    \]
        Combined with \eqref{eq:G1-G2} and \eqref{eq:second_moment_1}, we obtain for some $C_6>0$ 
    \[
    \begin{aligned}
    \E_x^{\beta_L}\big[T_{[0, K]}^2\big] &\le C_6\frac{\Si\big(4\beta_c(\beta_L-\beta_c) K^2, 1-x/ K\big)+\frac{1}{2\beta_c K}}{ K^2\Si^3\big(4\beta_c(\beta_L-\beta_c) K^2\big)}\re^{ K-x/2}+ C_6 MF(x).
    \end{aligned}
    \]
    As $MF(x)=\E_x^{\beta_L}\big[T_{[0,  K]}\big]$, we obtain by Proposition \ref{prop:first-moment-asymp},
    \begin{equation}\label{eq:second-mom-three-terms}
    \begin{aligned}
    \E_x^{\beta_L}\big[T_{[0, K]}^2\big] &\le C_7\frac{\Si\big(4\beta_c(\beta_L-\beta_c) K^2, 1-x/ K\big)+\frac{1}{2\beta_c K}}{ K^2\Si^3\big(4\beta_c(\beta_L-\beta_c) K^2\big)}\re^{ K-x/2}  \\
    &\hspace{15pt}+ C_7\frac{\Si\big(4\beta_c(\beta_L-\beta_c) K^2, x/ K\big)+\frac{1}{2\beta_c K}}{ \Si\big(4\beta_c(\beta_L-\beta_c) K^2\big)}\re^{(K-x)/2} + C_7.
    \end{aligned}
    \end{equation}
    To finish the proof, we have to argue that the terms on the second line are at most a constant multiple of the term on the right-hand side in the first line. That is, there should exist $C, C'$ (not depending on the sequence $\beta_L$) such that whenever $L\ge L_0$ (for some large $L_0$), $K\in[L_0, L]$, and $x\in[0, K]$, 
    \begin{equation}\label{eq:second-mom-first}
        \frac{\Si\big(4\beta_c(\beta_L-\beta_c) K^2, x/ K\big)+\frac{1}{2\beta_c K}}{ \Si\big(4\beta_c(\beta_L-\beta_c) K^2\big)}\re^{(K-x)/2}\le C\frac{\Si\big(4\beta_c(\beta_L-\beta_c) K^2, 1-x/ K\big)+\frac{1}{2\beta_c K}}{ K^2\Si^3\big(4\beta_c(\beta_L-\beta_c) K^2\big)}\re^{ K-x/2},
    \end{equation}
    and 
    \begin{equation}\label{eq:second-mom-const}
        \frac{\Si\big(4\beta_c(\beta_L-\beta_c) K^2, 1-x/ K\big)+\frac{1}{2\beta_c K}}{ K^2\Si^3\big(4\beta_c(\beta_L-\beta_c) K^2\big)}\re^{ K-x/2} \ge 1.
    \end{equation}
    We first show \eqref{eq:second-mom-first}. We bound the numerator on the right-hand side from below by $1/(2\beta_cK)$. Rearranging terms yields that it suffices if
    \begin{align}
        2\beta_cK^3\Si^2\big(4\beta_c(\beta_L-\beta_c) K^2\big)\Big(\Si\big(4\beta_c(\beta_L-\beta_c) K^2, x/ K\big)+\frac{1}{2\beta_c K}\Big)\le C\re^{ K/2}.
    \end{align}
    By Lemma \ref{lem:bound-si}, the factors involving $\Si$ are all at most $\re^{\sqrt{4\beta_c|\beta_L-\beta_c|}L}$. Whenever $L_0$ is sufficiently large, the left-hand side is at most $\re^{K/2}$ for all $K\ge L_0$. The bound \eqref{eq:second-mom-const} follows analogously, bounding $\re^{K-x/2}\ge \re^{K/2}$ for $x\in[0, K]$.
    This proves that the term on the first line in \eqref{eq:second-mom-three-terms} is the dominant term, and that there exists a constant $C_8$ only depending on $\gamma$ such that for all $L\ge L_0$ (for some $L_0$ that is allowed to depend on $\beta_L$), $K\in[0, L]$, and $x\in[0, K]$,
    \begin{equation*}
    \E_x^{\beta_L}\big[T_{[0, K]}^2\big] \le C_7\frac{\Si\big(4\beta_c(\beta_L-\beta_c) K^2, 1-x/ K\big)+\frac{1}{2\beta_c K}}{ K^2\Si^3\big(4\beta_c(\beta_L-\beta_c) K^2\big)}\re^{ K-x/2}.\qedhere
    \end{equation*}
\end{proof}

\section{Real, fake and colliding particles}\label{sec:real}
This section establishes preliminaries to bound from below the size of the largest connected component. In particular, we ensure that the number of fake particles (see Definition \ref{def:real}) is small by arguing that the number of collisions is small if we restrict to vertices arriving after a large constant amount of time. As a result, the total progeny of the branching random walk becomes a good approximation for the size of the connected component of a vertex.
\begin{proposition}[Fake particles]\label{prop:fake}
    Consider the $\gamma$-branching random walk with parameter $\gamma\in[0, 1/2)$. There exists a constant $C>0$ such that for any sequence  $\beta_n\to\beta_c$ such that $\limsup_{n\to\infty}4\beta_c(\beta_n-\beta_c)(\log n)^2<\pi^2$ there exists a constant $k_0$  such that for all $n\ge 1$ and $m\in[n]$ such that $n/m\ge k_0$, and any $v\in[m,n]$,
    \begin{equation}
    \begin{aligned}
        \E^{\beta_n}\Big[\big|\mathrm{Fake}(\sT_{[m,n]}^-(v))\big|\Big] 
        \le C\frac{\log(v/m)+1}{\sqrt{mv}}\E^{\beta_n}\Big[|\sT_{[m,n]}^-(m)|\Big].  
        \end{aligned}
    \end{equation}
\end{proposition}
We prove the proposition at the end of the section. We next show how a lower bound on the number of real particles follows.
\begin{corollary}[Real particles]\label{cor:real}
    Consider the $\gamma$-branching random walk with parameter $\gamma\in[0, 1/2)$. There exist positive-valued functions $\varepsilon\mapsto \delta_\varepsilon$, $\varepsilon\mapsto k_\varepsilon$, and $\varepsilon\mapsto m_\varepsilon$, such that for any sequence $\beta_n \to\beta_c$ such that $\limsup_{n\to\infty}4\beta_c(\beta_n-\beta_c)(\log n)^2<\pi^2$, there exists a constant $k_0$ such that for all $\eps>0$, $n\ge 1$ and all $m\in[m_\eps, n]$ such that $n/m\ge k_0k_\eps$, 
    \[
    \Prob\bigg(\big|\mathrm{Real}\big(\sT^-_{[m,n]}(m)\big)\big|\ge \delta_\varepsilon \frac{\sqrt{n/m}\big/\log(n/m)}{\Si\big(4\beta_c(\beta_n-\beta_c)\log^2\frac{n}{m}\big)}\bigg) \ge 1-\varepsilon.
    \]
    \begin{proof}
    To lower bound the number of real particles, we use the lower bound on the total progeny (union of real and fake particles) from Corollary \ref{cor:brw-progeny}, and derive an upper bound on the number of fake particles using Proposition \ref{prop:fake}.
    We denote by $\delta_\varepsilon'$ the function $\delta_\varepsilon$ from Corollary \ref{cor:brw-progeny}, and let $K_\varepsilon$ be from said corollary. We let $L=\log n$, $\beta_L=\beta_n$, and $K=\log((n+1)/m)$. 
    Recall that the killed branching random walk \smash{$\sT^-_{[m,n]}(m)$} starts with a particle located at $\log m$, and particles are killed upon leaving \smash{$I_{[m,n]}^-=[\log m, \log(n+1))$}. We translate the random walk by $-\log m$. With $L_0$ and $K_{\varepsilon}$ from Corollary~\ref{cor:brw-progeny}, if $n/m\ge \re^{L_0+K_{\varepsilon/2}}$, then    
    \[
    \Prob\bigg(\big|\sT^-_{[m,n]}(m)\big|\le \delta_{\varepsilon/2}' \frac{\sqrt{n/m}\big/\log(n/m)}{\Si\big(4\beta_c(\beta_n-\beta_c)\log^2\frac{n}{m}\big)}\bigg)\le \varepsilon/2.
    \]
    Now, let $C$ be the constant from Proposition \ref{prop:fake}. By Markov's inequality,
    \begin{equation*}
    \begin{aligned}
        \Prob\bigg(&\big|\mathrm{Fake}\big(\sT^-_{[m,n]}(m)\big|\ge (\delta_{\varepsilon/2}'/2) \frac{\sqrt{n/m}\big/\log(n/m)}{\Si\big(4\beta_c(\beta_n-\beta_c)\log^2\frac{n}{m}\big)}\bigg) \\
        &\le \frac{2C}{\delta'_{\varepsilon/2}m}\cdot\E^{\beta_n}\Big[|\sT_{[m,n]}^-(m)|\Big]\Big/\frac{\sqrt{n/m}}{\log(n/m)\Si\big(4\beta_c(\beta_n-\beta_c)\log^2\frac{n}{m}\big)}.
        \end{aligned}
    \end{equation*}
    If we assume that $\log(n/m)$ is sufficiently large, then by Proposition \ref{prop:first-moment-asymp}, the second factor on the right-hand side is at most $2$. Setting $m_\varepsilon$ to be the smallest integer such that the first factor on the right-hand side is at most $\varepsilon/4$ finishes the proof. 
    \end{proof}
\end{corollary}

The main technical task of this section is to obtain an upper bound on the number of collisions in the cell $I_u^-$ of a vertex $u$.  Perhaps surprisingly, we obtain an upper bound that is independent of the sequence $(\beta_n)_{n\ge 1}$ whenever $n$ is sufficiently large, which is crucial in proving Proposition \ref{prop:fake}. 
\begin{lemma}[Colliding particles]\label{lem:coll}
Consider the $\gamma$-branching random walk with parameter $\gamma\in[0, 1/2)$. There exists a constant $C>0$ such that for any sequence $\beta_n\to\beta_c$ such that $\limsup_{n\to\infty}4\beta_c(\beta_n-\beta_c)(\log n)^2<\pi^2$ there exists a constant $k_0$  such that for all $n\ge 1$ and $m\in[n]$ such that $n/m\ge k_0$, and any $u, v\in[m,n]$,
    \begin{equation}
    \begin{aligned}
        \E^{\beta_n}\Big[\big|\mathrm{Colliding}(\sT_{[m,n]}^-(v), I_u^-)\big|\Big] 
        \le C\frac{\log(v/m)+1}{u\sqrt{mv}}\bigg(1+\log \frac{u}{m}\bigg)^2.  
        \end{aligned}
    \end{equation}
\end{lemma}
In Section \ref{sec:fake} below, we prove Lemma~\ref{lem:coll} and combine it
with the branching property and  progeny bounds from Section \ref{sec:progeny} to prove Proposition~\ref{prop:fake}. \smallskip 

The 
proof of the next lemma shows how the number of collisions can be bounded using a second moment argument. In particular, we apply the many-to-few lemmas from Section \ref{sec:many-to-few} to reduce bounding the expected number of collisions to estimates on the Laplacian random walk.

\smallskip 
Recall the definition of $\rho_L^*$ in \eqref{eq:sl}. Given $L>0$ and an interval $I\subseteq[0,L]$, let us define for $\rho\in(-1, \rho_L^*)$, $K\in[0, L]$, an interval $I\subseteq[0,K]$, and $x\in[0,K]$, the function 
\begin{equation}\label{eq:R}
    R_I(x, \rho, K):=\E_x\Bigg[\sum_{\ell=1}^{\tau_K-1}(1+\rho)^\ell\ind{S_\ell\in I}\Bigg].
\end{equation}

\begin{lemma}[Collision bounding via random walk]\label{lem:fake}
    Consider the $\gamma$-branching random walk with parameter $\gamma\in[0,1/2)$. Let $m,u, v, n\in\N$ such that $u, v\in[m,n]$. Let $(S_j)_{j\ge 0}$ be a Laplacian random walk and let $\tau$ denote its hitting time of  $[0,2\beta_c\log\tfrac{n+1}{m})^\complement$. Then, 
    \begin{equation}\label{eq:lem-fake}
    \begin{aligned}
    \E^{\beta}&\big[|\mathrm{Colliding}(\sT_{[m,n]}^-(v), I_u^-)|\big] \\&\le\frac{u+1}{2\sqrt{mv}}\E_{2\beta_c\log (\tfrac{v}{m})}\Bigg[\sum_{j=0}^{\tau-1}(\tfrac\beta{\beta_c})^j \re^{-S_j/(4\beta_c)}\big(R_{2\beta_c(I_u^--\log m)}(S_j,\tfrac{\beta-\beta_c}{\beta_c}, 2\beta_c\log \tfrac{n+1}{m})\big)^2\Bigg].
    \end{aligned}
    \end{equation}
\end{lemma}
\begin{proof}
We omit the superindex ``$-$'' throughout the proof.
    Recall from Definition~\ref{def:real} that a particle $t$ is said to collide in an interval $I_u$ if $t$ lies in $I_u$ and there exists another real particle $s$ in $I_u$ with $s \prec t$, where the order $\prec$ is described above Definition~\ref{def:real}. 
    We consider a superset of the colliding particles by including particles $t$ that have a position in an interval that already contains a fake particle $s$, see Definition~\ref{def:real}. 
    In formulas, 
   \begin{align}
   \sum_{t\in \sU}\E^{\beta}\Big[&\Ind{t\in \mathrm{Colliding}(\sT_{[m,n]}(v),I_u)}\Big]  \nonumber\\&\le  
   \sum_{t\in \sU}\E^{\beta}_{\log v}\Big[\Ind{\exists s\in \sU: X_{s}, X_t\in I_u, s\prec t: X_w, X_z\in I_{[m,n]} \forall w\le s, z\le t}\Big] \nonumber\\
   &\le \frac{1}{2}\E^{\beta}_{\log v}\Bigg[\sum_{\substack{s,t\in \sU\\s\neq t}}\Ind{ X_{s}, X_t\in I_u: X_w, X_z\in I_{[m,n]} \forall w\le s, z\le t}\Bigg].\label{eq:fake-pr1}
   \end{align}
   
   We apply the many-to-two lemma (Lemma \ref{lem:many-to-two}) and many-to-one lemma (Lemma \ref{lem:many-to-one}) to the expectation. 
   We write $(S_j)_{j\ge 0}$ and $(S_m')_{m\ge 0}$ for two Laplacian random walks, and let $\tilde \tau=\inf\{j: S_j\notin 2\beta_c I_{[m,n]}\}$ and $\tilde \tau'=\inf\{j: S_j'\notin 2\beta_cI_{[m,n]}\}$. Then, the right-hand side in~\eqref{eq:fake-pr1} is equal to 
   \[
   \frac{1}{2\sqrt{v}}\E_{2\beta_c\log v}\Bigg[\sum_{j=0}^{\tilde \tau-1}(\beta/\beta_c)^j \re^{-S_j/(4\beta_c)}\E_{S_n}\Bigg[\sum_{\ell=1}^{\tilde \tau'-1}(\beta/\beta_c)^\ell\ind{S_\ell'\in 2\beta_c I_u}\re^{S_\ell'/(4\beta_c)}\Bigg]^2\Bigg].
   \]
   On the indicator in the inner expectation, $S'_\ell\le 2\beta_c\log(u+1)$, and thus the expression is bounded from above by 
   \[
   \frac{u+1}{2\sqrt{v}}\E_{2\beta_c\log v}\Bigg[\sum_{j=0}^{\tilde\tau-1}(\beta/\beta_c)^j \re^{-S_j/(4\beta_c)}\E_{S_j}\Bigg[\sum_{\ell=1}^{\tilde\tau'-1}(\beta/\beta_c)^\ell\ind{S_\ell'\in 2\beta_c I_u}\Bigg]^2\Bigg].
   \]
   We translate the random walks by $-2\beta_c\log m$ and kill the shifted walks when they leave $\big[0,2\beta_c\log\big(\tfrac{n+1}m\big)\big)$. Writing $ \tau$ and $ \tau'$ for the hitting times of the complement of this interval, \eqref{eq:fake-pr1} is at most
   \begin{align*}
   &\frac{u+1}{2\sqrt{v}}\E_{2\beta_c\log (\tfrac vm)}\Bigg[\sum_{j=0}^{\tau-1}(\beta/\beta_c)^j \re^{-(S_j+\log(m))/(4\beta_c)}\E_{S_j+\log(m)}\Bigg[\sum_{\ell=1}^{\tilde\tau'-1}(\beta/\beta_c)^\ell\Ind{S_\ell'\in 2\beta_c I_u}\Bigg]^2\Bigg]\\
  &=\frac{u+1}{2\sqrt{mv}}\E_{2\beta_c\log (\tfrac{v}{m})}\Bigg[\sum_{j=0}^{\tau-1}(\beta/\beta_c)^j \re^{-S_j/(4\beta_c)}\E_{S_j}\Bigg[\sum_{\ell=1}^{\tau'-1}(\beta/\beta_c)^\ell\Ind{S_\ell'\in 2\beta_c(I^-_u-\log m)}\Bigg]^2\Bigg].\qedhere
  \end{align*}
\end{proof}
We follow a similar structure as in Section \ref{sec:progeny}:
we bound the resolvent of the random walk in the next subsection, and give the proofs of Lemma \ref{lem:coll} and Proposition \ref{prop:fake} in Section \ref{sec:fake} afterwards.

\subsection{Resolvent of the stopped random walk}
Building upon Lemma \ref{lem:fake}, we bound the resolvent $R_I(x, \rho, K)$ in \eqref{eq:R} from above in this section. We will apply it for $[a,b]$ corresponding to the cells $I_u^-$ defined in \eqref{eq:i-lower} that have length at most $2\beta_c\le 1$. Again, the next lemma is independent of the sequence~$\rho_L$ under certain conditions, even though $R_I$ depends on it.
\begin{lemma}[Resolvent upper bound]\label{lem:resolvent}
    Let $\rho_L\to0$ be such that $\limsup_{L\to\infty}\rho_L L^2<\pi^2$. There exists a constant $L_0$ such that for all $L\ge L_0$, all $K\in[L_0,L]$,  all $x\in[0,K]$, and all  $a,b\in[0,K]$ such that $b-a\in[0, 1]$,
    \[
    R_{[a,b]}(x, \rho_L, K) \le 3(b-a)(b+2)(x+2)
.\]
\end{lemma}
We  take an analytic approach using integral equations to bound $R_{[a,b)}$ from above. We approximate the indicator function in its definition in \eqref{eq:R} by a differentiable function. Let 
\begin{equation}\label{eq:q}
Q_{a,b}(x) := \begin{dcases}
        1-\cos\Big(2\pi\frac{x-a}{b-a}\Big),&\text{if }x\in[a,b], \\
        0,&\text{if }x\notin[a,b],
    \end{dcases} 
    \qquad
     F_{a,b}(x,\rho, L) := \E_x\Bigg[\sum_{n=0}^{\tau_L-1}(1+\rho)^nQ_{a,b}(S_n)\Bigg].
\end{equation}
\begin{lemma}[Analysis of $F$]\label{lem:upper-f}
Let $L>0$, and assume $[a,b]\subseteq[0, L]$. For $\rho\le 0$,
\begin{equation}\label{eq:upper-f0}
F_{a,b}(x, \rho, L) \le (b-a)(x+1)
+ Q_{a,b}(x)
\end{equation}
Moreover, for $\rho\in(0, \rho^*_L\wedge 4\pi^2/(b-a)^2)$,

\begin{equation}\label{eq:upper-fpos}
F_{a,b}(x, \rho, L) \le (b-a)(1+\rho)\frac{4\pi^2}{4\pi^2-\rho(b-a)^2}\frac{\Si(\rho L^2)+(b+1)/L}{(1-\rho)\Si(\rho L^2)+2\Co(\rho L^2)/L}(x+1)
+ Q_{a,b}(x).
\end{equation}
\end{lemma}
    \begin{proof}
        Similar to~\eqref{eq:fredholm}, $ F_{a,b}$ solves the Fredholm equation 
        \begin{equation}\label{eq:fredholm-q}
         F_{a,b}(x, \rho, L) = Q_{a, b}(x) + \frac{1+\rho}{2}\int_0^L\re^{-|x-y|}F_{a,b}(y,\rho, L)\, \rd y=: Q_{a,b}(x) + \frac{1+\rho}{2}I(x).
        \end{equation}
        By Lemma~\ref{lem:pgf}, $ F_{a,b}$ is finite when $\rho<\rho^*_L$, which is satisfied by our assumption on $\rho$.  We will find a candidate solution of the integral equation by deriving an ordinary differential equation (ODE) with boundary conditions. The solution to this differential equation is then equal to $ F_{a,b}(x, \rho, L)$ as $F_{a,b}$ and the kernel $\re^{-|x-y|}$ are sufficiently regular, see the standard reference \cite{KressLinear} for details.  \smallskip

        In what follows, we write for simplicity $F = F_{a,b}(\cdot,\rho,L)$ and $Q = Q_{a,b}$. By differentiating $I(x)$ twice, we find 
        \[
        I'(x) =  -\re^{-x}\int_0^x \re^y F(y)\, \rd y + \re^x\int_x^L\re^{-y} F(y)\, \rd y, \qquad I''(x) = \int_0^L\re^{-|x-y|}F(y)\, \rd y - 2F(x).
        \]
        As a result, $I''(x)-I(x)=-2F(x)$. Substituting the expression for $F$ from~\eqref{eq:fredholm-q}, we get the second-order linear non-homogeneous ODE\footnote{We could also obtain this ODE using the coupling of the Laplacian random walk with a Brownian motion mentioned at the beginning of Section~\ref{sec:many-to-few}. Note however the Robin boundary conditions encountered here. In the case of Brownian motion killed outside the interval $[0,L]$, one would rather obtain Dirichlet boundary conditions, see e.g.~\cite[Appendix 1.6]{HandbookBM}.}
        \begin{equation}\label{eq:ode-i}
            I''(x) + \rho I(x) = -2Q(x), \qquad x\in[0,L].
        \end{equation}
        To derive the boundary conditions from the integral definition of $I$, we find that 
        \begin{equation}\label{eq:ode-i-boundary}I(0)=\int_0^LF(y)\re^{-y}\, \rd y=I'(0), \qquad I(L)=\int_0^LF(y)\re^{-(L-y)}\, \rd y=-I'(L).\end{equation}
        We distinguish the cases $\rho\le 0$ versus $\rho>0$. 

        \smallskip 
        \noindent 
        \emph{Case 1: $\rho\in[-1,0]$.\ } Since \eqref{eq:q} is increasing in $\rho\in [-1,0]$
        , we assume w.l.o.g.\ that $\rho=0$, so $I''(x)=-2Q(x)$. Integrating twice, we obtain for some constants $C, C'>0$ that  
        \begin{equation}\nonumber
        \begin{aligned}
        I(x) = C' + Cx &- 2\ind{x\in[a, b]}\bigg(\frac{(x-a)^2}{2} - \frac{(b-a)^2}{(2\pi)^2}\Big(1-\cos\Big(\frac{2\pi}{b-a}(x-a)\Big) \Big)\bigg) \\& -2\ind{x\in[b, L]}\bigg(\frac{(b-a)^2}{2}+(x-b)(b-a)\bigg).\end{aligned}
        \end{equation}
        The boundary condition $I(0)=I'(0)$ yields that $C=C'$. The condition $I(L)=-I'(L)$ yields that $C$ satisfies the equation 
        \[
        C(L+1)-2\Big(\frac{(b-a)^2}{2} + (L-b)(b-a)\Big)  = -C+2(b-a), \quad \text{solved for }\quad C'=C = (b-a)\frac{2L-b-a+2}{L+2}.
        \]
        The factor multiplying $(b-a)$ is at most $2$, so 
        \begin{equation}\label{eq:i-rho0}
        \begin{aligned}
        I(x) \le 2(b-a)(1+x) &- 2\ind{x\in[a, b]}\bigg(\frac{(x-a)^2}{2} - \frac{(b-a)^2}{(2\pi)^2}\Big(1-\cos\Big(\frac{2\pi}{b-a}(x-a)\Big) \Big)\bigg) \\& -2\ind{x\in[b, L]}\bigg(\frac{(b-a)^2}{2}+(x-b)(b-a)\bigg).\end{aligned}
        \end{equation}
        We next argue that the factors multiplying the indicators are negative. When $x\in[b,L]$, this is clear. When $x\in[a,b]$, using $1-\cos(x)\le x^2/2$,\[
 \Big(\frac{b-a}{2\pi}\Big)^2\bigg(1-\cos\Big(\frac{2\pi}{b-a}(x-a)\Big)\bigg) \le \Big(\frac{b-a}{2\pi}\Big)^2\frac{1}{2}\Big(\frac{2\pi}{b-a}(x-a)\Big)^2=\frac{(x-a)^2}2.
\]
Thus, $I(x)\le 2(b-a)(1+x)$. Substituting this bound into \eqref{eq:fredholm-q} proves
\[
F(x)=Q(x) + \frac{1}{2}I(x)\le Q(x)+(b-a)(x+1).
\] 

        \smallskip\noindent 
        \emph{Case 2: $\rho>0$.\ } We take the same approach, but the extra term $\rho I(x)$ in \eqref{eq:ode-i} makes the analysis more involved. 
        The ODE \eqref{eq:ode-i} needs to be solved for all $x\in[0,L]$. By definition of $Q$ in \eqref{eq:q}, the right-hand side in~\eqref{eq:ode-i} equals 0 when $x\notin[a,b]$, and the general solution for such $x$ (when $\rho>0$) is given by $C\sin(\sqrt{\rho}x) + C'\cos(\sqrt{\rho}x)$ for some constants $C, C'>0$. Thus, for some constants $C_1, C_1', C_2, C_2'$, $I$ must satisfy 
        \begin{equation}\label{eq:i-outside}
            I(x) = \begin{dcases}
                C_1\sin(\sqrt{\rho}x) + C_1'\cos(\sqrt{\rho}x),&\text{if } x\in[0, a], \\
                C_{2}\sin(\sqrt{\rho}x) + C_{2}'\cos(\sqrt{\rho}x),&\text{if }x\in[b, L].
            \end{dcases}
        \end{equation}
        When $x\in(a,b)$, we have $Q(x)=1-\cos(\omega(x-a))$ with $\omega=2\pi/(b-a)$, in which case the ODE~\eqref{eq:ode-i} is inhomogeneous. By differentiation, it is easy to verify by substitution into \eqref{eq:ode-i} that it has a particular solution $I_{p}(x) = -2/(\omega^2-\rho)\cos(\omega(x-a)) - 2/\rho$. Note that $\omega^2=4\pi^2/(b-a)^2>\rho$ by assumption. Therefore, the general solution for $x\in[a,b]$ satisfies for some constants $C_{3}, C_{3}'$ that
        \begin{equation}
            I(x) = C_{3}\sin(\sqrt{\rho}x) + C_{3}'\cos(\sqrt{\rho}x) -\frac{2}{\omega^2-\rho}\cos(\omega(x-a)) - \frac{2}{\rho},\qquad\text{if }x\in(a,b).
        \end{equation}By the boundary condition $I(0)=I'(0)$ in~\eqref{eq:ode-i-boundary}, we find that $C_{1}'=\sqrt{\rho}C_{1}$. Thus, 
        \[
        I(x) = C_1\sin(\sqrt{\rho}x) + \sqrt{\rho}C_1\cos(\sqrt{\rho}x), \qquad x\in[0, a].
        \]
        As $I(x)$ is a differentiable function, the constants $C_{2}, C_2', C_3, C_3'$ will be chosen such that the limits of $I(x)$ and $I'(x)$ exist as $x\to a$ and $x\to b$.  By the continuity of $I(x)$ and $I'(x)$ at $x=a$ we find that $C_{3}$ and $C_{3}'$ must satisfy the two equations
        \[
        \begin{aligned}
        C_{1}\sin(\sqrt{\rho}a) + \sqrt{\rho}C_{1}\cos(\sqrt{\rho}a) &= C_{3}\sin(\sqrt{\rho}a) + C_{3}'\cos(\sqrt{\rho}a)-2/(\omega^2-\rho)-2/\rho, \\ C_{1}\sqrt{\rho}\cos(\sqrt{\rho}a)-\rho C_{1}\sin(\sqrt{\rho}a)  &= C_3\sqrt{\rho}\cos(\sqrt{\rho}a) - C'_3\sqrt{\rho}\sin(\sqrt{\rho}a),        \end{aligned}\]
        which are solved for 
        \begin{equation}
            C_3=C_1+\sin(\sqrt{\rho}a)\Big(\frac2\rho + \frac{2}{\omega^2-\rho}\Big), \qquad C_3' = \sqrt{\rho}C_1 + \cos(\sqrt{\rho}a)\Big(\frac2\rho + \frac{2}{\omega^2-\rho}\Big).
        \end{equation}
        As $\cos(\theta-\phi)=\cos\theta\cos\phi + \sin\theta\sin\phi$, when $x\in[a,b]$,
        \begin{align}
            I(x) &= C_1\sin(\sqrt{\rho}x) + \sqrt{\rho}C_1\cos(\sqrt{\rho}x) + \Big(\frac{2}{\rho}+\frac{2}{\omega^2-\rho}\Big)\Big(\sin(\sqrt{\rho}a)\sin(\sqrt{\rho}x) + \cos(\sqrt{\rho}a)\cos(\sqrt{\rho}x)\Big) \nonumber\\
            &\hspace{15pt}-\frac{2}{\omega^2-\rho}\cos(\omega(x-a))-\frac{2}{\rho}\nonumber\\
            &= C_1\sin(\sqrt{\rho}x) + \sqrt{\rho}C_1\cos(\sqrt{\rho}x) + \Big(\frac{2}{\rho}+\frac{2}{\omega^2-\rho}\Big)\cos(\sqrt{\rho}(x-a)) -\frac{2}{\omega^2-\rho}\cos(\omega(x-a))-\frac{2}{\rho}\nonumber \\
            &= C_1\sin(\sqrt{\rho}x) + \sqrt{\rho}C_1\cos(\sqrt{\rho}x) + \frac{2\omega^2/\rho}{\omega^2-\rho}\big(\cos(\sqrt{\rho}(x-a))-1\big) +\frac{2}{\omega^2-\rho}Q(x).
        \end{align}
        Next, we use the differentiability at $x=b$ to derive analogously that $C_2$ and $C_2'$ must satisfy the equations
        \[
        C_2=C_3-\sin(\sqrt{\rho}b) \frac{2\omega^2/\rho}{\omega^2-\rho}, \qquad C_2'=C_3' -\cos(\sqrt{\rho}b)\frac{2\omega^2/\rho}{\omega^2-\rho}.
        \]
        Substituting the values of $C_3$ and $C_3'$, we find 
        \[
        C_2 = C_1 - \big(\sin(\sqrt{\rho}b)-\sin(\sqrt{\rho}a)\big)\frac{2\omega^2/\rho}{\omega^2-\rho}, \qquad C_2'=\sqrt{\rho}C_1- \big(\cos(\sqrt{\rho}b)-\cos(\sqrt{\rho}a)\big)\frac{2\omega^2/\rho}{\omega^2-\rho}.
        \]
        We use the boundary condition $I(L)=-I'(L)$ to derive the value of $C_1$. This boundary condition is equivalent to 
        \[
        C_2\sin(\sqrt{\rho}L)+C_2'\cos(\sqrt{\rho}L) = -\sqrt{\rho}C_2\cos(\sqrt{\rho}L) +\sqrt{\rho}C_2'\cos(\sqrt{\rho}L).
        \]
        Using the values of $C_2$ and $C_2'$ in terms of $C_1$, this gives a solution for $C_1$, which is given by 
        \begin{equation}\label{eq:denom}
        \begin{aligned}
        C_1 &= \frac{2\omega^2/\rho}{\omega^2-\rho}\frac{\cos(\sqrt{\rho}(L-b))-\cos(\sqrt{\rho}(L-a)) + \sqrt{\rho}\big(\sin(\sqrt{\rho}(L-b))-\sin(\sqrt{\rho}(L-a))\big)}{(1-\rho)\sin(\sqrt{\rho}L)+2\sqrt{\rho}\cos(\sqrt{\rho}L)} \\
        &=:\frac{2\omega^2}{\omega^2-\rho}\frac{ N(a, b, \rho, L)/\sqrt{\rho}}{D(\rho, L)}\frac{1}{\sqrt{\rho}}.
        \end{aligned}
        \end{equation}
        Thus, we find the expression
        \begin{equation}\label{eq:I-expression}
        \begin{aligned}
        I(x) &= \frac{2\omega^2}{\omega^2-\rho}\frac{N(a,b,\rho, L)/\sqrt{\rho}}{ D(\rho, L)}\bigg(\frac{\sin(\sqrt{\rho}x)}{\sqrt\rho} + \cos(\sqrt{\rho}x)\bigg)\\&\hspace{15pt}+ \begin{dcases}
            0,&\text{if }x\in[0, a],\\
            \frac{2\omega^2/\rho}{\omega^2-\rho}\big(\cos(\sqrt{\rho}(x-a))-1\big) +\frac{2}{\omega^2-\rho}Q(x),&\text{if }x\in(a,b),\\
            \frac{2\omega^2/\rho}{\omega^2-\rho}\Big(\cos(\sqrt{\rho}(x-a))-\cos(\sqrt{\rho}(x-b))\Big),&\text{if }x\in[b, L].
        \end{dcases}
        \end{aligned}
        \end{equation}
        As for the case $\rho\le 0$, we will prove that the terms on the second to fourth row are non-positive. When $x\le a$, this is obvious. When $x>b$, it follows because  $0\le \sqrt{\rho}(x-b)\le \sqrt{\rho}(x-a)\le (1+\rho^*_L) L\le\pi$, and $\cos(s)$ is decreasing on $[0, \pi]$. We verify the case $x\in(a,b)$, for which we have to show that 
       \begin{equation}\label{eq:i-abneg1}
       \frac{\omega^2}{\rho}\Big(\cos\big(\sqrt{\rho}(x-a)\big)-1\Big) + 1-\cos\big(\omega(x-a)\big) \le 0.
       \end{equation}
       Let 
        \[
        u(t) := \frac{1-\cos(t)}{t^2}, \qquad t_1:=\sqrt{\rho}(x-a), \quad t_2:=\omega(x-a).
        \]
        Our assumption $\rho<4\pi^2/(b-a)^2$ is equivalent to $\sqrt{\rho}<\omega$, so $t_1<t_2\le 2\pi$. With these definitions, 
        \[
        \frac{\omega^2}{\rho}\Big(\cos(\sqrt{\rho}(x-a))-1\Big) + 1-\cos\Big(\frac{2\pi}{b-a}(x-a)\Big) = t_2^2\Big(u(t_2)-u(t_1)\Big) 
        \]
        Now, $u$ is non-increasing for $t\in[0, 2\pi]$, because $u(t) = (\operatorname{sinc}(t/2))^2/2$ by the half-angle formula 
        \(
        1-\cos(t)=2\sin^2(\tfrac{t}{2}),
        \)
        and $\operatorname{sinc}$ is non-decreasing on $[0,\pi]$.
        This proves \eqref{eq:i-abneg1}, and obtains the upper bound 
        \[
        I(x)\le \frac{2\omega^2}{\omega^2-\rho}\frac{N(a,b, \rho, L)}{\sqrt{\rho}D(\rho, L)}\bigg(\frac{\sin(\sqrt{\rho}x)}{\sqrt{\rho}}+\cos(\sqrt{\rho}x)\bigg)\le \frac{2\omega^2}{\omega^2-\rho}\frac{N(a,b, \rho, L)}{\sqrt{\rho}D(\rho, L)}(x+1).
        \]
        Invoking this bound into the integral equation \eqref{eq:fredholm-q}, we find 
        \[
        F_{a,b}(x, \rho, L)=Q_{a,b}(x) + \frac{1+\rho}{2}I(x) \le Q_{a,b}(x) + (1+\rho)\frac{2\omega^2}{\omega^2-\rho}\frac{N(a,b, \rho, L)}{\sqrt{\rho}D(\rho, L)}(x+1).
        \]
        To prove \eqref{eq:upper-fpos}, we still have to show that 
        \begin{equation}\label{eq:before-sindifs}
        \begin{aligned}
        \frac{N(a,b, \rho, L)}{\sqrt{\rho}D(\rho, L)} &=\frac{\cos(\sqrt{\rho}(L-b))-\cos(\sqrt{\rho}(L-a)) + \sqrt{\rho}\big(\sin(\sqrt{\rho}(L-b))-\sin(\sqrt{\rho}(L-a))\big)}{(1-\rho)\sin(\sqrt{\rho}L)+2\sqrt{\rho}\cos(\sqrt{\rho}L)} \\&\le (b-a)\frac{\Si(\rho L^2) + (b+1)/L}{(1-\rho)\Si(\rho L^2) + 2\Co(\rho L^2)/L}.
        \end{aligned}
        \end{equation}
        We use $\sin(\theta-\phi)\le |\theta-\phi|$ to bound the differences of the sines in the numerator by $\sqrt{\rho}(b-a)$. 
        \[
        \frac{N(a,b, \rho, L)}{\sqrt{\rho}D(\rho, L)} \le \frac{b-a}{(1-\rho)\sin(\sqrt{\rho}L)+2\sqrt{\rho}\cos(\sqrt{\rho}L)}\bigg( \frac{\cos(\sqrt{\rho}(L-b))-\cos(\sqrt{\rho}(L-a))}{\sqrt{\rho}(b-a)}+ \sqrt{\rho}\bigg).
        \]
        We use the mean-value theorem to bound the first term within brackets, i.e., 
        \[
        \begin{aligned}
         \frac{\cos(\sqrt{\rho}(L-b))-\cos(\sqrt{\rho}(L-a))}{\sqrt{\rho}(b-a)} &\le \sup_{\xi\in[a,b]}\Big|\sin(\sqrt{\rho}(L-\xi))\Big|\\&\le \sin\big(\sqrt{\rho}(L-a)\big)+\sqrt{\rho}(b-a) \le \sin(\sqrt{\rho}L) + \sqrt{\rho}b.\end{aligned}
        \]
        This yields, 
        \[
        \frac{N(a,b, \rho, L)}{\sqrt{\rho}D(\rho, L)} \le (b-a)\frac{\sin(\sqrt{\rho}L) + \sqrt{\rho}(b+1)}{(1-\rho)\sin(\sqrt{\rho}L)+2\sqrt{\rho}\cos(\sqrt{\rho}L)}.
        \]
        Dividing both numerator and denominator by $\sqrt{\rho}L$, the bound in \eqref{eq:before-sindifs} follows by the definition of $\Si$ and $\Co$ in \eqref{eq:CSx}. This finishes the proof of the lemma.
    \end{proof}
    \pagebreak[3]
    
We next prove Lemma \ref{lem:resolvent}, and establish an upper bound independent of the sequence $(\rho_L)_{L\ge 1}$.
\begin{proof}[Proof of Lemma \ref{lem:resolvent}]
We set $a_1=a-(b-a)=2a-b$ and $b_1=b+(b-a)=2b-a$, and obtain $\ind{x\in[a,b]}\le \tfrac{2}{3}Q_{a_1,b_1}(x)$. To apply Lemma \ref{lem:upper-f}, we require $a_1,b_1\in[0, K]$, which we ensure by increasing the window $[0, K]$ by $2(b-a)$ and translating the random walk to the right by $b-a$. That is, we set $\tilde x=x+(b-a)$, $\tilde K=K+2(b-a)$, $\tilde a=a_1+(b-a)=a$, $\tilde b=b_1+(b-a)=3b-2a$.  
        We have,
        \[
        \begin{aligned}
        \E_x\Bigg[\sum_{n=1}^{\tau_K-1}(1+\rho)^n\ind{S_n\in[a,b]}\Bigg] &\le \E_{\tilde x}\Bigg[\sum_{n=1}^{\tau_{\tilde K}-1}(1+\rho)^n\ind{S_n\in[a+(b-a),b+(b-a)]}\Bigg]\\
        &\le \frac{2}{3}\E_{\tilde x}\Bigg[\sum_{n=0}^{\tau_{\tilde K}-1}(1+\rho)^nQ_{\tilde a,\tilde b}(S_n)\Bigg] -\frac{2}{3}Q_{\tilde a,\tilde b}(\tilde x).
        \end{aligned}
        \]
By  Lemma \ref{lem:upper-f}, we obtain (writing $\rho = \rho_L$ and $\rho^+ = \rho \vee 0$):
\[R_{[a,b]}(x, \rho, K) \le 2(b-a)(1+\rho^+)\frac{4\pi^2}{4\pi^2-9\rho^+(b-a)^2}\frac{\Si(\rho^+ \tilde K^2)+(2b-a+1)/\tilde K}{(1-\rho^+)\Si(\rho^+ \tilde K^2)+2\Co(\rho^+ \tilde K^2)/\tilde K}(x+b-a+1).\]
The third and fourth term on the right-hand side converge to $1$ as $\rho\to 0$. Since $b-a\le 1$, the last factor is at most $x+2$. Therefore, it is sufficient to show that 
\[
\frac{\Si\big(\rho^+(K+2(b-a))^2\big)+(2b-a+1)/(K+2(b-a))}{(1-\rho^+)\Si\big(\rho^+(K+2(b-a))^2\big) + 2\Co\big(\rho^+(K+2(b-a))^2\big)/(K+2(b-a))}\le b+2.
\]
Similar to the proof of \eqref{eq:asymp-h-2}, we may assume that $K$ and $L$ are sufficiently large such that the denominator is at least $\frac23\Si\big(\rho^+(K+2(b-a))^2\big)$. Thus, it suffices if 
\[
\frac{3}{2}\frac{\Si\big(\rho^+(K+2(b-a))^2\big)+(b+2)/K}{\Si\big(\rho^+(K+2(b-a))^2\big)}\le b+2\quad \left(\Longleftrightarrow \frac{b+2}{K\Si\big(\rho^+(K+2(b-a))^2\big)}\le b+1/2\right).
\]
We may assume that $K$ and $L$ are at least so large that 
\[
K\ge \frac{8}{\liminf_{L\to\infty}\Si\big(\rho^+_LL^2\big)}, \quad\text{and }\quad \Si\big(\rho^+(K+2(b-a))^2\big) \ge \min\bigg(4, \frac{1}{2}\liminf_{L\to\infty}\Si\big(\rho^+_LL^2\big)\bigg),
\]
and it follows that for $L_0$ sufficiently large that, for $L\ge L_0$, $K\in[L_0, L]$, and $x,a,b\in[0, K]$ such that $b-a\le 1$, we have
\[R_{[a,b]}(x, \rho, K)\le 3(b-a)(b+2)(x+2). \qedhere\]
\end{proof}

\subsection{Applications to the $\gamma$-branching random walk}\label{sec:fake}
We apply the bound on the resolvent in Lemma~\ref{lem:resolvent} to prove Lemma \ref{lem:coll}, which bounds the expected number of collisions.
\begin{proof}[Proof of Lemma \ref{lem:coll}]
    By Lemma \ref{lem:fake},
    \begin{align}
        \E^{\beta_n}&\big[|\mathrm{Colliding}(\sT_{[m,n]}^-(v), I_u^-)|\big] \nonumber\\&\le\frac{u+1}{2\sqrt{mv}}\E_{2\beta_c\log (\tfrac{v}{m})}\Bigg[\sum_{j=0}^{\tau-1}(\tfrac\beta{\beta_c})^j \re^{-S_j/(4\beta_c)}\big(R_{2\beta_c(I_u^--\log m)}(S_j,\tfrac{\beta-\beta_c}{\beta_c}, 2\beta_c\log \tfrac{n+1}{m})\big)^2\Bigg].\label{eq:coll-pr}
    \end{align}
    To bound $R$, we use Lemma \ref{lem:resolvent} for $[a,b]=2\beta_c (I_u^--\log m)=2\beta_c[\log \tfrac{u}{m}, \log\tfrac{u+1}{m}]$, so $b-a\le 2\beta_c/u\le 1/2$, and $\rho_L=\beta_L/\beta_c-1$. We assume that $n/m$ is sufficiently large that $2\beta_c\log(n/m)\ge L_0$, and use the bound $b-a= 2\beta_c\log(1+1/u)\le 2\beta_c/u$.  
   We obtain
   that
    \[
    R_{2\beta_c(I_u^--\log m)}(S_j,\tfrac{\beta-\beta_c}{\beta_c}, 2\beta_c\log \tfrac{n+1}{m})
    \le 
     \frac{6\beta_c}{u}\bigg(\log\frac{u+1}{m} + 2\bigg)(S_j+2),
    \]
    abbreviating $L=2\beta_c\log \tfrac{n+1}{m}$ and $\rho=1-\beta/\beta_c$. Substituting this into \eqref{eq:coll-pr} we find,  for some $C_\gamma>0$,
    \[
    \begin{aligned}
        \E^{\beta_n}\big[|\mathrm{Colliding}(\sT_{[m,n]}^-(v), I_u^-)|\big]
        &\le \frac{18\beta_c^2(u+1)}{u^2\sqrt{mv}}\bigg(\log\frac{u+1}{m} + 2\bigg)^2\E_{2\beta_c\log (\tfrac{v}{m})}\Bigg[\sum_{j=0}^{\tau-1}(\tfrac\beta{\beta_c})^j (S_j+2)^2 \re^{-S_j/(4\beta_c)}\Bigg] \\
        &\le \frac{C_\gamma}{u\sqrt{mv}}\bigg(\log\frac{u+1}{m} + 2\bigg)^2\E_{2\beta_c\log (\tfrac{v}{m})}\Bigg[\sum_{j=0}^{\tau-1}(\tfrac\beta{\beta_c})^j  \re^{-S_j/(8\beta_c)}\Bigg]
    \end{aligned}
    \]
    using in the last bound that $(x+2)^2\re^{-x/(4\beta_c)}\le C\re^{-x/(8\beta_c)}$ for some $C>0$ depending on $\gamma$.  We use the upper bound in Corollary \ref{cor:hb-asymp} on $H^-_{1/(8\beta_c)}$ to bound the remaining expectation. We abbreviate $x=2\beta_c\log(v/m)$ and $K=2\beta_c\log((n+1)/m)$ and $\rho=\beta_n/\beta_c-1$. We find for another constant $C'_\gamma>0$,
\[
    \begin{aligned}
        \E^{\beta_n}\big[|\mathrm{Colliding}(\sT_{[m,n]}^-(v), I_u^-)|\big] 
        &\le \frac{C_\gamma'}{u\sqrt{mv}}\bigg(\log\frac{u+1}{m} + 2\bigg)^2\cdot \bigg(\frac{\Si\big(\rho^+ K^2, 1-x/K\big)+1/K}{\Si(\rho^+ K^2)}+1\bigg).
    \end{aligned}
    \]
    If $\rho^+=0$, $\Si(\rho^+ K^2, x)=x$, and the last factor is equal to $2-(x-1)/K\le 2$. When $\rho^+>0$, 
    \[
    \frac{\Si\big(\rho^+ K^2, 1-x/K\big)+1/K}{\Si(\rho^+ K^2)} = \frac{\sin(\sqrt{\rho}(K-x))+\sqrt{\rho}}{\sin(\sqrt{\rho}K)} \le \frac{\sin(\sqrt{\rho}K) + \sqrt{\rho}(x+1)}{\sin(\sqrt{\rho}K)} = 1+ \frac{x+1}{K\Si(\rho^+K^2)}.
    \]
    We may assume that $K$ is sufficiently large that $K\Si(\rho^+K^2)\ge 1$. Thus, for some $C''_\gamma>0$ 
    \begin{align*}
    \E^{\beta_n}\big[|\mathrm{Colliding}(\sT_{[m,n]}^-(v), I_u^-)|\big]
        &\le \tfrac{C_\gamma'}{u\sqrt{mv}}\big(\log\tfrac{u+1}{m} + 2\big)^2\cdot \big(x+3\big)\le \tfrac{C''_\gamma}{u\sqrt{mv}}\big(1+\log\tfrac{u}{m} \big)^2\big(1+\log\tfrac{v}{m}\big).\qedhere
        \end{align*}
\end{proof}
We are ready to prove the main proposition of this section. Having bounded the expected number of collisions in each cell, its proof uses the branching property and the progeny bounds from Section \ref{sec:progeny}.
\begin{proof}[Proof of Proposition \ref{prop:fake}] 
Recall from Definition~\ref{def:real} that colliding particles are particles that have a real parent and that have position in an interval $I_u^-$ that contains a real particle with smaller label with respect to the order $\prec$ given above Definition \ref{def:real}. The union of these colliding particles with their descendants form the set of fake particles. We decompose the set of fake particles according to the position of their unique colliding ancestor. We write \smash{$T_{I^-_{[m,n]}}(s)$} for the number of descendants of a particle $s\in \sU$ in the killed branching random walk. We now use  the branching property to obtain 
    \[
    \begin{aligned}
    \E^{\beta_n}\Big[\big|\mathrm{Fake}\big(\sT_{[m,n]}^-(v)\big)\big|\Big] &= \sum_{u=m}^n\sum_{s\in \sU}\E^{\beta_n}\Big[\ind{s\in \mathrm{Colliding}(\sT^-_{[m,n]}(v), I_u)}T_{I^-_{[m,n]}}(s)\Big]\\
    &\le \sum_{u=m}^n\E^{\beta_n}\Big[\big|\mathrm{Colliding}\big(\sT^-_{[m,n]}(v), I^-_u\big)\big|\Big]\sup_{x\in I^-_u}\E_{x}^{\beta_n}\Big[T_{I^-_{[m,n]}}\Big].
    \end{aligned}
    \]
    The second expectation is maximized at the left boundary of $I_u^-=\big[\log u, \log(u+1)\big)$, in which case the killed branching random walk agrees with $\sT^-_{[m,n]}(u)$. We obtain 
    \[
    \E^{\beta_n}\Big[\big|\mathrm{Fake}\big(\sT_{[m,n]}^-(v)\big)\big|\Big] \le 
\sum_{u=m}^n\E^{\beta_n}\Big[\big|\mathrm{Colliding}\big(\sT^-_{[m,n]}(v), I^-_u\big)\big|\Big]\cdot\E^{\beta_n}\Big[\big|
    \sT_{[m,n]}^-(u)\big|\Big].
    \]
    We assume that $n/m$ is sufficiently large such that the bound in Lemma \ref{lem:coll} applies. Abbreviating $x_u=\log(u/m)$, we obtain for some $C>0$,
    \begin{equation}\label{eq:fake-proof}
        \E^{\beta_n}\big[|\mathrm{Fake}(\sT_{[m,n]}^-(v))|\big]
        \le \frac{C(1+x_v)}{\sqrt{mv}}\E^{\beta_n}\big[\big|\sT^-_{[m,n]}(m)\big|\big]\sum_{u=m}^n\frac{(1+x_u)^2}{u} \frac{\E^{\beta_n}\big[\big|\sT^-_{[m,n]}(u)\big|\big]}{\E^{\beta_n}\big[\big|\sT^-_{[m,n]}(m)\big|\big]}.
    \end{equation}
    We finish by showing that the remaining sum is bounded from above by a constant that does not depend on the sequence $\beta_n$ whenever $n/m$ is sufficiently large. \smallskip
    
    We assume that $K=\log((n+1)/m)$ is sufficiently large that so that Proposition \ref{prop:first-moment-asymp} applies for $\varepsilon = 1/2$, $L=\log n$ and $\beta_L=\beta_n$. Then, for some constant $C_\gamma>0$,
    \begin{align*}
        \sum_{u=m}^n\frac{(1+x_u)^2}{u} \frac{\E^{\beta_n}\big[\big|\sT^-_{[m,n]}(u)\big|\big]}{\E^{\beta_n}\big[\big|\sT^-_{[m,n]}(m)\big|\big]}&\le 3\sum_{u=m}^n\frac{(1+x_u)^2}{u} \frac{\frac{\Si(4\beta_c(\beta_n-\beta_c)K^2, x_u/K)+1/(2\beta_cK)}{\Si(4\beta_c(\beta_n-\beta_c)K^2)}\sqrt{(n+1)/u}+C_\gamma}{\frac{1}{2\beta_cK\Si(4\beta_c(\beta_n-\beta_c)K^2)}\sqrt{(n+1)/m}}
        \\
        &\le 3\sum_{u=m}^n\frac{(1+x_u)^2}{u} \bigg(\big(2\beta_cK\Si(4\beta_c(\beta_n-\beta_c)K^2, x_u/K)+1\big)\sqrt{m/u} \\
        &\hspace{100pt}+2K\beta_cC_\gamma\Si(4\beta_c(\beta_n-\beta_c)K^2)\sqrt{m/n}\bigg) \\
        &=3\sqrt{m}\sum_{u=m}^n\frac{(1+x_u)^2}{u^{3/2}}\big(2\beta_cK\Si(4\beta_c(\beta_n-\beta_c)K^2, x_u/K)+1\big) \\
        &\hspace{15pt} + 6K\beta_cC_\gamma\Si\big(4\beta_c(\beta_n-\beta_c)K^2\big)\re^{-K/2}\sum_{u=m}^n\frac{(1+x_u)^2}{u}.
    \end{align*}
    We first discuss the last line. We recall that $x_u=\log(u/m)$, and that there exists a constant $C_1>0$ such that for all $m,n$ the sum is at most $C\log^3((n+1)/m)=C_1K$. By Lemma \ref{lem:bound-si}, \smash{$\Si(x)\le 3\re^{\sqrt{|x|}}$}. Assuming $n$ is sufficiently large that $\sqrt{4\beta_c(\beta_n-\beta_c)}<1/4$, it follows that the second line is at most 
    \[
    18\beta_c C_\gamma C_1 K^4 \re^{-K/4} < C_2,
    \]
    whenever $n/m$ is sufficiently large, for some constant $C_2$ only depending on $\gamma$.
    Thus, 
    \[
    \sum_{u=m}^n\frac{(1+x_u)^2}{u} \frac{\E^{\beta_n}\big[\big|\sT^-_{[m,n]}(u)\big|\big]}{\E^{\beta_n}\big[\big|\sT^-_{[m,n]}(m)\big|\big]}\le C_2 + 3\sqrt{m}\sum_{u=m}^n\frac{(1+x_u)^2}{u^{3/2}}\big(2\beta_cK\Si(4\beta_c(\beta_n-\beta_c)K^2, x_u/K)+1\big).
    \]
    To bound the remaining sum, we assume that $n$ is sufficiently large that  $\sqrt{4\beta_c|\beta_n-\beta_c|}<1/4$. We bound $K\Si(\alpha K^2, x/K)\le 3x\re^{\sqrt{|\alpha|}x}$ by Lemma \ref{lem:bound-si}. 
    We obtain, substituting $x_u=\log(u/m)$,
    \[
    \begin{aligned}
    \sum_{u=m}^n\frac{(1+x_u)^2}{u} \frac{\E^{\beta_n}\big[\big|\sT^-_{[m,n]}(u)\big|\big]}{\E^{\beta_n}\big[\big|\sT^-_{[m,n]}(m)\big|\big]}&\le C_2 + 3\sqrt{m}\sum_{u=m}^n\frac{(1+x_u)^2}{u^{3/2}}\Big(3\beta_cx_u\re^{x_u/4}+1\Big) \\
    &\le C_2 + 6\beta_cm^{1/4}\sum_{u=m}^nx_u(1+x_u)^2u^{-5/4} +3\sqrt{m}\sum_{u=m}^n (1+x_u)^2u^{-3/2}.
    \end{aligned}
    \]
    The first sum after the second inequality is bounded from above by $C_3m^{-1/4}$ for some absolute constant~$C_3$.
    Similarly, the second sum is bounded from above by $C_4/\sqrt{m}$ for some absolute constant $C_4>0$.  Substituting these bounds into \eqref{eq:fake-proof}, we obtain 
    \begin{equation*}
    \E^{\beta_n}\big[|\mathrm{Fake}(\sT_{[m,n]}^-(v))|\big]
        \le \frac{C(1+x_v)}{\sqrt{mv}}\E^{\beta_n}\big[\big|\sT^-_{[m,n]}(m)\big|\big]\cdot \big(C_2 + 6\beta_cC_3 + 3C_4\big), 
    \end{equation*}
    which finishes the proof.
\end{proof}

\section{Largest component}\label{sec:largest}
In this section we prove the main result of the paper, Theorem \ref{thm:largest} on the largest connected component. We prove a short auxiliary lemma. Recall $\Si$ from \eqref{eq:CS}.
\begin{lemma}\label{lemma:si-log2}
    Let $\rho_n\to0$ such that $\limsup_{n\to\infty}\rho_n(\log n)^2<\pi^2$. Then, 
    \[
    \liminf_{n\to\infty}\frac{\Si\big(\rho_n(\log 2n)^2\big)}{\Si\big(\rho_n(\log n)^2\big)}\ge 1.
    \]
    \begin{proof}
        Since $\alpha\mapsto\Si(\alpha)$ is non-increasing and positive for $\alpha<\pi^2$, the ratio is at least $1$ whenever $\rho_n\le 0$. Thus, we assume without loss of generality that $\rho_n>0$ for all $n$. We recall that $\sin(\theta+\phi)\ge\sin(\theta)-\phi$. As $\limsup\rho_n(\log n)^2<\pi^2$, and $\rho_n\downarrow 0$,
        \[
        \frac{\Si\big(\rho_n(\log 2n)^2\big)}{\Si\big(\rho_n(\log n)^2\big)} = \frac{\sin\big(\sqrt{\rho_n}\log 2n\big)}{\sin\big(\sqrt{\rho_n}\log n\big)}\cdot\frac{\log n}{\log2n}
        \ge\Big(1- \frac{\sqrt{\rho_n}\log 2}{\sin(\sqrt{\rho_n}\log n)}\Big)\cdot\frac{\log n}{\log n + \log 2}\longrightarrow1. \qedhere
        \]
    \end{proof}
\end{lemma}
\begin{proof}[Proof of Theorem \ref{thm:largest}]
We start with the upper bound. 

\medskip\noindent\emph{Upper bound.\ } \label{pr:largest-upper}
Let $\varepsilon>0$. We aim to find a function $M_\varepsilon$ such that for any sequence $\beta_n\to\beta_c$ such that $\limsup 4\beta_c(\beta_n-\beta_c)(\log n)^2<\pi^2$, there exists $n_0$ such that for all $n\ge n_0$, 
\begin{equation}\nonumber
\Prob^{\beta_n}\bigg(|\sL_n| \ge M_\varepsilon \frac{3(1+4\beta_c)}2\frac{\sqrt{2n}/\log 2n}{\Si\big(4\beta_c(\beta_n-\beta_c)(\log 2n)^2\big)}\bigg) \le \varepsilon.
\end{equation}
In view of Lemma~\ref{lemma:si-log2} this suffices for the theorem.
Applying Lemma \ref{prop:first-moment-asymp} for $\varepsilon=1/2$, $K=\log 2n$ and $x=0$, we may assume that $n_0$ is sufficiently large that for all $n\ge n_0$, 
\begin{equation}\label{eq:upper-pr-exp}
\E_0^{\beta_n}\big[T_{[0, \log 2n]}]\Big/ (1+4\beta_c)\frac{\sqrt{2n}/\log 2n}{\Si\big(4\beta_c(\beta_n-\beta_c)(\log 2n)^2\big)} \,\in\,  [1/2, 3/2]. 
\end{equation}
Using the upper bound, implies that 
\begin{equation}\label{eq:upper-largest-pr1}
\Prob^{\beta_n}\bigg(|\sL_n| \ge M_\varepsilon \frac{3(1+4\beta_c)}2\frac{\sqrt{2n}/\log 2n}{\Si\big(4\beta_c(\beta_n-\beta_c)(\log 2n)^2\big)}\bigg)\le \Prob^{\beta_n}\Big(|\sL_n| \ge M_\varepsilon \E_0^{\beta_n}\big[T_{[0, \log 2n]}]\Big).
\end{equation}
We rely on Markov's bound for the individual component sizes. We distinguish vertex one from the other vertices, as Proposition \ref{prop:domination-upper} requires that we consider the induced subgraph $\sG_{[m,n]}$ with $m\ge 2$ if $v\neq 1$. 
    If $|\sL_{n}|\ge s$, then at least one of the following two events holds: The component containing the first vertex has size at least~$s$, or there is a vertex $v\notin\sC_n(1)$ in a component of size at least $s$, which is contained in the event that the induced subgraph on $[2,n]$ contains a component of size at least $s$. Let us write $\sL_{[2,n]}$ for the largest connected component in this subgraph. Then, 
    \[
    \Prob^{\beta_n}\big(|\sL_{n}|\ge s\big) \le \Prob^{\beta_n}\big(|\sC_{n}(1)|\ge s\big) + \Prob^{\beta_n}\big(|\sL_{[2,n]}|\ge s\big).
    \]
    The event in the second probability is equivalent to $\{\sum_{u\in[n]}\ind{|\sC_{[2,n]}(v)| \ge s} \ge s\}$. Applying Markov's inequality yields  
    \[
    \begin{aligned}
    \Prob^{\beta_n}\big(|\sL_{n}|\ge s\big) &\le \Prob^{\beta_n}\big(|\sC_{n}(1)|\ge s\big)+\frac{1}{s}\sum_{v\in[2,n]}\Prob^{\beta_n}\Big(|\sC_{[2,n]}(v)| \ge s\Big).
    \end{aligned}
    \]
	We use the stochastic domination of the component sizes by the progenies of the corresponding KBRWs from Proposition~\ref{prop:domination-upper}, taking a union over all $k\in\N$, and bounding the number of real particles from above by the total progeny of the killed-branching random walk. By definition of $I_i^+$ in \eqref{eq:i-upper}, both for $v=1$ and $v\ge2$, a particle $u$ is killed in the branching random walk along with its descendants if $X_u\notin[0, \log(2n-1)]\subseteq[0, \log 2n]$. Thus, 
    \[
    \begin{aligned}
\Prob^{\beta_n}\Big(|\sL_n|\ge M_\varepsilon &\E_0^{\beta_n}\big[T_{[0, \log 2n]}]\Big) \\&\le 
    \Prob_0^{\beta_n}\bigg(T_{[0, \log 2n]}\ge M_\varepsilon \E_0^{\beta_n}\big[T_{[0, \log 2n]}]\bigg) \\
    &\hspace{15pt}+ \frac{1}{M_\varepsilon \E_0^{\beta_n}\big[T_{[0, \log 2n]}\big]}\sum_{v\in[2,n]}\Prob_{\log v}^{\beta_n}\bigg(T_{[0, \log 2n]}\ge M_\varepsilon \E_0^{\beta_n}\big[T_{[0, \log 2n]}]\bigg).
    \end{aligned}
\]
We apply Markov's bound to the first probability on the right-hand side. For the summands we apply Corollary \ref{cor:upper} for $x=\log (2v-1)$, $K=L=\log (2n)$, and $R=M_\varepsilon$. We obtain for some constant $C=C(\gamma)$
\begin{equation}\label{eq:upper-largest-pr2}
\begin{aligned}
\Prob^{\beta_n}\bigg(&|\sL_n| \ge M_\varepsilon \frac{3(1+4\beta_c)}{2}\frac{\sqrt{2n}/\log 2n}{\Si\big(4\beta_c(\beta_n-\beta_c)(\log 2n)^2\big)}\bigg)\\ &\le \frac{1}{M_\varepsilon} + \frac{C}{M_\varepsilon^3}\sum_{v\in[2,n]}\frac{\Si\big(4\beta_c(\beta_n-\beta_c)(\log 2n)^2, 1- \frac{\log(2v-1)}{\log 2n}\big) + \frac{1}{\log 2n}}{\E_0^{\beta_n}\big[T_{[0, \log 2n]}\big]\cdot \Si\big(4\beta_c(\beta_n-\beta_c)(\log 2n)^2\big)}\frac{1}{\sqrt{2v-1}}.
\end{aligned}
\end{equation}
We can make the right-hand side smaller than $\varepsilon$, provided that we show that the sum is bounded by a constant $C'$ only depending on $\gamma$. We use the lower bound from \eqref{eq:upper-pr-exp} to bound the denominator of the first factor in the sum from below by $\sqrt{2n}/(2 \log 2n)$. 
Then, using the upper bound on $\Si$ in Lemma \ref{lem:bound-si} for $\sqrt{|\rho_n|}=\sqrt{4\beta_c|\beta_n-\beta_c|}$, and $2v-1\ge v$,
\begin{equation}\nonumber
\begin{aligned}
\Prob^{\beta_n}&\bigg(|\sL_n| \ge M_\varepsilon \frac{3(1+4\beta_c)}2\frac{\sqrt{2n}/\log 2n}{\Si\big(4\beta_c(\beta_n-\beta_c)(\log 2n)^2\big)}\bigg)\\&\le \frac{1}{M_\varepsilon} + \frac{2C\log 2n}{M_\varepsilon^3\sqrt{2n}} \sum_{v\in[2,n]}\frac{\Si\big(4\beta_c(\beta_n-\beta_c)(\log 2n)^2, \frac{\log (2n/(2v-1))}{\log 2n}\big) + \frac{1}{\log 2n}}{\sqrt{2v-1}}\\
&\le 
\frac{1}{M_\varepsilon} + \frac{6C}{M_\varepsilon^3\sqrt{2n}} \sum_{v\in[2,n]}\Big(\frac{\log(2n/(2v-1))}{\sqrt{2v-1}}\Big(\frac{2n}{2v-1}\Big)^{\sqrt{|\rho_n|}} +\frac{1}{\sqrt{2v-1}}\Big) \\
&\le 
\frac{1}{M_\varepsilon} + \frac{6C}{M_\varepsilon^3\sqrt{2n}} \sum_{v\in[2,n]}\Big(\frac{\log(2n/v)}{\sqrt{v}}\Big(\frac{2n}{v}\Big)^{\sqrt{|\rho_n|}} +\frac{1}{\sqrt{v}}\Big) \\
&\le 
\frac{1}{M_\varepsilon} + \frac{6C}{M_\varepsilon^3}\sum_{v\in[2,n]}\Big(\log(2n/v)(2n)^{-1/4} v^{-3/4}+\frac{1}{\sqrt{2nv}}\Big)
.
\end{aligned}
\end{equation}%
In the last bound we used that when $n_0$ is sufficiently large, $\sqrt{|\rho_n|}<1/4$. The sum is bounded from above by an absolute constant $C'>0$. Choosing $M_\varepsilon$ to be a sufficiently large constant, the right-hand side is at most $\varepsilon$, which proves the upper bound.

\medskip
\noindent\emph{Lower bound, $n$ large depending on $\eps$.\ }
We proceed to the lower bound for $\sG_n$, which we split into two parts. We start by showing that
there exist functions $\varepsilon\mapsto\delta_\varepsilon$, $\varepsilon\mapsto m_\varepsilon$, such that for any sequence $\beta_n\to\beta_c$ such that $\limsup_{n\to\infty}4\beta_c(\beta_n-\beta_c)(\log n)^2<\pi^2$, there exists $n_0\in\N$ such that for all $\varepsilon>0$ and all $n\ge n_0m_\varepsilon$, 
\begin{equation}\label{eq:lower-pr1}
    \Prob^{\beta_n}\bigg(|\sL_n(\beta_n)|\ge \delta_\varepsilon \frac{\sqrt{n}/\log n}{\Si\big(4\beta_c(\beta_n-\beta_c)(\log n)^2\big)}\bigg) \ge 1-\varepsilon.
\end{equation}
Here, $n$ is required to be sufficiently large depending on $\eps$. In the subsequent step we give a direct argument to lower bound the largest component for $n\le n_0m_\eps$, still assuming it is larger than some $n_0^*$ depending on the sequence $(\beta_n)_{n\ge1}$.\smallskip

For fixed $\varepsilon>0$, we let $m_\varepsilon$ and $\tilde\delta_\varepsilon$ be the functions from Corollary \ref{cor:real}. The size of the component containing the vertex $m_\varepsilon$ is a lower bound for the size of the largest connected component. By Proposition~\ref{prop:domination-upper}, the size of this connected component stochastically dominates the number of real particles in the killed branching random walk $\sT^-_{[m,n]}(m)$. By Corollary \ref{cor:real}, 
   \[
   \begin{aligned}
   \Prob^{\beta_n}\bigg(|\sL_n|\ge \delta_\varepsilon' \frac{\sqrt{n/m_\varepsilon}/\log (n/m_\varepsilon)}{\Si\big(4\beta_c(\beta_n-\beta_c)\log^2\frac{n}{m_\varepsilon}\big)}\bigg) &\ge \Prob^{\beta_n}\bigg(\big|\sC_{[m,n]}(m)\big|\ge \tilde \delta_\varepsilon \frac{\sqrt{n/m_\varepsilon}/\log (n/m_\varepsilon)}{\Si\big(4\beta_c(\beta_n-\beta_c)\log^2\frac{n}{m_\varepsilon}\big)}\bigg)\\
   &\ge\Prob^{\beta_n}\bigg(\big|\mathrm{Real}\big(\sT^-_{[m,n]}(m)\big)\big|\ge \tilde \delta_\varepsilon \frac{\sqrt{n/m_\varepsilon}/\log (n/m_\varepsilon)}{\Si\big(4\beta_c(\beta_n-\beta_c)\log^2\frac{n}{m_\varepsilon}\big)}\bigg)\\ &\ge 1-\varepsilon.
   \end{aligned}
   \]
   Similar to the reasoning after \eqref{eq:ratio-progeny-exp}, 
   \[
   \frac{\sqrt{n/m_\varepsilon}/\log (n/m_\varepsilon)}{\Si\big(4\beta_c(\beta_n-\beta_c)\log^2\frac{n}{m_\varepsilon}\big)}\bigg/\frac{\sqrt{n}/\log (n)}{\Si\big(4\beta_c(\beta_n-\beta_c)\log^2n\big)} \ge \frac{1}{\sqrt{m_\eps}(1+\log m_\eps)}
   \]
   for all $n/m_\varepsilon$ sufficiently large depending on $(\beta_n)_{n\ge1}$. This proves the lower bound for the graph $\sG_n$ with $\delta_\varepsilon=\tilde\delta_\varepsilon/(\sqrt{m_\eps}(1+\log m_\eps))$ when $n\ge n_0m_\eps$. 
\medskip

\pagebreak[3]
\noindent\emph{Lower bound, removing $\eps$-dependence on $n$.\ } We next prove the bound \eqref{eq:lower-pr1} for $n\in[n_0^*,n_0m_\varepsilon]$, where $n_0^*\ge n_0$ is a large constant only depending on $(\beta_n)_{n\ge1}$, and a (possibly) smaller value $\delta_\varepsilon$. Let $\mathrm{deg}_n(1)$ denote the degree of vertex one in $\sG_n$. We will lower bound the size of the largest component by $1+\mathrm{deg}_n(v)$. 
We assume $n_0^*$ is such that for $n\ge n_0^*$ the following three bounds hold:
\[
\begin{aligned}
&(1)\quad \Si\big(4\beta_c(\beta_n-\beta_c)(\log n)^2\big)\ge \min\bigg(1,\frac12 \liminf_{n\to\infty}\Si\big(4\beta_c(\beta_n-\beta_c)(\log n)^2\big)\bigg)=:C_0, \\
&(2)\quad \beta_n\ge \beta_c/2, \\
&(3)\quad \E^{\beta_c/2}[\mathrm{deg}_{n_0^*}(1)]\ge \max\Big(\frac{\sqrt{n_0}}{C_0},1\Big).
\end{aligned}
\]
The latter condition can be met since the degree of a fixed vertex diverges as the graph grows. For $n\in[n_0^*, n_0m_\varepsilon]$,
\begin{align}
    \Prob^{\beta_n}\bigg(|\sL_n|\ge \delta_\varepsilon \frac{\sqrt{n}/\log n}{\Si\big(4\beta_c(\beta_n-\beta_c)(\log n)^2\big)}\bigg) &\ge \Prob^{\beta_n}\bigg(|\sL_n|\ge \delta_\varepsilon \frac{\sqrt{n}/\log n}  {C_0}\bigg) \nonumber\\
    &\ge \Prob^{\beta_n}\bigg(|\sL_n|\ge \delta_\varepsilon \sqrt{m_\varepsilon} \frac{\sqrt{n_0}} {C_0}\bigg) \nonumber\\
    &\ge \Prob^{\beta_n}\Big(1+\mathrm{deg}_n(1) \ge \delta_\varepsilon \sqrt{m_\varepsilon}\E^{\beta_c/2}\big[\mathrm{deg}_{n_0^\ast}(1)\big] \Big)\nonumber \\
    &\ge \Prob^{\beta_c/2}\Big(\mathrm{deg}_{n_0^*}(1) \ge \delta_\varepsilon \sqrt{m_\varepsilon}\E^{\beta_c/2}\big[\mathrm{deg}_{n_0^\ast}(1)\big] -1 \Big).\nonumber 
    \end{align}
    We still have the freedom to choose $\delta_\varepsilon$ sufficiently small so that the last line is at least $1-\varepsilon$. Such existence follows from large deviations for sums of Bernoulli random variables with summed expectation at least $1$. This finishes the proof of the tightness from below.
    \end{proof}

\section{Typical components and progeny of the local limit}\label{sec:progeny-local}
 We will prove Theorem \ref{thm:typtail} on  the tail of the typical component size distribution. In view of Corollary~\ref{cor:truncated}, it suffices to study the distribution of the progeny of the local limit, which is item $(1)$ in the next lemma. Items $(2, 3)$ will be used to prove Theorem \ref{thm:susceptibility} on the susceptibility in the next section.

\begin{lemma}\label{lem:local-lim-progeny}
    Consider the $\gamma$-branching random walk for some $\gamma\in[0, 1/2)$ at $\beta=\beta_c$. Let $X\sim \mathrm{Exp}(1)$. The following hold:
    \begin{enumerate}
        \item There exist constants $c,C\in(0,\infty)$, such that for all $k\ge 2$,
        \[
        \frac{c}{k(\log k)^2} \le \Prob_{0}(T_{(-\infty,X]} \ge k) \le \frac{C}{k(\log k)^2}.
        \]
        \item $\E_{0}\big[T_{(-\infty, X]}\big]=4(1-\gamma^2)$.
        \item $\E_{0}\big[T_{(-\infty, X]}\log T_{(-\infty, X]}\big]=\infty$.
    \end{enumerate}
\end{lemma}

 \begin{proof}
 In what follows, for two sequences $(a_k)_k$ and $(b_k)_k$, the notation $a_k\asymp b_k$ means that there exists a constant $C$ such that $C^{-1}b_k \le a_k \le C b_k$ for all $k$. Furthermore, $C$ denotes a finite positive constant whose value may change from line to line.\smallskip
 
     We first prove item $(1)$, following Aidekon \cite{aidekon2010tail}.  
      We translate the random walk by $-X$ so that particles are killed upon entering $(0, \infty)$, and consider the value $y$ of $X$. Let $\mathrm{Min}$ be the total minimum of the killed $\gamma$-BRW. We first prove that  for  $L\ge 1$ and $y\ge0$,
     \begin{equation}
     \label{eq:tail_M}
     \Prob_{-y}(\mathrm{Min} \le -L) \asymp \re^{-(L-y)/2} \frac{y+1}{L+1}.
     \end{equation}
     To see this, fix $L\ge 1$ and stop particles whenever they first enter $(-\infty,-L)$. 
     Let $H_{-L}$ denote the number of those particles. Let $(S_n)_{n\ge0}$ denote the Laplacian random walk from the beginning of Section~\ref{sec:many-to-few} and denote by $\tau_-$ the hitting time of $(-\infty,-2\beta_cL)$ and by $\tau_+$ the hitting time of $(0,\infty)$. By Lemma~\ref{lem:many-to-one}, 
     \begin{align*}
     \E_{-y}[H_{-L}] 
     &= \re^{y/2} \E_{-2\beta_c y}\bigg[\sum_{n=0}^\infty \Ind{S_n < -2\beta_c L,\,S_k \in [-2\beta_c L,0]\,\forall k\le n-1} \re^{S_n/(4\beta_c)}\bigg] \\
     &= \re^{-(L-y)/2} \E_{-2\beta_c y}[\re^{(S_{\tau_{-}} +2\beta_c L)/(4\beta_c)}\ind{\tau_- < \tau_+}].
     \end{align*}
     On the event in the indicator, $-(2\beta_cL+S_{\tau})$ is an independent $\mathrm{Exp}(1)$ random variable by the memoryless property. So, 
     \begin{equation}\label{eq:H_first_moment}
      \E_{-y}[H_{-L}] = \re^{-(L-y)/2}\frac{1}{1+1/(4\beta_c)}\Prob_{-2\beta_c y}(\tau_-<\tau_+)\asymp \re^{-(L-y)/2} \frac{y+1}{L+1},
     \end{equation}
     using for instance Lemma 2.2 in \cite{aidekon2010tail} to bound the `gambler's ruin' probability in the last step.
      This shows the upper bound in \eqref{eq:tail_M} since, by Markov's inequality,
     \begin{equation}
     	\label{eq:M_upper_bound}
     \Prob_{-y}(\mathrm{Min}\le -L) = \Prob_{-y}(H_{-L} \ge 1) \le \E_{-y}[H_{-L}].
     \end{equation}
     For the lower bound, we upper bound the second moment of $H_{-L}$. By Lemma~\ref{lem:many-to-two}, we have
     \begin{equation}
     \label{eq:H_second_moment}
         \E_{-y}[H_{-L}^2] \le \E_{-y}[H_{-L}] + \re^{y/2}\E_{-2\beta_c y}\left[\sum_{n=0}^{(\tau_-\wedge \tau_+)-1} \re^{S_n/(4\beta_c)}\E_{S_n/(2\beta_c)}[H_{-L}]^2\right]
     \end{equation}
     We use \eqref{eq:H_first_moment} to bound the second term on the right-hand side (bounding the fraction \smash{$\frac{y+1}{L+1}$} by one). Recalling the definition of $H^-_b$ from \eqref{eq:h}, and the identity with $H^+_b$ from \eqref{eq:h-minus}, the second term on the right-hand side of \eqref{eq:H_second_moment} is upper bounded by
	\begin{align*}
		C \re^{y/2-L}\E_{-2\beta_c y}\left[\sum_{n=0}^{(\tau_-\wedge \tau_+)-1} \re^{S_n/(4\beta_c)}\re^{-S_n/(2\beta_c)}\right] &=C\re^{y/2-L}H^-_{1/(4\beta_c)}\big(2\beta_c (L-y), 0, 2\beta_c L\big)\\&=C\re^{y/2-3L/2}H^+_{1/(4\beta_c)}\big(2\beta_c y, 0, 2\beta_c L\big).
	\end{align*}
    Using the asymptotics from Corollary \ref{cor:hb-asymp}, with $\Si(x, 0, L)=x$ and $\Co(x, 0, L)=1$ by \eqref{eq:CSx},  for all $L$ sufficiently large and $y\in [0, L]$,
    \begin{equation}\label{eq:H_second_moment_final}
    \E_{-y}[H_{-L}^2]\le \E_{-y}[H_{-L}] + C\re^{y/2-3L/2}\Big(\frac{y+1}{L}\re^{L} + \re^{y}\Big) 
    \le C\E_{-y}[H_{-L}].
    \end{equation}
	 The Paley--Zygmund inequality then yields
	\[
	\Prob_{-y}(\mathrm{Min}\le -L) = \Prob_{-y}(H_{-L}\ge 1) \ge \frac{\E_{-y}[H_{-L}]^2}{\E_{-y}[H_{-L}^2]} \ge \re^{-(L-y)/2}\frac{y+1}{C(L+1)},
	\]
	which, together with \eqref{eq:M_upper_bound}, implies \eqref{eq:tail_M}.\smallskip
	
	Armed with \eqref{eq:tail_M}, we now finish the proof of the first item of the lemma. It is enough to prove the inequalities for $k\ge k_0$ for some $k_0\in\N$. Let $k\ge 2$ and let $L$ be such that $\re^{L/2}/(1+L) = k$, so that $L = 2\log k + 2\log\log k+ O(1)$ as $k\to\infty$. Hence, there is $k_0\in \N$, such that $L \ge 1$ for $k\ge k_0$. Aiming for the lower bound first, note that it is enough to prove the lower bound under $\Prob_0$, since 
	\[
	 \Prob_0(T_{(-\infty,X]}\ge 0) \ge \Prob_0(T_{(-\infty,0]}\ge 0).
	\]
	Bounding from below the total progeny by the total progeny of the particle reaching the minimum position, we have for all $k'\in\N$,
	\begin{equation}
		\label{eq:T_lower_bound}
		\Prob_{0}(T_{(-\infty,0]} \ge k') \ge \Prob_{0}(\mathrm{Min} \le -L) \Prob_{-L}(T_{[-L,0]} \ge k').
	\end{equation}
	By \eqref{eq:tail_M}, we have for all $y\in[0,x]$,
	\begin{equation}
		\label{eq:tail_M_k}
	 \Prob_{-y}(\mathrm{Min} \le -L) \asymp \re^{-(L-y)/2}\frac{1+y}{1+L} \asymp \re^{y/2}(y+1) \frac {1} {k(\log k)^2}.
	 \end{equation}
	 Furthermore, by Proposition~\ref{prop:first-moment-asymp}, we have for $L$ large enough (hence, for $k$ large enough),
	 \begin{equation}
	 	\label{eq:T_expectation_k}
	 \E_{-L}[T_{[-L,0]}] = \E_0[T_{[0,L]}] \asymp \re^{L/2}/(1+L) \asymp k.
	 \end{equation}
	 It follows from \eqref{eq:T_expectation_k} and Corollary~\ref{cor:brw-progeny} that
	 \begin{equation}
	 	\label{eq:T_lowerbound_k}
	 \Prob_{-L}(T_{[-L,0]} \ge C^{-1} k) \ge 1/2.
	 \end{equation}
	 Plugging \eqref{eq:tail_M_k} and \eqref{eq:T_lowerbound_k} into \eqref{eq:T_lower_bound}, we get for $k$ large enough,
	 \[
	 \Prob_{0}(T_{(-\infty,0]} \ge C^{-1} k) \ge C^{-1} \frac{1}{k(\log k)^2},
	 \]
	 and therefore, for $k$ large enough,
	 \[
	 	 \Prob_{0}(T_{(-\infty,0]} \ge k) \ge C^{-1} \frac{1}{k(\log k)^2},
	 	 \]
	 which, proves the lower bound in the first item of the lemma.\smallskip
	 
	 We now prove the upper bound. Using that on the event $\{\mathrm{Min}>-L\}$, we have $T_{(-\infty,0]} = T_{[-L,0]}$, 
	 \begin{equation}
	 	\label{eq:T_upper_bound}
	 	\Prob_{-y}(T_{(-\infty,0]} \ge k) \le \Prob_{-y}(\mathrm{Min}\le -L) + \Prob_{-y}(T_{[-L,0]} \ge k).
	 \end{equation}
	 The first term was upper bounded in \eqref{eq:tail_M_k}. For the second term, using \eqref{eq:T_expectation_k} we get from Corollary~\ref{cor:upper},
	 \[
	 \Prob_{-y}(T_{[-L,0]} \ge k) = \Prob_{L-y}(T_{[0,L]}\ge k) \le C \re^{-(L-y)/2} \frac{y+1}{L+1} \le C \re^{y/2}(y+1)\frac 1 {k(\log k)^2}
	 \]
	 Plugging this into \eqref{eq:T_upper_bound} and using \eqref{eq:tail_M_k}, it follows that for $y\le L$, we have
	 \begin{equation}
	 	\label{eq:T_upper_bound_k}
	 	\Prob_{-y}(T_{(-\infty,0]} \ge k) \le C \re^{y/2}(y+1)\frac 1 {k(\log k)^2}.
	 \end{equation}
	 It remains to integrate over $y$. Since $L = 2\log k + 2\log\log k+ O(1)$ as $k\to\infty$, 
	 \begin{align*}
	 \Prob_{0}(T_{(-\infty,X]} \ge k) 
	 &\le \int_0^L \Prob_{-y}(T_{(-\infty,0]} \ge k)\re^{-y}\,\rd y + \re^{-L}\\
	 &\le \frac C {k(\log k)^2}\int_0^L (y+1)\re^{-y/2}\,\rd y + C\frac{1}{(1+L)\re^{L/2}}\\
	 &\le \frac{C}{k(\log k)^2}.
	 \end{align*}
	 This concludes the proof of the upper bound in the first part of the lemma.\smallskip
	 
	 Item $(2)$ of the lemma could be proved directly, using the many-to-one lemma and an expression of the Green's kernel of the Laplacian random walk killed above 0, but we can also deduce it from Proposition~\ref{prop:first-moment-asymp}. The expectation is finite by item $(1)$ of the lemma. By the monotone convergence theorem, we have for every $y\ge 0$,
	 \begin{align*}
	 \E_{-y}[T_{(-\infty,0]}] 
	 = \lim_{L\to\infty} \E_{-y}[T_{[-L,0]}] 
	 = \lim_{L\to\infty}\E_{L-y}[T_{[0,L]}]
	 = (1+4\beta_c)\re^{y/2} + 1-(4\beta_c)^2,
	 \end{align*}
	 using Proposition~\ref{prop:first-moment-asymp} for the last equality. It follows that
	 \begin{align*}
	 \E_{0}\big[T_{(-\infty,0]}\big] 
	 = \int_0^\infty \left((1+4\beta_c)\re^{y/2} + 1-(4\beta_c)^2\right) e^{-y}\,\rd y
	 = 2(1+4\beta_c) + 1 - (4\beta_c)^2.
	 \end{align*}
	 After using $\beta_c=1/4-\gamma/2$, this gives the expression in item $(2)$ of the lemma.\smallskip
	 
	 Item $(3)$ is a direct consequence of item $(1)$, using that $\E[Z\log Z] = \int_0^\infty (\log z+1) \Prob(Z\ge z)\,\rd z$ for any positive random variable $Z$.
 \end{proof}

 \begin{proof}[Proof of Theorem \ref{thm:typtail}]
     Immediate from Lemma \ref{lem:local-lim-progeny} and Corollary \ref{cor:truncated}.
 \end{proof}
 
 \section{Finite susceptibility in the critical window}\label{sec:sus}
We aim to prove Theorem \ref{thm:susceptibility}, for which we use the following auxiliary lemma that quantifies the intuition that the mean component size is driven by components of size $O(1)$.
\begin{lemma}[Large components contribute negligibly to the mean]\label{lem:negligible}
    Consider the the $\gamma$-growing random graph with $\gamma\in[0, 1/2)$. Assume $\beta_n\to\beta_c$ such that $\limsup4\beta_c(\beta_n-\beta_c)(\log n)^2<\pi^2$. For all $\varepsilon>0$, there exists $k=k(\varepsilon)>0$ such that for all $n$ sufficiently large,
    \[
    \E^{\beta_n}\Bigg[\sum_{v=1}^n \big|\sC_n(v)\big |\ind{|\sC_n(v)|\ge k}\Bigg] \le \varepsilon n.
    \]
\end{lemma}
\begin{proof}
    We distinguish components containing old vertices---that are likely to be larger than a large constant $k$---from components containing only young vertices that arrived after time $\delta n$ for some small $\delta>0$---that are unlikely to be larger than $k$. For convenience, we assume $\delta n\in\N$, and  decompose
\begin{equation}\label{eq:sus-before-2}
\begin{aligned}
\sum_{v=1}^n|\sC_n(v)|\ind{|\sC_n(v)\ge k} = \sum_{\sC: \sC\cap[1,\delta n-1]\neq \emptyset} \!\!\!\!|\sC|^2\ind{|\sC\ge k} &+ \sum_{v=\delta n}^n |\sC_n(v)|\ind{|\sC_n(v)|\ge k, \sC_n(v)\cap[1,\delta n-1]=\emptyset}.
\end{aligned}
\end{equation}
We increase the two sums on the right-hand side, starting with the first sum. If $v$ is the oldest vertex in its component $\sC_n(v)$, then $\sC_n(v)=\sC_{[v, n]}(v)$. Therefore, ordering components  according its oldest vertex yields,
\[
\sum_{\sC: \sC\cap[1,\delta n-1]\neq \emptyset} |\sC|^2\ind{|\sC|\ge k}\le \sum_{v=1}^{\delta n-1} |\sC_{[v, n]}(v)|^2. 
\]
We next increase the second sum on the first line on the right-hand side in \eqref{eq:sus-before-2}: We include all vertices arrived after $\delta n-1$ in the sum, but for each vertex we consider the component induced on vertices with label in $[\delta n, n]$. We obtain,
\begin{equation}\label{eq:sus-before-proof}
\begin{aligned}
    \E^{\beta_n}\Bigg[\sum_{v=1}^n|\sC_n(v)|\ind{|\sC_n(v)|\ge k}\Bigg]\le  \E^{\beta_n}\Bigg[\sum_{v=1}^{\delta n-1} |\sC_{[v, n]}(v)|^2\Bigg] + \E^{\beta_n}\Bigg[\sum_{v=\delta n}^n |\sC_{[\delta n,n]}(v)|\ind{|\sC_{[\delta n, n]}(v)|> k}\Bigg].
    \end{aligned}
\end{equation}
We apply the stochastic domination from Proposition \ref{prop:domination-upper}, and use that
    real particles of the $\gamma$-branching random walk started from $\log v$ restricted to $I^+_{[v,n]}=(\log(2v-3), \log (2n-1)]$ all have position at least $\log (2v-1)$ by Definition \ref{def:pa-brw}. Translating the branching random walk by $-\log (2v-1)$, and using $(2n-1)/(2v-1)\le 2n/v$, we obtain  
    \[
    \begin{aligned}
    \E^{\beta_n}\Bigg[\sum_{v=1}^{\delta n-1} |\sC_{[v, n]}(v)|^2\Bigg]&\le 
    \sum_{v=1}^{\delta n-1}\E^{\beta_n}\Big[\big|\mathrm{Real}\big(\sT^+_{[1,n]}(1)\big)\big|^2\Big] \\
    &\le 
    \sum_{v=1}^{\delta n-1}\E_{\log (2v-1)}^{\beta_n}\Big[T^2_{[\log (2v-1), \log(2n-1]}\Big]\le\sum_{v=1}^{\delta n-1}\E_{0}^{\beta_n}\Big[T^2_{[0, \log (2n/v)]}\Big].
    \end{aligned}
    \]
    We aim to use the upper bounds from Proposition \ref{prop:second-moment-upper}. So we assume that $n$ is sufficiently large and $\delta$ sufficiently small. We abbreviate $\alpha_v=4\beta_c(\beta_n-\beta_c)\log^2(2n/v)$ and $L_v=\log(2n/v)$. The upper bounds from Propositions~\ref{prop:first-moment-asymp} and~\ref{prop:second-moment-upper} yield for some absolute constant $C_1>0$  (only depending on $\gamma$) that 
    \[
    \begin{aligned}
    \E^{\beta_n}\Bigg[\sum_{v=1}^{\delta n-1} |\sC_{[v, n]}(v)|^2\Bigg]
    &\le C_1\sum_{v=1}^{\delta n-1}\frac{\Si(\alpha_v) + 1/L_v}{L_v^2\Si(\alpha_v)^3}\frac{n}{v} \\
    &=C_1\sum_{v=1}^{\delta n-1}\bigg(\frac{1}{L_v^2\Si(\alpha_v)^2}\frac{n}{v}+\frac{1}{L_v^3\Si(\alpha_v)^3}\frac{n}{v} \Bigg).
    \end{aligned}
    \]
    We recall that $\alpha\mapsto\Si(\alpha)$ defined in \eqref{eq:CS} is decreasing. Thus, we may assume that $n$ is sufficiently large, such that  \[
    \Si(\alpha_v) = \Si\big(4\beta_c(\beta_n-\beta_c)\log^2\tfrac{2n}{v}\big) \ge  
    \min\Big(1,\tfrac{1}{2}\liminf_{n\to\infty}\Si\big(4\beta_c(\beta_n-\beta_c)(\log 2n)^2\big)\Big).
    \]
    Lemma~\ref{lemma:si-log2} implies that $\Si(\rho_n(\log 2n)^2)/\Si(\rho_n(\log 2n)^2)\ge 1/2$ for all large $n$ for any sequence $\rho_n\to0$ such that $\limsup \rho_n(\log n)^2<\pi^2$. Therefore, the right-hand side is strictly positive. Thus, there exists a constant $C_2>0$ (depending on the sequence $(\beta_n)_{n\ge 1}$), such that for all $n$ sufficiently large, 
    \[
    \begin{aligned}
    \E^{\beta_n}\Bigg[\sum_{v=1}^{\delta n-1} |\sC_{[v, n]}(v)|^2\Bigg]
    &\le C_2n\sum_{v=1}^{\delta n-1}\bigg(\frac{1}{v\log^2(2n/v)}+\frac{1}{v\log^3(2n/v)} \Bigg).
    \end{aligned}
    \]
    Switching summation to integrals, there exist constants $C_3>0$ and $\delta=\delta(\varepsilon)$, such that for all $n$ sufficiently large,
    \begin{equation}\label{eq:expectation-second}
    \E^{\beta_n}\Bigg[\sum_{v=1}^{\delta n-1} |\sC_{[v, n]}(v)|^2\Bigg]
    \le C_3n\bigg(\frac{1}{\log(2/\delta)} + \frac{1}{\log^2(2/\delta)}\bigg)\le (\varepsilon/2) n.
    \end{equation}
    This bounds the first sum on the right-hand side in \eqref{eq:sus-before-proof}. We next show that also the second in  \eqref{eq:sus-before-proof} is negligible when $k$ is large depending on $\delta$. Its summands are decreasing, and hence, by Markov's inequality, and the domination by the $\gamma$-branching random walk,
    \[
    \begin{aligned}
    \sum_{v=\delta n}^{n}\E^{\beta_n}\Big[ |\sC_{[\delta n, n]}(v)|\ind{|\sC_{[\delta n, n]}(v)|\ge k}\Big]&\le n\E^{\beta_n}\Big[ |\sC_{[\delta n, n]}(\delta n)|\ind{|\sC_{[\delta n, n]}(\delta n)|\ge k}\Big]\\&\le \frac{n}{k}\E^{\beta_n}\Big[ |\sC_{[\delta n, n]}(\delta n)|^2\Big]\\
    &\le \frac{n}{k}\E^{\beta_n}_{\log(2\delta n-1)}\Big[ |T_{[\log(\delta n-1), \log (2n-1)]}|^2\Big]\le \frac{n}{k}\E^{\beta_n}_0\Big[ |T_{[0,\log(2/\delta)]}|^2\Big].
    \end{aligned}
    \]
    The expectation on the right-hand side is bounded from above by a constant depending on $\delta$ by Proposition \ref{prop:second-moment-upper}, which applies when $\delta$ is sufficiently small. Therefore, 
    we can pick $k=k(\delta,\varepsilon)$ sufficiently large such that the right-hand side is at most $(\varepsilon/2)n$ for all $n$ sufficiently large. Combined with \eqref{eq:sus-before-proof} and~\eqref{eq:expectation-second}, this finishes the proof.
\end{proof}
This establishes all preliminaries for the last proof of the paper.
\begin{proof}[Proof of  Theorem~\ref{thm:susceptibility}]
 We first prove the convergence in probability of \eqref{eq:thm-sus}.
 By Lemma \ref{lem:local-lim-progeny},  for every $\varepsilon>0$, there exists $k=k(\varepsilon)\in\N$, such that 
\[
\E^{\beta_c}_{0}\big[T_{(-\infty, X]}\ind{T_{(-\infty, X]}\le k}\big] \ge 4(1-\gamma^2) -\varepsilon.
\]
By the law of large numbers in \eqref{eq:local-sus}, for all $n$ sufficiently large, 
\[
\begin{aligned}
\Prob^{\beta_n}\Bigg(\frac{1}{n}\sum_{v=1}^n&|\sC_n(v)| \ge 4(1-\gamma^2)-2\varepsilon\Bigg)\\ &\ge \Prob^{\beta_n}\Bigg(\frac{1}{n}\sum_{v=1}^n|\sC_n(v)|\ind{|\sC_n(v)|\le k} \ge \E^{\beta_c}_{0}\big[T_{(-\infty, 0]}\ind{T_{(-\infty, 0]}\le k}\big] -\varepsilon\Bigg) \ge 1-\varepsilon.
\end{aligned}
\]
The lower bound of the convergence in probability in \eqref{eq:thm-sus} follows as $\varepsilon>0$ is arbitrarily small. \smallskip

We turn to a matching upper bound. By Lemma~\ref{lem:negligible}, there exists for each $\varepsilon>0$ some $k=k(\varepsilon^2)$ such that for all $n$ sufficiently large, by Markov's inequality,
\[
\Prob^{\beta_n}\Bigg(\sum_{v=1}^n |\sC_{n}(v)|\ind{|\sC_n(v)|\ge k}\ge \varepsilon n\Bigg)\le \frac{\varepsilon^2n}{\varepsilon n}=\varepsilon.
\]
As a result, also invoking the law of large numbers in \eqref{eq:local-sus}, and that $\E^{\beta_c}_{0}\big[T_{(-\infty, X]}\ind{T_{(-\infty, X]}\le k}\big]$ is strictly smaller than $4(1-\gamma^2)$, it follows that as $n\to\infty$,
\begin{align*}
    \Prob^{\beta_n}\Bigg(\frac{1}{n}\sum_{v=1}^n&|\sC_n(v)| \le 4(1-\gamma^2) +\varepsilon\Bigg)\\& \ge \Prob^{\beta_n}\Bigg(\frac{1}{n}\sum_{v=1}^{n}|\sC_n(v)|\ind{|\sC_n(v)|\le k} \le 4(1-\gamma^2)\Bigg) -\varepsilon \longrightarrow 1-\varepsilon.
\end{align*}
This proves the upper bound of the convergence in probability in \eqref{eq:thm-sus}.  Convergence in expectation follows similarly, and convergence in $L^1$ is then implied.

We finish the proof of the theorem with a proof of \eqref{eq:thm-logsus}. Let $M$ be arbitrary large. By Lemma \ref{lem:local-lim-progeny}, there exists $k\in \N$ such that 
\[
\E_{0}^{\beta_c}\big[T_{(-\infty, X]}\log T_{(-\infty, X]}\ind{T_{(-\infty, X]}\le k}\big] \ge M+1.
\]
By the law of large numbers \eqref{eq:local-logsus}, it follows that 
\[
\Prob^{\beta_n}\Bigg(\frac{1}{n}\sum_{v=1}^n|\sC_n(v)|\log|\sC_n(v)| \ge M\Bigg) \ge \Prob^{\beta_n}\Bigg(\frac{1}{n}\sum_{v=1}^n|\sC_n(v)|\log|\sC_n(v)|\ind{|\sC_n(v)|\le k} \ge M\Bigg) \overset{n\to\infty}\longrightarrow 1.
\]
Since $M$ is arbitrarily large, \eqref{eq:thm-logsus} follows.
\end{proof}

\emph{Acknowledgment:}
This material is based upon work supported by the National Science Foundation under Grant No. DMS-1928930, while the authors were in residence at the Simons Laufer Mathematical Sciences Institute in Berkeley, California, during the spring semester of 2025. JJ additionally thanks Magdalen College for a Senior Demyship.

\bibliographystyle{abbrv}
\bibliography{criticality}

\end{document}